\documentclass[reqno,centertags,11pt]{amsart}
\usepackage{amsmath,amsthm,amscd,amssymb}
\usepackage{latexsym}
\usepackage{mathabx}  

\usepackage[colorlinks=true]{hyperref}

\usepackage {verbatim}
\usepackage{enumerate, verbatim}
\newcommand{\Hmm}[1]{\leavevmode{\marginpar{\tiny%
$\hbox to 0mm{\hspace*{-0.5mm}$\leftarrow$\hss}%
\vcenter{\vrule depth 0.1mm height 0.1mm width \the\marginparwidth}%
\hbox to 0mm{\hss$\rightarrow$\hspace*{-0.5mm}}$\\\relax\raggedright #1}}}

\newcommand{\N}{{\mathbb{N}}}

\newcommand{\R}{{\mathbb{R}}}

\newcommand{\C}{{\mathbb{C}}}

\newcommand{\Z}{{\mathbb{Z}}}

\newcommand{\f}{\frac}

\newcommand{\ol}{\overline}

\newcommand{\wti}{\widetilde  }
\newcommand{\Oh}{O}
\newcommand{\oh}{o}

\newcommand{\beq}{\begin{equation}}
\newcommand{\eeq}{\end{equation}}
\newcommand{\bdm}{\begin{displaymath}}
\newcommand{\edm}{\end{displaymath}}
\newcommand{\ba}{\begin{align}}
\newcommand{\ea}{\end{align}}
\newcommand{\bpf}{\begin{proof}}
\newcommand{\epf}{\end{proof}}

\newcommand{\la}{\langle}
\newcommand{\ra}{\rangle}

\newcommand{\vphi}{\varphi}

\newcommand{\supp}{\mathrm{supp}\, }               
\newcommand{\dist}{\mathrm{dist}}               

\newcommand{\veps}{\varepsilon}

\newcommand{\re}{\mathrm{Re}}
\newcommand{\im}{\mathrm{Im}}

\newcommand{\sgn}[1]{{\mathrm{sgn}(#1)}}

\newcommand{\id}{\mathbf{1}}                

\newcommand{\dav}{{d_{\mathrm{av}}}}

\allowdisplaybreaks

\newcommand{\calM}{\mathcal{M}}

\newtheorem{theorem}{Theorem}
\newtheorem{proposition}[theorem]{Proposition}
\newtheorem{lemma}[theorem]{Lemma}
\newtheorem{corollary}[theorem]{Corollary}

\theoremstyle{definition}

\newtheorem{definition}[theorem]{Definition}

\newtheorem{remark}[theorem]{Remark}
\newtheorem{remarks}[theorem]{Remarks}

\newcounter{theoremi}[theorem]

\newcommand{\itemthm}{\refstepcounter{theoremi} {\rm(\roman{theoremi})}{~}}

\numberwithin{theorem}{section}
\numberwithin{equation}{section}

\newcounter{smalllist}

\newcounter{smallenum}
\newenvironment{SE}{\begin{list}{{\rm\arabic{smallenum})}}{%
\setlength{\topsep}{0mm}\setlength{\parsep}{0mm}\setlength{\itemsep}{0mm}%
\setlength{\labelwidth}{2em}\setlength{\leftmargin}{2em}\usecounter{smallenum}%
}}{\end{list}}
\newcounter{assumptions}
\newenvironment{assumptions}{\begin{list}{\textbf{A\rm\arabic{assumptions}})}{%
\setlength{\topsep}{0mm}\setlength{\parsep}{0mm}\setlength{\itemsep}{0mm}%
\setlength{\labelwidth}{2em}\setlength{\leftmargin}{2em}\usecounter{assumptions}%
}}{\end{list}}

\begin{document}

\title[]{Discrete diffraction managed solitons: Threshold phenomena and rapid decay for general nonlinearities}
\author[M.-R.~Choi, D.~Hundertmark, Y.-R.~Lee]{Mi-Ran Choi, Dirk Hundertmark, Young-Ran~Lee}
\address{Department of Mathematics, Sogang University, 35 Baekbeom-ro (sinsu-dong),
    Mapo-gu, Seoul, 121-742, South Korea.}%
\email{rani9030@sogang.ac.kr}
\address{Department of Mathematics, Institute for Analysis, Karlsruhe Institute of Technology, 76128 Karlsruhe, Germany.}%
\email{dirk.hundertmark@kit.edu}%
\address{Department of Mathematics, Sogang University,  35 Baekbeom-ro (sinsu-dong),
    Mapo-gu, Seoul, 121-742, South Korea.}%
\email{younglee@sogang.ac.kr}

\thanks{\copyright 2016 by the authors. Faithful reproduction of this article,
        in its entirety, by any means is permitted for non-commercial purposes}
\date{\today, version general-diffraction-8-1.tex}


\begin{abstract}
We prove a threshold phenomenon for the existence/non-existence of energy minimizing solitary solutions of the diffraction management equation for strictly positive and zero average diffraction. Our methods allow for a large class of nonlinearities, they are, for example, allowed to change sign, and the weakest possible condition, it only has to be locally integrable, on the local diffraction profile. The solutions are found as minimizers of a nonlinear and nonlocal variational problem which is translation invariant. There exists a critical threshold $\lambda_{\mathrm{cr}}$ such that minimizers  for this variational problem exist if their power is bigger than
$\lambda_{\mathrm{cr}}$ and no minimizers exist with power less than the critical threshold.
We also give simple criteria for the finiteness and strict positivity of the
critical threshold.
Our proof of existence of minimizers is rather direct and avoids the use of Lions' concentration compactness argument.

Furthermore, we give precise quantitative lower bounds on the exponential decay rate of the diffraction management solitons, which confirm the physical heuristic prediction for the asymptotic decay rate. Moreover, for ground state solutions, these bounds give a quantitative lower bound for the divergence of the exponential decay rate in the limit of vanishing average diffraction.
For zero average diffraction, we prove quantitative bounds which show that the solitons decay much faster than exponentially. Our results considerably extend and strengthen the results of \cite{HuLee 2012-existence} and \cite{HuLee 2012}.

\end{abstract}

\maketitle
{\hypersetup{linkcolor=black}
\tableofcontents}

\section{Introduction}\label{introduction}

We study the existence and properties of solutions of the  diffraction managed non-linear discrete Schr\"odinger equation
\beq\label{eq:GT}
     \omega \varphi (x)
    = -d_{\text{av}} (\Delta \varphi) (x)
    	- \int_\R T_r^{-1}\big[ P(T_r \varphi(x))\big]\mu( dr),
 \eeq
on $l^2(\Z)$, where $\mu$ is a finite measure with compact support
and $\omega$ a constant.
Here, $\Delta f(x)= f(x+1)-2f(x) +f(x-1)$ is the discrete Laplacian on $\Z$,  $T_r=e^{ir\Delta}$ is the solution operator of the free discrete Schr\"{o}dinger equation in one dimension, the average diffraction $d_{\text{av}}$ is either positive or zero, and $P$ is the nonlinear term.
Previously, either only very simple pure power nonlinearities  $P$ together with simple measures $\mu$, which correspond to piecewise constant local diffraction profiles $d_0$, or the specific  third order nonlinearity $P(z)=|z|^2z,\ z \in \C$ and general probability measures $\mu$ have been studied, see the discussion in Appendix \ref{sec:connection-nonlinear-optics}.  We will extend this to a large class of nonlinearities.

Discrete nonlinear dispersive equations such as the discrete nonlinear Schr\"odinger equation  \eqref{eq:GT} arise in the context of nonlinear optics \cite{AART,AALR, ESMBA98, MPAES, WY},
the study of dynamics of biological molecules \cite{Davydov73, ELS85}, localized modes in anharmonic crystal in condensed matter physics, \cite{BS75, TKS88}.
Here the discrete models arise as phenomenological models or as tight binding approximations, see, for example, \cite{MKSP, PS}.

In the application to nonlinear optics, which is our main motivation for studying solutions of \eqref{eq:GT},
$\dav$ is the average diffraction along an array of waveguides and $\mu$ will be a probability measure with compact support related to the local periodic diffraction $d_0$ along the waveguide. Since we can treat \emph{arbitrary} probability measures $\mu$ with compact support, our results hold for \emph{any local diffraction profile} $d_0$ which is locally integrable. In particular, $\mu=\delta_0$, the Dirac mass at zero, is allowed, so our results include the well-known discrete NLS. We will discuss this more thoroughly in  Appendix \ref{sec:connection-nonlinear-optics}.

To get the weak formulation of \eqref{eq:GT}, let $\la f,g \ra\coloneq \sum_{x\in\Z} \ol{f(x)}g(x)$ be the usual scalar product in $l^2(\Z)$  and take the scalar product of \eqref{eq:GT} with $h\in l^2(\Z)$ to see that since   $-\la \Delta\varphi, h\ra= \la D_+\varphi, D_+h \ra$, where the forward difference operator $D_+$ is defined by
$$
(D_+ f)(x):=f(x+1)-f(x)
$$
for any $x\in \Z$ and using the unicity of $T_r$ one has
 \begin{align*}
 	\la \int_\R T_r^{-1}\big[ P(T_r \varphi)\big]\mu( dr), h\ra
 	=
 		\int_\R \la P(T_r \varphi) , T_r h\ra\, \mu( dr)
 \end{align*}
and therefore the weak formulation of \eqref{eq:GT} is given by
\begin{align}\label{eq:GT-weak}
	\omega \la \varphi, h\ra = \dav\la D_+\varphi, D_+ h\ra
		- \int_\R \la P(T_r \varphi), T_rh \ra\mu( dr)
\end{align}
for all $h\in l^2(\Z)$.

The diffraction management equation \eqref{eq:GT}, or better, its weak form \eqref{eq:GT-weak}, has a variational structure. We assume that $P$ is an odd nonlinearity of the form
 	\begin{align}\label{eq:P}
 		P(z)= p(|z|)z
 	\end{align}
for $z\in\C$. To use this, let $V$ be a differentiable function with $V'(a)=P(a)$ for $a\ge 0$, for example,
\begin{align}\label{eq:V}
	V(a)= \int_0^a P(s)\, ds \quad \text{ for } a\ge 0.
\end{align}
Then the constrained minimization problem associated with \eqref{eq:GT} is given by
 \beq\label{eq:min}
    E^{\dav}_\lambda := \inf \{H(\vphi) : \vphi \in l^2(\Z),\ \|\vphi\|_{l^2(\Z)}^2=\lambda\},
 \eeq
where $\lambda>0$ and the Hamiltonian, or the energy, takes the form
 \beq\label{eq:energy}
    H(\vphi):=\f{d_{\text{av}}}{2}\| D_+ \vphi\|^2_{\l^2(\Z)}-N(\vphi) ,
 \eeq
with the \emph{nonlocal} nonlinear `potential'
\begin{align} \label{eq:N(f)}
   N(\vphi):=\int_{\R}\sum_{x\in \Z} V(|T_r \vphi(x)|)\mu (dr).
 \end{align}
It turns out that any minimizer of \eqref{eq:min}, that is,  any $\varphi\in l^2(\Z)$ with $\|\varphi\|_{l^2(\Z)}^2=\lambda$ such that $E^{\dav}_\lambda = H(\varphi)$, will be a solution of corresponding Euler--Lagrange equation \eqref{eq:GT}. Thus we are led to study the minimization problem \eqref{eq:min} and to investigate the properties of its solution. An obstacle for the existence proof is the invariance of the Hamiltonian under shifts so the variational problem is invariant under a non-compact group. Hence there is a potential \emph{loss of compactness}, since minimizing sequences can easily converge weakly to zero.

While it is possible to formulate conditions directly on the nonlinearity $P$ in
\eqref{eq:GT},  we find it more convenient to use conditions on the nonlinear potential $V$ related to it by  \eqref{eq:V}.
Our main assumptions on the nonlinear potential $V:\R_+\to\R $ are
\begin{assumptions}
\item \label{ass:A1}  $V$ is continuous on $\R_+=[0,\infty)$ and differentiable on $(0,\infty)$ with $V(0)=0$. There exist $2 < \gamma_1\le \gamma_2 <\infty$ such that
	\begin{align} \label{eq:q'bound}
		|V'(a)|\lesssim a^{\gamma_1-1} + a^{\gamma_2-1} \quad\text{ for all } a>0.
	\end{align}
\item \label{ass:A2} $V$ is continuous on $\R_+$ and differentiable on $(0,\infty)$ with $V(0)=0$. There exists $\gamma_0>2$ such that
	\begin{align*}
		V'(a)a \ge \gamma_0 V(a)\quad\text{ for all } a>0.
	\end{align*}
\item \label{ass:A3} There exists $a_0>0$ such that $V(a_0)>0$.
\end{assumptions}
\vspace{2pt}

The three assumptions above are our main requirements on the nonlinear potential. They are enough to prove a threshold phenomenon: solutions exist at least for large enough power $\lambda=\|\varphi\|_2^2$.
In order to guarantee the existence of solutions for arbitrarily small
$\lambda$, we need to strengthen assumption \ref{ass:A3} to
\vspace{2pt}

\begin{assumptions}
\setcounter{assumptions}{3}
\item \label{ass:A4} If $\dav>0$ we assume that there exist $\veps>0 $
	 and $2\le \kappa < 6$ such that
	\begin{align*}
		V(a)\gtrsim a^{\kappa} \quad\text{ for all } 0< a\le\veps.
	\end{align*}
 If $\dav = 0$ we assume that $V(a)>0$ for all $0<a\le \veps$.
%
\end{assumptions}
\begin{remarks}
 	\itemthm An integration shows that \ref{ass:A1} implies
		\begin{align}\label{eq:Vbound}
		|V(a)|\lesssim a^{\gamma_1} + a^{\gamma_2}.
	\end{align}
	Much more important for us is the fact that \ref{ass:A1} allows us to control the \emph{nonlocal} nonlinearity $N$ under \emph{splitting}, see Lemma \ref{lem:V-splitting} and the discussion in  section \ref{subsec:splitting}. \\
  \itemthm Examples of nonlinearities obeying assumptions \ref{ass:A1} through \ref{ass:A3} are given by
  \begin{align*}
  	V(a) = \sum_{j=1}^J c_j a^{s_j}
  \end{align*}  	
  with $c_j\ge 0$, $2<s_j<\infty$, and $J\in\N$, but our assumptions also allow nonlinear potentials which can become negative, for example,
  \begin{align*}
  	V(a) = -a^4 +a^6\quad \text{for } a\ge 0
  \end{align*}
  is allowed. It certainly fulfillls \ref{ass:A1}. Since
  \begin{align*}
  	V'(a)a = -4a^4 +6 a^6 = 4(-a^4+a^6) +2a^6 \ge 4V(a),
  \end{align*}
  it also obeys \ref{ass:A2}. Moreover,  $V(a_0)>0$ for all large enough $a_0$, so \ref{ass:A3} holds.

  If we did not assume \ref{ass:A3}, then the nonlinearities could also be strictly negative for all $a>0$, for example, $V(a)= -a^4-a^6$ obeys \ref{ass:A1} and because of
  \begin{align*}
  	V'(a)a = -4a^4 -6 a^6 = 6(-\frac{4}{6}a^4 - a^6) \ge 6 V(a)
  \end{align*}
  also \ref{ass:A2}, but then the critical threshold $\lambda_{\mathrm{cr}}$ given in Theorem \ref{thm:threshold-phenomena} would be infinite. The threshold is finite if and only if, for some $f\in l^2(\Z) $ we have $N(f)>0$, see part \ref{thm:threshold-phenomena-4} of Theorem \ref{thm:threshold-phenomena} below.
\end{remarks}

Concerning the existence and nonexistence of solutions, we have
 \begin{theorem}[Threshold phenomenon for existence/non-existence] \label{thm:threshold-intro}
 	Assume that $V$ obeys assumptions \ref{ass:A1} through \ref{ass:A3}
 	and that $\dav\ge 0$. 	\\
    \itemthm 	
		There exists a threshold $0\le\lambda_\mathrm{cr}<\infty$ such that $E^{\dav}_\lambda =0$ for $0\le \lambda \le \lambda_{\mathrm{cr}}$ and $-\infty<E^{\dav}_\lambda<0$ for  $\lambda> \lambda_{\mathrm{cr}}$. \\
	\itemthm
		If $\dav>0$ and $0<\lambda<\lambda_{\mathrm{cr}}$, then no minimizer for the constrained  minimization problem  \eqref{eq:min} exists.
		If $\gamma_1\ge 6$, then $\lambda_{\mathrm{cr}}>0$. \\
	\itemthm
		If $\dav\ge 0$ and $\lambda>\lambda_{\mathrm{cr}}$, then any minimizing sequence for  \eqref{eq:min} is up to translations relatively compact in $l^2(\Z)$, in particular,  there exists a minimizer for \eqref{eq:min}.
		This minimizer is also a solution of the diffraction management equation \eqref{eq:GT} for some Lagrange multiplier $\omega< 2E^{\dav}_\lambda/\lambda<0$. \\
	\itemthm
		If $V$ obeys, in addition, \ref{ass:A4},  then $\lambda_{\mathrm{cr}}=0$.
\end{theorem}

\begin{remarks}
  \itemthm The proof of Theorem \ref{thm:threshold-intro} is given at the end of Section \ref{sec:threshold}. The precise definition of the threshold $\lambda_{\text{cr}}$ is given in Definition \ref{def:threshold}. As we will see in Theorem \ref{thm:existence}, minimizing sequences for \eqref{eq:min} are relatively compact in $l^2(\Z)$ modulo translations if and only if $E^{\dav}_\lambda<0$.
  So when $\lambda=\lambda_{\mathrm{cr}}$ minimizers might exist, but minimizing sequences do not have be be precompact modulo translations. \\
  \itemthm Using $h=\varphi$ in \eqref{eq:GT-weak}, it is clear that the Lagrange multipliers are
  \begin{align*}
  	\omega = \omega(\varphi)= \frac{\dav\la D_+\varphi, D_+\varphi\ra
		- \int_\R \la P(T_r \varphi),T_r \varphi\ra\, \mu( dr) }{\la \varphi,\varphi\ra}
  \end{align*}	
 and using assumption \ref{ass:A2} this will yield a rather direct proof of
 $\omega(\varphi)< 2 E^{\dav}_\lambda/\lambda < 0$ for all minimizers $\varphi$, see \eqref{eq:negative-Lagrange-multiplier}.
\end{remarks}

If $\varphi$ is a solution of \eqref{eq:GT}, or rather of its weak version \eqref{eq:GT-weak}, one can ask how well it will be localized. As it turns out, the answer to this depends on whether $\dav=0$ or $\dav>0$. In an earlier paper \cite{HuLee 2012}, super-exponential decay of solutions for $\dav=0$ was shown in the case that the nonlinearity is cubic, $P(a)=|a|^2a$ or $V(a)=\frac{1}{4} |a|^4$. The case of positive average diffraction was not studied.

There is a simple physical heuristic guess for decay rate of solutions of \eqref{eq:GT-weak}: Assume that $\varphi$ decays exponentially and make the ansatz $\varphi(x)= e^{-\nu x}$ for $x\gg 1$. Plugging this into \eqref{eq:GT} and hoping that
, even despite possible nonlocal effects, the nonlinearity in \eqref{eq:GT} is of higher order than $e^{-\nu x}$, then
\begin{align*}
	\omega e^{-\nu x}
	& = -\dav \Delta(e^{-\nu \cdot})(x) +\oh(e^{-\nu x}) \\
	& = -\dav (e^{-\nu(x+1)} -2e^{-\nu x} + e^{-\nu (x-1)}) + \oh(e^{-\nu x}) \\
	& = -2\dav (\cosh(\nu)-1)e^{-\nu x} + \oh(e^{-\nu x}).
\end{align*}
Letting $x\to\infty$, one sees that this implies $\omega<0$ and $2\dav(\cosh(\nu)-1)= |\omega|$, or, with $\cosh^{-1}$ the inverse function of $\cosh:[0,\infty)\to [1,\infty)$,
\begin{align}
	\nu = \cosh^{-1}\left( \frac{|\omega|}{2\dav} +1 \right)
\end{align}
which is a rather precise prediction for the exponential decay rate. A remarkable feature of it is that it predicts $\nu\to\infty$ if $\dav\to 0$ as long as $\omega$ stays away from zero.

Of course, this all depends on in which sense the nonlocal nonlinear terms in \eqref{eq:GT} are really of lower exponential order. Nevertheless, this simple physical heuristic is not far from the truth, because of

\begin{theorem}[Decay for positive average diffraction]\label{thm:exponential decay}
Assume $\dav>0$ and $V$ obeys assumption \ref{ass:A1}. Then any solution $\varphi$ of \eqref{eq:GT} with $\omega<0$ decays exponentially and the decay rate is given by the above heuristic in the sense that
	\begin{align}\label{eq:exponential decay rate}
		\nu_*(\varphi):= \sup\left\{ \nu>0|\, (x\mapsto e^{\nu|x|}\varphi(x))\in l^2(\Z) \right\}
		\ge \cosh^{-1}\left(\frac{|\omega|}{2\dav}+1\right).
	\end{align}
\end{theorem}
\begin{remark}
 As we will see in Theorem \ref{thm:threshold-phenomena} below, the ground state solutions, that is, the ones with minimal energy, are solutions with   $\omega< 2E^{\dav}_\lambda/\lambda<0$
 for all $\dav >0$.  	
 At the moment, we cannot rule out that there are solutions of \eqref{eq:GT} for which $\nu_*> \cosh^{-1}\left(\frac{|\omega|}{2\dav}+1\right)$.
\end{remark}

Given the lower bound on the exponential decay rate given in \eqref{eq:exponential decay rate}, one expects that $\nu_*(\varphi_\dav)\to\infty$ as $\dav \to 0$, as long as the corresponding Lagrange multipliers  $\omega= \omega(\varphi_\dav)$ stay away from zero. In general, this might not be the case,  but it is true for ground state solutions.
\begin{corollary}\label{cor:divergence}
	Let $\lambda>0$, $\dav > 0$, and $\calM^\dav_\lambda $ the set of
	minimizers of the constrained minimization problem \eqref{eq:min}. Then for fixed $\lambda>0$ and any choice $\varphi_\dav\in \calM^\dav_\lambda$ the exponential decay rates diverge in the limit of small average dispersion. More precisely, we have the lower bound
	\begin{align*}
		\liminf_{\dav \to 0}\frac{\nu_*(\varphi_\dav)}{\cosh^{-1}\left(\frac{|E^0_\lambda|-\delta}{\lambda\dav}+1\right)} \ge 1
	\end{align*}
	for any $0<\delta< |E^0_\lambda|$, so the exponential decay rate $\nu_*(\varphi_\dav)$ diverges at least logarithmically as $\dav\to 0$.
\end{corollary}
\begin{proof}
  This is, in fact, a simple consequence of the lower bound \eqref{eq:exponential decay rate},
  the negativity of  $E^{\dav}_\lambda$, guaranteed by Theorem \ref{thm:existence},
  its monotonicity\footnote{which follows immediately from  definition \eqref{eq:min}.} in $\dav\ge 0$, and the bound on the Lagrange multipliers from Theorem \ref{thm:existence}, which imply
  that for all $\delta>0$ one has $|\omega(\varphi_\dav)|\lambda \ge 2|E^{\dav}_\lambda|
  \ge 2(|E^0_\lambda|-\delta)$ for all small enough $\dav>0$.
\end{proof}

Given that the exponential decay rate of the ground states for average diffraction $\dav>0$ diverges as $\dav\to 0$, one can ask how fast solutions of \eqref{eq:GT-weak} decay
when   $\dav=0$. This was done in \cite{HuLee 2012} for the special fourth order
nonlinearity $V(a)\sim a^4$, but it holds in much greater generality.
\begin{theorem}[Super-exponential decay for zero average diffraction]\label{thm:super-exp decay}
Assume $\dav=0$ and $V$ obeys assumption \ref{ass:A1}. Then any solution $\varphi$ of \eqref{eq:GT} with $\omega \neq 0$ decays super-exponentially, more precisely,
 \begin{align}
    \nu_{**}(\varphi):= \sup\left\{ \nu>0|\,  (x\mapsto (|x|+1)^{\nu(|x|+1)}\varphi(x))\in l^2(\Z) \right\} \ge \frac{2\gamma_1-3}{2(\gamma_1-1)} .
  \end{align}
\end{theorem}
\begin{remark}
	For $\gamma_1=4$, this yields the lower bound $\nu_{**}(\varphi)\ge 5/6$ which
	is much better than the lower bound $\nu_{**}(\varphi)\ge 1/4$
	proven in  \cite{HuLee 2012}.
\end{remark}

Our paper is organized as follows: In Section \ref{sec:nonlinear Estimates} we develop the main tools needed for the existence proof. This includes new fractional linear bounds on the building blocks from Definition \ref{def:M2}, which are needed to control the nonlocal nonlinearity $N(f)$ under splitting.
That minimizing sequences for \eqref{eq:min} are precompact modulo translations, that is, there exist suitable translations such that the translated minimizing sequence has a strongly convergent subsequence, if and only if $E^{\dav}_\lambda<0$ is the content of Theorem \ref{thm:existence}. Our proof in Section \ref{sec:The Existence Proof} is based on non-splitting bounds for minimizing sequences given in  Propositions \ref{prop:fat-tail proposition} and \ref{prop: tightness modulo translations}, which together with a simple characterization of strong convergence in $l^2(\Z)$ given in Lemma \ref{lem:strong-convergence} imply precompactness of minimizing sequences modulo translations once $E^{\dav}_\lambda<0$.
 This is similar, at least in spirit, to our companion paper \cite{ChoiHuLee2015} for the continuous case.

The threshold phenomenon is then studied in Section \ref{sec:threshold} and the proof of Theorem \ref{thm:threshold-intro} is given at the end of this section. It turns out that Assumptions \ref{ass:A1} and \ref{ass:A2} are enough to yield a threshold phenomenon, see Theorem \ref{thm:threshold-phenomena}, but  it could happen that $\lambda_{\text{cr}}$ is infinite, in which case no minimizers of \ref{eq:min} exist for any $\lambda>0$. Assumption \ref{ass:A3} is used only to guarantee the finiteness of the threshold and \ref{ass:A4} guarantees that the threshold is zero.

Unlike the continuous case we are able to prove strong lower bounds on the exponential decay rate for positive average diffraction, which confirm the physical heuristic, and strong lower bounds on the super-exponential decay rate for vanishing average diffraction, which improve earlier bounds given in \cite{HuLee 2012}. These bounds are established in a
two-step process: First we prove some (super-) exponential decay, see Proposition \ref{prop:some exponential decay} in Section \ref{subsec:some exponential decay}, respectively Proposition \ref{prop:some super-exp decay} in Section \ref{subsec:some super-exp decay}, and then give arguments which allow us to boost the decay rate, see Proposition \ref{prop:exponential boost} in Section \ref{subsec:boost}, respectively Proposition \ref{prop:boost super-exp decay} in  Section \ref{subsec:boost super-exp decay}. These results are based on several intermediate results, in particular, we need suitable a-priori bounds on exponentially twisted versions of the building blocks from Definition \ref{def:L} for the derivative of the nonlinearity $N$.

In Appendix \ref{sec:useful bounds}, we gather some useful bounds for the space time norms of solutions of the free discrete Schr\"odinger equation on $l^2(\Z)$. These estimates have analogous results on $l^2(\Z^d)$, similar to the discussion in  \cite{HuLee 2012}, for example, but we give them only for $l^2(\Z)$ for brevity. Lemma \ref{lem:useful} looks somewhat technical, at first, but is at the heart of most of  our results in this work.

In Appendix \ref{sec:Strict concavity of the ground state energy}, we give the somewhat technical proof of negativity and subadditvity of the ground state energy $E^{\dav}_\lambda$ from \eqref{eq:min}. The proof of subadditivity is similar to the continuous case and given for the convenience of the reader, it also immediately yields strict subadditivity once $E^{\dav}_\lambda<0$.  That Assumption \ref{ass:A4} implies $E^{\dav}_\lambda<0$ for any $\lambda>0$ and all $\dav\ge 0$ turns out to be \emph{very much different} from the continuous case where Gaussians form a convenient set of initial conditions, since on $l^2(\Z)$ there is no simple family of initial conditions for which one can explicitly compute the time evolution under the free discrete Schr\"odinger evolution.

Appendix \ref{sec:IMS} discusses a discrete version of the well-known\footnote{See, for example, Section 3.1 of \cite{CFKS} and references therein.} IMS localization formula, which is needed for strictly positive average diffraction. Finally, in Appendix \ref{sec:connection-nonlinear-optics}, we give for the convenience of the reader a short discussion on how the highly nonlocal diffraction management equation \eqref{eq:GT} arises in the study of solitary solutions of diffraction managed waveguides arrays.

\section{Nonlinear estimates}\label{sec:nonlinear Estimates}
\subsection{Fractional linear estimates}\label{subsec:fractional linear estimates}
First, we gather some bounds which will be used in the proofs of Proposition \ref{prop:fat-tail proposition} and Proposition \ref{prop:some exponential decay}, which are the basis for the proofs of Theorems \ref{thm:existence} and \ref{thm:exponential decay} .
 We use $\| \cdot \|_p$ for $\| \cdot \|_{l^p(\Z)}$. For two functions $g$ and $h$, we write
 $g\lesssim h$ if there exists a constant $C>0$ such that $g\leq Ch$.

The space $L^p(\Z\times \R, dx\mu(dr) )$ consists of all space-time functions with finite norm
\begin{equation*}
\|f\|_{L^p(\Z\times \R, dx\mu(dr) )}:=
\begin{cases}
\left(\int_{\R}\sum _{x\in\Z}|f(x,r)|^p\mu(dr)\right)^{1/p}& \text{ if } 1 \le p < \infty , \\
\mathrm{esssup}_{r\in\supp\mu}\|f(\cdot, r)\|_\infty & \text{ if } p = \infty
\end{cases}
\end{equation*}
where the essential supremum is with respect to the measure $\mu$, that is, modulo sets of $\mu$-measure zero.

A simple but useful bound is given in
\begin{lemma}\label{lem:T_rf estimate}
Let $1\le q\le p\le \infty$ and $f\in l^q(\Z)$. Then
\bdm
\|T_rf\|_{L^p(\Z\times \R,dx \mu(dr))}\lesssim\|f\|_q
\edm
where the implicit constant depends only on $q,p$, $\mu(\R)$ and, if $q\neq 2$, also on  $\supp\mu$.
\end{lemma}
\begin{proof}
Since $\|g\|_{p}\leq\|g\|_{q}$ for all $1\leq q\leq p\le   \infty$,
we get for $p<\infty$
$$
\int_{\R}\sum_{x\in \Z}|T_r f|^{p}\mu(dr)=\int_{\R}\|T_rf\|_{p}^{p} \mu(dr)
\leq\int_{\R}\|T_rf\|_q^{p} \mu(dr)\le \mu(\R) e^{4Bp|1-2/q|} \|f\|_q^{p}.
$$
where we used \eqref{eq:lpbound} and chose $B>0$ such that
$\supp\mu\subset[-B,B]$.
If $p=\infty$,
\begin{align*}
	\|T_rf\|_{L^\infty(\Z\times \R,dx \mu(dr))}
	\le \sup_{r\in [-B,B]} \|T_rf\|_\infty
	\le \sup_{r\in [-B,B]} \|T_rf\|_q
	\le e^{4B|1-2/q|} \|f\|_q
\end{align*}
\end{proof}

\begin{lemma}[Bilinear estimate]\label{lem:bilinear-estimate}
Let $1\le p\le \infty$, $f_1, f_2\in l^2(\Z)$, set $ s= \dist (\supp f_1, \supp  f_2)$, the distance of their supports, and $B>0$ such that $\supp\mu\subset[-B,B]$. Then
\begin{equation}\label{eq:bilinear-estimate-1}
\|T_r f_1T_r f_2\|_{L^p(\Z\times \R,dx \mu(dr))}\lesssim \min\left(1,\frac{8e^{16B}(4B)^{\lceil\frac{s}{2}\rceil}}{\lceil\frac{s}{2}\rceil!} \right) \|f_1\|_2\|f_2\|_2,
\end{equation}
where we used  $\lceil s\rceil:= \min\{n\in\Z|\, s\le n\}$ and the implicit constant depends only on $\mu(\R)$ and $p$.
\end{lemma}

\bpf

The proof of \eqref{eq:bilinear-estimate-1}  is based on the strong bilinear estimate from Lemma \ref{lem:useful} in the appendix, which strengthens and simplifies the strong bilinear bound from \cite{HuLee 2012-existence, HuLee 2012}.
Choosing $B$ large enough so that $\supp\mu\subset[-B,B]$ and using  \eqref{eq:strong bilinear}  we get
\begin{align*}
\|T_r f_1T_r f_2\|_{L^p(\Z\times \R, dx\mu(dr))}
&= \left( \int_{\R}\|T_rf_1 T_rf_2\|_p^p\mu(dr)\right)^{1/p}
\leq \mu(\R)^{1/p}\sup_{r\in [-B,B]}\|T_r f_1T_r f_2\|_p\\
&\lesssim \min\left(1,\frac{8e^{16B}(4B)^{\lceil\frac{s}{2}\rceil}}{\lceil\frac{s}{2}\rceil!} \right) \|f_1\|_2\|f_2\|_2
\end{align*}
which proves \eqref{eq:bilinear-estimate-1}.
\epf

\begin{remark}
 Since $n! \ge e^{n\ln n -n}$ and $(4B)^{\lceil \frac{s}{2}\rceil}\lesssim e^{\frac{s}{2}\ln(4B)}$, Lemma \ref{lem:bilinear-estimate} implies the bounds
  \begin{equation}\label{eq:bilinear-estimate-2}
  	\|T_r f_1 T_r f_2\|_{L^p(\Z\times \R, dx\mu(dr))}
  		\lesssim \min(1,s^{-\alpha s})\,\|f_1\|_2\|f_2\|_2
  \end{equation}
for all $0<\alpha <\frac{1}{2}$ and all $1\le p\le \infty$. Here, if $s=0$, we set
$0^{-\alpha 0} \coloneq \lim_{s\to 0+} s^{-\alpha s} =1 $.
\end{remark}

The following will be the building blocks for our bounds on the nonlocal nonlinearity, their definition is motivated by the splitting of the nonlinear potential in Lemma
\ref{lem:V-splitting}.
\begin{definition} \label{def:M2} For any $ 2\le \gamma <\infty$, let
	\begin{align*}
		M_\mu^\gamma(f_1,f_2):=
		\int_{\R} \sum_{x\in \Z}|T_rf_1(x)||T_rf_2(x)|(|T_rf_1(x)|+|T_rf_2(x)|)^{\gamma-2} \, \mu(dr).
	\end{align*}
\end{definition}

\begin{proposition}\label{prop:M-bounded}
	Let $s= \dist (\supp f_1, \supp  f_2)$, $2\le  \gamma <\infty$, and $0<\alpha <\frac{1}{2}$, then
 	\begin{align}\label{eq:M-time-1}
	M_\mu^\gamma(f_1,f_2)
	\lesssim \min(1,s^{-\alpha s})\|f_1\|_2\|f_2\|_2(\|f_1\|_2+\|f_2\|_2)^{\gamma-2}
	\end{align}
where the implicit constant depends only on $\supp\mu$, $\mu(\R)$, $\gamma$, and $\alpha$.
\end{proposition}
\begin{proof}
Taking a supremum out of the integral we get
  \begin{align*}
 	M_\mu^\gamma(f_1,f_2)
 	&=\int_{\R}\sum_{x\in \Z} |T_rf_1(x)T_rf_2(x)|(|T_rf_1(x)|+|T_rf_2(x)|)^{\gamma-2} \, \mu(dr) \\
 	&\le \|T_rf_1 T_rf_2\|_{L^{1}(\Z\times\R, dx\mu(dr))}
		 (\sup_{r\in\R }\|T_rf_1\|_\infty+\sup_{r\in\R}\|T_rf_2\|_\infty)^{\gamma-2}.
 \end{align*}
 Applying $\|T_rf\|_\infty \le \|T_rf\|_2 =\|f\|_2$ and  \eqref{eq:bilinear-estimate-2} for the first factor yields \eqref{eq:M-time-1}.
\end{proof}

\subsection{Splitting the nonlocal nonlinearity}\label{subsec:splitting}
Recall
\begin{align*}
  N(f):= \int_{\R} \sum _{x\in \Z} V(|T_rf(x)|)\, \mu(dr)	.
\end{align*}
The inequality \eqref{eq:Vbound} and  Lemma \ref{lem:T_rf estimate} immediatley yield
\begin{proposition}[Boundedness]\label{prop:N-boundedness}
	Let $2\le \gamma_1\le \gamma_2< \infty$. Then for all $f\in l^2(\Z)$
	\begin{align*}
		N(f)\lesssim \|f\|_2^{\gamma_1} + \|f\|_2^{\gamma_2} ,
	\end{align*}
	where the implicit constant depends only on $\mu(\R)$. 	
\end{proposition}

Since $N(f)$ is \emph{highly nonlocal} in $f$, it is difficult to control $N(f)$, when $f$ splits into $f=f_1+f_2$ where $f_1$ and $f_2$ have widely separated supports.
The following simple observation helps at this stage and is at the heart of all our estimates.

\begin{lemma}\label{lem:V-splitting}
	Assume that $V$ obyes \ref{ass:A1}. Then
	\begin{align}\label{eq:V-splitting}
		\left| V(|z+w|) - V(|z|) - V(|w|)) \right|
		\lesssim \left( (|z|+|w|)^{\gamma_1-2} +(|z|+|w|)^{\gamma_2-2} \right)
		|z||w|
	\end{align}
	for all $z,w\in\C$.
\end{lemma}
\begin{proof} Since $V(0)=0$, we have $V(|z+w|)-V(|z|)-V(|w|)=0$ if at least one of $z$
and $w$ equals zero. So assume $z,w\neq 0$ in the following. Then
	\begin{equation}
	\begin{split}\label{eq:diff1}
		V(|z+w|) -V(|z|) - V(|w|)
		&=
			 \left[ \frac{1}{|z|+|w|}V(|z+w|) - \frac{1}{|z|}V(|z|) \right] |z| \\
		&\phantom{=} \, + \left[ \frac{1}{|z|+|w|}V(|z+w|) - \frac{1}{|w|}V(|w|) \right] |w| .
	\end{split}
	\end{equation}
	Moreover,
	\begin{align} \label{eq:diff2}
		\frac{1}{|z|+|w|}V(|z+w|) - \frac{1}{|z|}V(|z|)
			&=
				\frac{1}{|z|+|w|} \left( V(|z+w|) - V(|z|) \right) - \frac{|w|}{(|z|+|w|)|z|} V(|z|)
	\end{align}
	Let $c=\min(|z|,|z+w|)$ and $d=\max(|z|,|z+w|)\le |z|+|w|$. Then $d-c=||z+w|-|z||\le |w|$ and using \ref{ass:A1}, we have
	\begin{align*}
		\left| V(|z+w|) -V(|z|) \right|
		&\le
			\int\limits_{c}^{d} |V'(a)|\, da
			\lesssim  (d^{\gamma_1-1}+d^{\gamma_2-1})(d-c) \\
		&\le ((|z|+|w|)^{\gamma_1-1}+(|z|+|w|)^{\gamma_2-1})|w|
	\end{align*}
	Since $V(0)=0$, \ref{ass:A1} also implies
	\begin{align*}
		|V(|z|)|\lesssim (|z|^{\gamma_1-1}+ |z|^{\gamma_2-1})|z|.
	\end{align*}
	Using the two inequalities above in \eqref{eq:diff2} shows
	\begin{align*}
		\left|\frac{1}{|z|+|w|}V(|z+w|) - \frac{1}{|z|}V(|z|)\right|
		\lesssim ((|z|+|w|)^{\gamma_1-2}+(|z|+|w|)^{\gamma_2-2})|w|
	\end{align*}
	and a similar inequality holds when we interchange $z$ and $w$. Hence  \eqref{eq:diff1} implies \eqref{eq:V-splitting}.
\end{proof}

The following is our main tool to control the nonlocal nonlinearity.
\begin{proposition}[Splitting]\label{prop:N-splitting}
Let $f_1,f_2 \in l^2(\Z)$ and $s=\dist(\supp f_1,\supp f_2 )$. Then
for all $2\leq \gamma_1 \leq \gamma_2 < \infty $
  \begin{align} \label{eq:N-splitting1}
     \left|N(f_1+f_2)- N(f_1) - N(f_2)\right|
     \lesssim
       \min(1, s^{-\alpha  s})\|f_1\|_2\|f_2\|_2 \left(1+\|f_1\|_2^{\gamma_2-2}+\|f_2\|_2^{\gamma_2-2}\right)
  \end{align}
  for all $0<\alpha <\frac{1}{2}$.
\end{proposition}

\begin{proof}
 Because of Lemma \ref{lem:V-splitting} we have
 \begin{align*}
   	\Big|N(f_1+f_2) &- N(f_1) - N(f_2)\Big| \\
   	&\le  \int_{\R}\sum_{x\in \Z} \left| V(|T_rf_1(x)+T_rf_2(x)|) -V(|T_rf_1(x)|) -V(|T_rf_2(x)|) \right|\,  \mu(dr) \\
   	&\lesssim
   	  M_\mu^{\gamma_1}(f_1,f_2) + M_\mu^{\gamma_2}(f_1,f_2).
  \end{align*}
 So \eqref{eq:N-splitting1} follows from
 \eqref{eq:M-time-1}, noting also that
 \begin{align*}
 	(a+b)^{\gamma_1-2} + (a+b)^{\gamma_2-2} \lesssim  1+ a^{\gamma_2-2} +b^{\gamma_2-2} ,
 \end{align*}
 for all $a,b\ge 0$.
\end{proof}

\section{The existence proof}\label{sec:The Existence Proof}
In this section we will give the proof of
\begin{theorem}\label{thm:existence}
	Let $\lambda>0$ and assume that $V$ obeys \ref{ass:A1} and \ref{ass:A2}. Then every minimizing sequence for the constrained variational problem \eqref{eq:min} is precompact modulo translations if and only if
	$E^{\dav}_\lambda<0$.
	In particular, if $E^{\dav}_\lambda<0$, then minimizers of \eqref{eq:min} exist and these miniminzers are solutions of the diffraction management equation \eqref{eq:GT} for some Lagrange multiplier
	$\omega< 2E^{\dav}_\lambda/\lambda<0$.
\end{theorem}

Key for our proof of Theroem \ref{thm:existence} is the following proposition,
which will help to eliminate a possible splitting of minimizing sequences
when $E^{\dav}_\lambda$ is strictly negative.
In the following  we will assume that the nonlinear potential fulfills assumptions \ref{ass:A1} and \ref{ass:A2}.
For $s\in\R$, we let $s_+:=\text{max}(s,0)$.

\begin{proposition} \label{prop:fat-tail proposition}
Assume that $V$ obeys \ref{ass:A1} and \ref{ass:A2}. Then there exists a universal
constant $C>0$ such that for any  $\lambda>0$,
 $f\in l^2(\mathbb{Z})$ with $\|f\|_2^2=\lambda $ and $ 0<\delta<\f{\lambda}{2}$, and $a,b\in\mathbb{Z}$ with
\begin{equation}\label{ineq:condition-a-b}
\sum_{x\le a}|f(x)|^2\geq\delta\ \ \text{and}\ \ \sum_{x\ge b}|f(x)|^2\geq\delta
\end{equation}
we have
\begin{equation}\label{estimate:fat-tail-positive}
H(f)\geq \left[1-(2^{\f{\gamma_0}{2}}-2)\left(\f{\delta}{\lambda}\right)^{\f{\gamma_0}{2}}\right]
    E^{\dav}_\lambda  - C(\lambda+\lambda^{\gamma_2/2})  \left((b-a+1)_+^{1/2} -1\right)_+^{-1/2}.
\end{equation}
\end{proposition}
\begin{remark}
Note that $E^{\dav}_\lambda\le 0$ for all $\lambda>0$ (see Proposition \ref{prop:strict-subadditivity} in the appendix). As soon as $E^{\dav}_\lambda<0$ we have
\begin{align*}
	\left[1-(2^{\f{\gamma_0}{2}}-2)\left(\f{\delta}{\lambda}\right)^{\f{\gamma_0}{2}}\right]
    E^{\dav}_\lambda > E^{\dav}_\lambda .
\end{align*}	

Therefore if $E^{\dav}_\lambda<0$, taking a mininimizing sequence $f_n$ with
$\|f_n\|_2^2=\lambda>0$ and $H (f_n)\to E^{\dav}_\lambda$, and taking any $a_n$ and $b_n$ according to \eqref{ineq:condition-a-b}, the bound \eqref{estimate:fat-tail-positive} shows
\begin{align*}
	\limsup_{n\to\infty} (b_n-a_n)<\infty
\end{align*}
since $\lim_{s\to\infty} \left((s+1)_+^{1/2} -1\right)_+^{-1/2}=0$. Thus Proposition \ref{prop:fat-tail proposition} implies that the regions where a minimizing sequence $f_n$ has $\delta$-fat tails do not separate too much \emph{as soon as} the energy $E^{\dav}_\lambda$ is \emph{strictly negative}. This is the key to our proof of compactness modulo translations for minimizing sequences.
\end{remark}

\begin{proof}
First, let us consider $\dav>0$, that is, strictly positive average diffraction.
If $b\le a$, $\eqref{estimate:fat-tail-positive}$ trivially holds since the right hand side of \eqref{estimate:fat-tail-positive} equals minus infinity. So assume $b-a \ge 1$. Let $a'$ and $b'$ be arbitrary integers satisfying $a\leq a'< b'\leq b$ and $l:=b'-a'$. We will choose suitable $a'$ and $b'$ at the end of the proof.

The lower bound on $\langle f,- \Delta f\rangle$ is based on a discrete version of the well-known IMS localization formula, see Lemma \ref{lem:IMS} in the appendix. Take any smooth cutoff functions $\wti{\chi}_j$, $j=-1,0,1$, with
\begin{SE}
 \item  $\wti{\chi}_j \ge 0$ for  $j=-1,0,1$.
 \item $ \supp \wti{\chi}_{-1} \subset (-\infty,-\frac{1}{4}] $ with
  		$ \wti{\chi}_{-1}>0$  on $ (-\infty, -\frac{3}{8}]$,
  	$ \supp \wti{\chi}_{1} \subset [\frac{1}{4},\infty) $ with
  	$ \wti{\chi}_{1} >0 $  on $  [\frac{3}{8},\infty)$,  and
    $\supp  \wti{\chi}_0 \subset [-\frac{1}{2}, \frac{1}{2}] $ with
    $ \wti{\chi}_0 >0$ on  $ [-\frac{3}{8},\frac{3}{8}]$.
 \end{SE}
and set
$$
	\chi_j:= \frac{\wti{\chi}_j}{\sqrt{\sum_{l=-1}^1 \wti{\chi}_l^2}}
	\quad \text{for }  j=-1,0,1.
$$
  Then, since the denominator is always strictly positive and, by construction,  $\sum_j \chi_j^2=1$, this gives a smooth partition of unity where $\chi_j$ has the same support as $\wti{\chi_j}$, and $\chi_0 =1$ on $[-\frac{1}{4},\frac{1}{4}]$,
  $\chi_{-1} =1$  on  $(-\infty,-\frac{1}{2}] $, and
	$\chi_{1} =1$ on $[\frac{1}{2},\infty)$.
	Finally, define $\xi_j:\Z \to \R$ by
 \begin{align*}
    \xi_j(x)=\chi_j \left(\frac{x-\frac{1}{2}(a'+b')}{b'-a'} \right)\ \  \text{for}\  j=-1,0,1.
 \end{align*}
  Then $\sum_{j=-1}^1\xi_j^2=1$ and $\xi_{-1}=1$ on $[-\infty,a']$, $\xi_1=1$ on $[b',\infty)$ and the supports of $\xi_{-1}$ and $\xi_{1}$ have distance at least $l/2$,
  where $l=b'-a'$. Furthermore, the forward and backward differences
  $D_\pm f(x)= \pm( f(x\pm 1)-f(x))$  satisfy
\begin{align*}
 |D_\pm \xi_{j}(x)|&= |\xi_j(x\pm 1)-\xi _j (x)|
  =\left| \chi_j \left(\f{x-\f{1}{2}(a'+b')}{b'-a'} \pm \f{1}{b'-a'}\right)-\chi_j \left(\frac{x-\frac{1}{2}(a'+b')}{b'-a'} \right)\right|\\
& = \f{1}{b'-a'} |\chi'_j(\zeta)|
\end{align*}
for some $\zeta\in\R$.
Therefore, since $\chi'_j$ is bounded, we see that
\begin{equation}\label{gradient}
\|D_\pm\xi_j\|^2_{\infty}\leq\frac{C}{3(b'-a')^2}=\f{C}{3l^2},
\end{equation}
where $C=3\max_j\|\chi'_j\|_{L^\infty(\R)}^2$.
Using \eqref{gradient} in \eqref{eq:IMS-lower-2}, we get
\begin{align}\label{eq:kinetic}
\|D_+ f\|_2^2 &= \la f,-\Delta f\ra \geq
\sum_{j=-1}^{1} \|D_+ (\xi_j f)\|_2^2-\frac{\|f\|_2^2}{2}\sum_{j=-1}^{1}(\|D_+\xi_j\|^2_\infty+\|D_-\xi_j\|^2_\infty)\notag\\
&\geq \|D_+ (\xi_{-1} f)\|_2^2 + \|D_+ (\xi_{1} f)\|_2^2-\frac{C\|f\|_2^2}{l^2}.
\end{align}
We set $f_{j}:=\xi_{j}f,\ j=-1,1$ and define $f_0:=f-f_{-1}-f_1=(1-\xi_{-1}-\xi_{1})f$.
Obviously, $\|f_j\|_2\leq\|f\|_2$ for  $j=-1,0,1$. Moreover, because of \eqref{ineq:condition-a-b} and  $a'\geq a$, $b'\leq b$, we also have
$$
\|f_j\|_2^2\geq \delta\ \ \ \text{for}\ \ j=-1,1.
$$

Set $h:= f_{-1} + f_1$. Then $f=h+f_0$ and Proposition \ref{prop:N-splitting} shows
\begin{align*}\label{ineq:fat-tail-pf-nonlinear-estimate1}
         N(f)-N(h)-N(f_0)\lesssim \|f_0\|_2\|h\|_2(1+\|f_0\|_2^{\gamma_2-2}+\|h\|_2^{\gamma_2-2}) .
\end{align*}
Using Proposition \ref{prop:N-boundedness}, we have
$$
N(f_0)\lesssim \|f_0\|_2^{\gamma_1}+\|f_0\|_2^{\gamma_2}
	\lesssim \|f_0\|_2^2\big( 1 +  \|f_0\|_2^{\gamma_2-2} \big)
$$
and combining the above two bounds we arrive at
\beq \label{eq:main-split-1}
N(f)-N(h)\lesssim \|f_0\|_2\|f\|_2(1+\|f\|_2^{\gamma_2-2})
\eeq
where we used $\|f_0\|_2, \|h\|_2 \leq \|f\|_2$.

Since the supports of $f_{-1}$ and $f_1$ have distance at least $l/2=(b'-a')/2$, we can again use Proposition
\ref{prop:N-splitting} with $\alpha=\frac{1}{4}$ to split $N(h)$ as
\begin{align}\label{eq:main-split-2}
N(h)-N(f_{-1})-N(f_1)
&\lesssim  (l/2)^{-l/8}\|f_{-1}\|_2\|f_1\|_2(1+\|f_{-1}\|_2^{\gamma_2-2}+\|f_1\|_2^{\gamma_2-2})\notag\\
& \lesssim (l/2)^{-l/8}\|f\|_2^2(1+\|f\|_2^{\gamma_2-2}).
\end{align}
Combining \eqref{eq:main-split-1} and \eqref{eq:main-split-2}, we get
\begin{align*}
 	N(f) - N(f_{-1}) - N(f_1)
 	 \lesssim  \left( \|f_0\|_2\|f\|_2 + (l/2)^{-l/8}\|f\|_2^2 \right)
 	  \left(1+\|f\|_2^{\gamma_2-2}\right),
 \end{align*}
 which together with  \eqref{eq:kinetic}  yields
 \begin{align} \label{eq:energy-estimate}
     H(f)-H(f_{-1})-H(f_1)
     \gtrsim
      -\left[ \frac{\|f\|_2^2}{l^2}+\left(\|f_0\|_2\|f\|_2 +  (l/2)^{-l/8}\|f\|_2^2\right)
 	  \left(1+\|f\|_2^{\gamma_2-2}\right)\right] .
 \end{align}

Once we have such a bound on the splitting of the energy, we use a reasoning similar to the one in \cite{HuLee 2012-existence}: 		
 By definition of $f_0$, we have $\|f_0\|_2^2\le \sum_{x=a^\prime+1}^{b^\prime-1}|f(x)|^2$.
 To choose $a'$ and $b'$, set $I_\eta:=\{\eta+1, \eta+2,\ldots,\eta+l-1\}$ when $l\ge 2$,
 $ I_\eta \coloneq \emptyset$ when $l=1$, and note that, since the number of integers in  $[a,b-l]$ is $b-a-l+1$,
    \begin{equation*} \label{eq:pidgeonholing}
        \begin{split}
        (b-a-l+1)\min_{a\le \eta\le b-l} \sum_{x\in I_\eta} |f(x)|^2
        &\le
        \sum_{\eta= a}^{b-l} \sum_{x\in I_\eta} |f(x)|^2
            \le \sum_{x= a+1}^{b-1} \sum_{\eta= x-l+1}^{x-1} |f(x)|^2 \\
           &\le   (l-1) \|f\|_2^2 \,\, .
    \end{split}
    \end{equation*}
Hence there exists $\eta^\prime$ with $a\le \eta^\prime \le b-l$ and
    \bdm
         \sum_{x= \eta^\prime+1}^{\eta^\prime+l-1} |f(x)|^2
            \le \frac{l-1}{b-a-l+1}\|f\|_2^2\,\, .
    \edm
With $a^\prime=\eta^\prime$ and $b^\prime=\eta^\prime+l$ we therefore have
$$
\|f_0\|_2^2\leq \frac{l-1}{b-a-l+1}\|f\|_2^2\,\,.
$$
Plugging this into  $\eqref{eq:energy-estimate}$ yields
\begin{align}
     &H(f)- H(f_{-1})-H(f_1) \notag\\
     &\gtrsim -\left[ \frac{\|f\|_2^2}{l^2}+\left(\left(\frac{l-1}{b-a-l+1}\right)^{1/2}\|f\|_2^2 +  (l/2)^{-l/8}\|f\|_2^2 \right)
 	  \left(1+\|f\|_2^{\gamma_2-2}\right)\right]\notag \\
 	 &\ge -\|f\|_2^2\left(1+\|f\|_2^{\gamma_2-2}\right)
       \left[ \frac{1}{l^2}
       +\left(\frac{l-1}{b-a-l+1}\right)^{1/2} + (l/2)^{-l/8}\right].\label{eq:great}
 \end{align}
Since $\|f\|_2^2=\lambda$, $\|f_j\|_2^2 \geq \delta,\  j=-1,1$ and $\|f_{-1}\|_2^2+\|f_1\|_2^2 \leq \lambda$, Proposition \ref{prop:strict-subadditivity} shows
 \begin{align*}
 	H(f_{-1})+ H(f_1)\ge \left[1-(2^{\f{\gamma_0}{2}}-2)\left(\f{\delta}{\lambda}\right)^{\f{\gamma_0}{2}}\right]
    E_{\lambda}^\dav
 \end{align*}
 and inequality \eqref{eq:great} yields
 \begin{equation}
 \begin{split}	\label{ineq:fat-tail-pf-nonlinear-estimate3}
     H(f)- &\left[1-(2^{\f{\gamma_0}{2}}-2)\left(\f{\delta}{\lambda}\right)^{\f{\gamma_0}{2}}\right]
    E_{\lambda}^\dav
     \gtrsim\\
 &- \left(\lambda+\lambda^{\gamma_2/2}\right)
      \left[ \frac{1}{l^2}
       +\left(\frac{l-1}{b-a-l+1}\right)^{1/2} + (l/2)^{- l/8}\right]
 \end{split}
 \end{equation}
for any $1\leq l\le b-a$.

Finally, we choose $l\in \N$ with $l\leq (b-a+1)^{1/2} < l +1$. Note that this is allowed, since when $b-a=1$, we have $l=1$, and when $b-a\ge 2$, then $1\le l\le (b-a+1)^{1/2}\le b-a$.  With this choice of $l$ we have
\begin{align*}
	\f{1}{l^2} \le \frac{1}{l^{1/2}}\le \left((b-a+1)^{1/2}-1\right)^{-1/2},
\end{align*}
\begin{align*}
	\left(\frac{l-1}{ b-a -l+1}\right)^{1/2}\le \left(\frac{(b-a+1)^{1/2}-1}{b-a - (b-a+1)^{1/2}+1}\right)^{1/2}
	\le  \left((b-a+1)^{1/2}-1\right)^{-1/2}
\end{align*}
and
\begin{align*}
	(l/2)^{-l/8}\lesssim l^{-1/2}\le ((b-a+1)^{1/2}-1)^{-1/2}.
\end{align*}
Therefore, \eqref{ineq:fat-tail-pf-nonlinear-estimate3} yields \eqref{estimate:fat-tail-positive}.

If $\dav=0$, we do not have the term $\f{1}{l^2}$ in \eqref{ineq:fat-tail-pf-nonlinear-estimate3} and get the same estimate \eqref{estimate:fat-tail-positive}.
\end{proof}

An immediate consequence of Proposition \ref{prop:fat-tail proposition} is
\begin{proposition}[Tightness]\label{prop: tightness modulo translations}
Assume that $E^{\dav}_\lambda<0$.
Let $ (f_n)_{n\in\N}\subset l^2(\mathbb{Z})$ be a minimizing sequence for the variational problem $\eqref{eq:GT}$ with $\lambda=\|f_n\|_2^2>0$. Then there exist shifts $\xi_n$ such that
\begin{equation*}
\lim_{R\rightarrow\infty} \sup_{n\in{\mathbb N}}\sum_{|x-\xi_n|>R}|f_n(x)|^2=0.
\end{equation*}
\end{proposition}
\begin{proof}
	 Since the function $s \mapsto (\sqrt{s+1}-1)^{-1/2}$ is decreasing on $(0,\infty)$ and goes to zero at infinity, Proposition \ref{prop:fat-tail proposition} has the same consequences as \cite[Proposition 2.4]{HuLee 2012-existence} replacing \cite[inequality (2.29)]{HuLee 2012-existence} by \eqref{estimate:fat-tail-positive}.
\end{proof}

To prove Theorem \ref{thm:existence}, we need two more results on the continuity and differentiablity of the non-linear functional $N(f)$. The proof mimics the one in  \cite{ChoiHuLee2015} for the continuous case and is therefore omitted.

\begin{lemma}\label{lem:continuous mapping}
The functional $N:l^2(\Z)\to \R$ given by
$$
 f \mapsto N(f)= \int_{\R}\sum _{x\in \Z} V(|T_r f(x)|)\mu(dr)
$$
is locally Lipshitz continuous on $l^2(\Z)$.
\end{lemma}

\begin{lemma} \label{lem:differentiability}
For any $f\in l^2(\Z)$, the functional $N$ as above
 is continuously differentiable with derivative
 \begin{align*}
 	l^2(\Z)\ni h\mapsto DN(f)[h]=
 	   \re \int_\R \left\langle \left[ V'(|T_rf|)\sgn{T_r f} \right],T_r h \right\rangle \mu(dr),
 \end{align*}
where $\sgn{z}:=\f{z}{|z|}$ if $z\neq 0$ and $\sgn{0}:=0$. In particular, the nonlinear Hamiltonian given in \eqref{eq:energy} is continuously differentiable with derivative
\begin{align*}
	l^2(\Z)\ni h\mapsto DH(f)[h]=
 	   \dav \re\la D_+f, D_+h \ra
 	   -\re \int_\R \left\langle \left[ V'(|T_rf|)\sgn{T_r f} \right],T_r h \right\rangle \mu(dr),
\end{align*}
\end{lemma}


\begin{remark}\label{rem:DN}
Recall that we assume that the nonlinearity $P$ is odd, so it is of the form
$P(a)= p(|a|)a$ for $a\in\R $. If $V'(a) = P(a)$ for all $a\ge 0$, then
$V'(|z|)\sgn{z}= p(|z|)z = P(z)$ for all $z\in\C$, and therefore
\begin{align*}
DN(f)[h] & = \re \int_\R \left\langle  P(T_rf),T_rh \right\rangle \mu(dr) \\
	&= \re \left\langle\int_\R T_r^{-1}\left[ P(T_rf) \right]  \mu(dr), h\right\rangle
\end{align*}
is, modulo the real part,  the weak form of the nonlinearity in the diffraction management equation \eqref{eq:GT-weak}.
\end{remark}

It remains to prove Theorem \ref{thm:existence}.
A last step in our existence proof of minimizers of the variational problems \eqref{eq:min} is the following characterization of strong convergence in $l^2(\Z)$.

\begin{lemma}[Lemma A.1 in \cite{HuLee 2012-existence}]\label{lem:strong-convergence}
A sequence $ (f_n)_{n\in\N}\subset l^2(\Z)$ is strongly converging to $f$ in $l^2(\Z)$
if and only if it is weakly convergent to $f$ and the sequence is tight, i.e.,
 \begin{equation}\label{eq:tight-discr}
   \lim_{L\to\infty} \limsup_{n\to\infty} \sum_{|x|>L} |f_n(x)|^2 = 0.
 \end{equation}
\end{lemma}
\begin{proof}[Sketch of the proof:]
 Let $P_lf:= \id_{[-l,l]} f$ and note that the range of $P_l$ is finite dimensional,
 in fact, $2l+1$ dimensional. Thus, if $f_n$ converges weakly to $f$, then
 $\lim_{n\to\infty}\|P_l(f-f_n)\|_2=0$. Since
 \begin{align*}
 	\|f- f_n\|_2 \le \|P_l(f-f_n)\|_2 + \|(1-P_l)(f-f_n)\|_2
 \end{align*}
 we see that for all $l\in\N$
 \begin{align*}
 	 \limsup_{n\to\infty} \|f- f_n\|_2 \le \limsup_{n\to\infty} \|(1-P_l)(f-f_n)\|_2
 	 \le  \|(1-P_l)f\|_2 + \limsup_{n\to\infty} \|(1-P_l) f_n\|_2 .
 \end{align*}
 As $l\to\infty$, the first term goes to zero since $f\in l^2(\Z)$ and the second goes to zero because of \eqref{eq:tight-discr}. So $f_n$ converges to $f$ in norm.

 Conversely, if $f_n$ converges to $f$ in norm, then it is easy to see that it converges to
 $f$ weakly and \eqref{eq:tight-discr} holds.
\end{proof}

Now we can come to the
\bpf[Proof of Theorem \ref{thm:existence}]
We know from Lemma \ref{lem:E-boundedness} and
\ref{lem:negativity} that $-\infty<E^{\dav}_\lambda\le 0 $. Assume that $E^{\dav}_{\lambda}=0$ for some $\lambda>0$. Define the sequence $ (f_n)_n$ by
\begin{align*}
	f_n(x) \coloneq c_n \id_{[-n,n]}(x)
\end{align*}
with $c_n= \left(\tfrac{\lambda}{2n+1}\right)^{1/2}$. Then $\|f_n\|_2^2=\lambda$. Note that $f_n$ converges weakly to zero and that any shift of $f_n$ also converges weakly to zero. So the sequence $ (f_n)_n$ is not precompact in $l^2(\Z)$ modulo translations. Moreover, we have
\begin{align*}
	\|D_+ f_n\|_2^2 = 2c_n^2 \to 0 \quad\text{as } n\to\infty
\end{align*}
and, because of \eqref{eq:Vbound},
\begin{align*}
	|N (f_n)|\lesssim \int (\|T_r f_n\|_{\gamma_1}^{\gamma_1} + \|T_r f_n\|_{\gamma_2}^{\gamma_2}) \, \mu(dr)
	\lesssim \|f_n\|_{\gamma_1}^{\gamma_1} + \|f_n\|_{\gamma_2}^{\gamma_2}
\end{align*}
where we also used the bound \eqref{eq:lpbound} from Lemma \ref{lem:useful}. Since for any
$\gamma>2$
\begin{align*}
	\|f_n\|_{\gamma}^\gamma = \left(\frac{\lambda}{2n+1}\right)^{\gamma/2} (2n+1) \to 0
\end{align*}
as $n\to\infty$, we have $N (f_n)\to 0$ as $n\to\infty$. Thus $f_n$ is a minimizing sequence for \eqref{eq:min} which is not precompact modulo translations. By contrapositive, this shows that if every minimizing sequence is precompact modulo translations, then $E^{\dav}_\lambda<0$.

 Conversely, assume that $E^{\dav}_\lambda< 0 $ and  let $ (f_n)_{n\in\N}\subset l^2(\Z)$ be a minimizing sequence of the variational problem \eqref{eq:GT}.
 First, applying Proposition \ref{prop: tightness modulo translations}, we see that there exist shifts $\{\xi_n\}$ such that for any $\epsilon>0$ there exists an $R_\epsilon >0$ for which
\begin{equation}\label{eq:shift1}
\sum_{|x-\xi_n|>R_\epsilon}|f_n(x)|^2 \leq \epsilon \ \text{ for any } n \in \N.
\end{equation}
Define the shifted sequence $(\tilde{f}_n)_n$ by $\tilde{f}_n(x):=f_n(x-\xi_n)$ for
$x\in \Z$. It is also a minimizing sequence, due to the invariance of the Hamiltonian $H$ given in \eqref{eq:energy} under shifts.

Noting that $(\|\tilde{f}_n\|_2)$ is bounded as it is a minimizing sequence, we can see there exists a subsequence, also denoted by $(\tilde{f}_n)_{n\in\N}$, which converges weakly to some $\varphi$ in $l^2(\Z)$.
Due to \eqref{eq:shift1} the shifted sequence $(\tilde{f}_n)_{n\in\N}$ is tight in the sense of Lemma \ref{lem:strong-convergence}, hence by Lemma \ref{lem:strong-convergence} it  converges strongly in $l^2(\Z)$ and $\|\varphi\|_2^2=\lim_{n\rightarrow\infty}\|\tilde{f}_n\|_2^2=\lambda$. Thus the minimizing sequence $ (f_n)_n$ is precompact modulo tranlations.

Moreover, it follows from Lemma \ref{lem:continuous mapping} that $H(\varphi)=\lim_{n\rightarrow \infty}H(g_n)= E_\lambda$ which finishes the proof of existence of a minimizer for the constraint variational problem \eqref{eq:min}.

Now we prove that any minimizer is a solution of the associated Euler-Lagrange equation \eqref{eq:GT-weak} for some Lagrange multiplier
$\omega\in\R$. This is standard in the calculus of variations, for the convenience of the reader, we will give the argument.
Let $\varphi$ be a minimizer for \eqref{eq:min} and $h\in l^2(\Z)$ arbitrary. Furthermore define
\begin{align*}
	G(t,s)&\coloneq \la \varphi+t h+s\varphi, \varphi+t h+s\varphi \ra \\
	F(t,s)&\coloneq H(\varphi+t h+s\varphi) ,
\end{align*}
then a short calculation gives
\begin{align*}
	\nabla G(t,s) = \left( \begin{array}{c}
		\partial_t G(t,s) \\
		\partial_s G(t,s)
	\end{array} \right)
	=
	2 \left( \begin{array}{c}
		\re \la \varphi+t h+s\varphi, \varphi\ra  \\
		\re \la\varphi+t h+s\varphi, h\ra
	\end{array} \right)
\end{align*}
and
\begin{align*}
	\nabla F(t,s) = \left( \begin{array}{c}
		\partial_t F(t,s) \\
		\partial_s F(t,s)
	\end{array} \right)
	=
	    \left( \begin{array}{c}
		DH(\varphi+t h+s\varphi)[\varphi] \\
		DH(\varphi+t h+s\varphi)[h]
	\end{array} \right)
\end{align*}
where $DH$ is the derivative of the nonlinear Hamiltonian,
\begin{align*}
	DH(\varphi)[h]
		&= \dav \re \la -\Delta\varphi, h \ra - DN(\varphi)[h] \\
		&= \dav \re\la D_+\varphi, D_+h \ra - \re \int_\R \la V'(|T_r\varphi|)\sgn{T_r\varphi}, T_rh \ra\, \mu(dr) \\
		&= \dav \re\la D_+\varphi, D_+h \ra - \re \int_\R \la P(T_r\varphi), T_rh \ra\, \mu(dr)
\end{align*}
where we used Remark \ref{rem:DN} for the last equality.

 We have $\partial_t G(0,0)=\la\varphi,\varphi\ra=\lambda>0$, hence by the implicit function theorem, there exists $\delta>0$ and a differentiable function
 $g:(-\delta,\delta)\to \R $ with $g(0)=0$ such that $G(g(s),s)=G(0,0)=\lambda$ for all $|s|<\delta$. Thus, since $\varphi$ is a minimizer of the constrained minimization problem \eqref{eq:min},  the function
\begin{align*}
	(-\delta,\delta)\ni s\mapsto \tilde{F}(s)\coloneq F(g(s),s)
\end{align*}
has a local minimum at $s=0$ and together with the chain rule this  implies
\begin{align}\label{eq:euler lagrange 0}
	0 &= \partial_s \tilde{F}(s)\vert_{s=0}= \partial_t F(0,0) g'(0) + \partial_s F(0,0)
		= DH(\varphi)[\varphi] g'(0) + DH(\varphi)[h]
\end{align}
Moreover, since $G(g(s),s)$ is constant, we also have
\begin{align*}
	0= \partial_t G(0,0) g'(0) +\partial_s G(0,0) = 2\lambda g'(0) + 2\re \la \varphi,h \ra
\end{align*}
solving for $g'(0)$ and plugging it back into \eqref{eq:euler lagrange 0} yields
\begin{align}\label{eq:euler lagrange 1}
	\omega \re\la\varphi,h\ra = \dav \re\la D_+\varphi, D_+h \ra - \re \int_\R \la V'(|T_r\varphi|)\sgn{T_r\varphi}, T_rh \ra\, \mu(dr)
\end{align}
with the Lagrange multiplier
\begin{align}\label{eq:lagrange multiplier}
	\omega=\omega(\varphi)\coloneq \frac{DH(\varphi)[\varphi]}{\lambda} \in \R \,.
\end{align}
Replacing $h$ by $-ih$ in \eqref{eq:euler lagrange 1} yields
\begin{align*}
	\omega \im\la\varphi,h\ra = \dav \im\la D_+\varphi, D_+h \ra - \im \int_\R \la V'(|T_r\varphi|)\sgn{T_r\varphi}, T_rh \ra\, \mu(dr)
\end{align*}
and together with \eqref{eq:euler lagrange 1} this proves \eqref{eq:GT-weak}.

It remains to prove that $\omega< 2E^{\dav}_\lambda$. Recall that assumption \ref{ass:A2} states that
\begin{align*}
	V'(a)a \ge \gamma_0 V(a) \quad \text {for all } a>0. 	
\end{align*}
Thus
\begin{align*}
	DN(\varphi)[\varphi]
	= \int_\R \sum_{x\in\Z} V'(|T_r\varphi(x)|)|T_r\varphi(x)| \,\mu(dr)
	\ge \gamma_0 \int_\R \sum_{x\in\Z} V(|T_r\varphi(x)|) \,\mu(dr) =\gamma_0 N(\varphi)
\end{align*}
and since $E^{\dav}_\lambda<0$, we must have $N(\varphi)>0$ for any minimizer
$\varphi$, so \eqref{eq:lagrange multiplier} gives
\begin{align}\label{eq:negative-Lagrange-multiplier}
	\omega(\varphi)\lambda= DH(\varphi)[\varphi] &= \dav\la D_+\varphi,D_+\varphi\ra -DN(\varphi)[\varphi]
	\le \dav\la D_+\varphi,D_+\varphi\ra -\gamma_0 N(\varphi) \\
	&= 2H(\varphi) -(\gamma_0-2) N(\varphi) < 2H(\varphi) = 2E^{\dav}_\lambda <0
\end{align}
for all $\varphi$ in the ground state set $\calM^\dav_\lambda$.

\end{proof}

\section{Threshold phenomena}\label{sec:threshold}
As we showed in the previous section, assumptions \ref{ass:A1} and \ref{ass:A2} guarantee the existence of minimizers for arbitrary $\lambda>0$ and $\dav\ge 0$ as soon as the ground state  energy $E^{\dav}_\lambda$ is \emph{strictly negative}.  In this section we will prove  a threshold phenomenon: There exists $0\le\lambda_{\mathrm{cr}}\le \infty$ such that solutions exist if the power $\lambda= \|f\|_2^2>\lambda_{\mathrm{cr}}$.
Furthermore $\lambda_{\mathrm{cr}}<\infty$ under assumption \ref{ass:A3}.

For pure power law nonlinearities and the model case $d_0=\id_{[0,1)}-\id_{[1,2]}$ for the diffraction profile, this had been partly investigated earlier in \cite{Moeser05} for
the diffraction management equation and for pure power nonlinearities in \cite{weinstein} for the discrete nonlinear Schr\"odinger equation. We are not aware of any work which investigates threshold phenomena for general nonlinearities obeying only \ref{ass:A1} and \ref{ass:A2}.

 In the following we will always assume, without explicitly mentioning it every time,
 that $\mu$ is a finite measure on $\R$ with compact support, that is, there exists $0<B<\infty$ such that $\supp\mu\subset [-B,B]$.
Our main result in this section is
\begin{theorem}[Threshold phenomenon]
 \label{thm:threshold-phenomena}
	Assume that $V$ obeys \ref{ass:A1} 
 and \ref{ass:A2}. Then \\
	\itemthm\label{thm:threshold-phenomena-1} For any average diffraction
	$\dav\ge  0$ and any $\lambda>0$ we have $E^{\dav}_\lambda\le 0$, the map $\lambda\mapsto E^{\dav}_\lambda$ is decreasing on $(0,\infty)$, and there exists  a critical threshold
	$0\le \lambda_{\mathrm{cr}}(\dav)\le \infty$ such that for
	$0 < \lambda < \lambda_{\mathrm{cr}}(\dav)$ we have $E^{\dav}_\lambda = 0$ and for $\lambda>\lambda_{\mathrm{cr}}(\dav)$ we have $-\infty<E^{\dav}_\lambda < 0$. \\
	\itemthm\label{thm:threshold-phenomena-2} 	If $\lambda> \lambda_{\mathrm{cr}}$, then minimizers of \eqref{eq:min} exist and any minimizing sequence is, up to translations, precompact in $l^2(\Z)$ and thus has a subsequence which converges, up to translations, to a minimizer. \\
	\itemthm\label{thm:threshold-phenomena-3} If $0< \lambda< \lambda_{\mathrm{cr}}(\dav)$ and $\dav>0$, then no minimizers of the variational problem \eqref{eq:min} exist. \\
	\itemthm\label{thm:threshold-phenomena-4} $\lambda_{\mathrm{cr}}(\dav)< \infty$ for all $\dav\ge 0$ if and only if there exists
		$f\in l^2(\Z)$ such that $N(f)>0$.  \\
	\itemthm\label{thm:threshold-phenomena-5} If in assumption \ref{ass:A1} we have $\gamma_1\ge 6$, then $\lambda_{\text{cr}}(\dav)>0$ for all $\dav>0$.
\end{theorem}

\begin{remark}
The precise definition of  $\lambda_{\mathrm{cr}}(\dav)$ is given below in Definition \ref{def:threshold}.
When $\lambda>\lambda_{\mathrm{cr}}(\dav)$ we have $E^{\dav}_\lambda<0$ and Theorem \ref{thm:existence} shows that  any minimizing sequence is precompact modulo translations, that minimizers exist and that these minimizers yield solutions of \eqref{eq:GT} for some Lagrange multiplier $\omega < 2E^{\dav}_\lambda/\lambda<0$.

Since $E^{\dav}_\lambda=0$ when $0<\lambda<\lambda_{\mathrm{cr}}$, Theorem \ref{thm:existence} also shows that there are minimizing sequences which are not precompact modulo translations in this case. Nevertheless, it could be that minimizers still exist. At least when $\dav>0$, Theorem \ref{thm:threshold-phenomena} shows that this cannot be the case.
At the moment, we need $\dav>0$ to conclude nonexistence of minimizers when $0<\lambda<\lambda_{\mathrm{cr}}$.
\end{remark}

We give the proof of Theorem \ref{thm:threshold-phenomena} at the end of this section after
some preparations.
Recall
\begin{align*}
	H(f) = \frac{\dav}{2}\|D_+f\|_2^2  - N(f)
\end{align*}
and
\begin{align*}
	E^{\dav}_\lambda = \inf\{H(f):\, f\in l^2(\Z), \|f\|_2^2=\lambda\} .
\end{align*}
Given $f\in l^2(\Z)$ with $\lambda=\|f\|_2^2>0$, write it as $f=\sqrt{\lambda} h$ then $h\in l^2(\Z)$ with $\|h\|_2=1$ and
\begin{align}\label{eq:energy-threshold}
	H(f) = \frac{\dav}{2}\|D_+f\|_2^2  - N(f)
		= \|D_+f\|_2^2  \left( \frac{\dav}{2}- \frac{N(\sqrt{\lambda}h)}{\lambda \|D_+h\|_2^2} \right).
\end{align}
In the case of vanishing average diffraction, we can still write
\begin{align*}
	H(f)= -N(f) = - \|D_+f\|_2^2  \left(\frac{N(\sqrt{\lambda}h)}{\lambda \|D_+h\|_2^2} \right),
\end{align*}
so defining\footnote{Note that the kernel of $D_+$ on $l^2(\Z)$ is trivial, so $R(\lambda,h)$ is well defined for any $h\neq 0$.}
\begin{align*}
	R(\lambda,h)\coloneq \frac{N(\sqrt{\lambda}h)}{\lambda \|D_+h\|_2^2}
\end{align*}
and
\begin{align}\label{def:R}
	R(\lambda)\coloneq \sup_{\|h\|_2=1} R(\lambda,h)
		=  \sup_{\|f\|_2^2=\lambda} \frac{N(f)}{ \|D_+f\|_2^2}
\end{align}
 we see that the following holds
\begin{lemma}\label{lem:threshold-basic}
	For any $\dav\ge 0$ and $\lambda> 0$ one has $E^{\dav}_\lambda<0$ if and only if $R(\lambda,h)>\frac{\dav}{2}$ for some $h\in l^2(\Z)$ with $\|h\|_2=1$ and this is the case if and only if $R(\lambda)>\frac{\dav}{2} $.
\end{lemma}

The function $R$ defined above has very interesting properties, which make $R$ ideal for the study of the threshold phenomenon.
First we give a simple Lemma, which is at the heart of our study of $R$.

\begin{lemma}\label{lem:R-lower}
	Assume that $V$ obeys assumption \ref{ass:A2}. Then for any $\lambda_2\ge \lambda_1>0$ one has
	\begin{align}\label{eq:R-lower}
		R(\lambda_2) \ge \left( \frac{\lambda_2}{\lambda_1} \right)^{\frac{\gamma_0-2}{2}} R(\lambda_1).
	\end{align}
\end{lemma}
\begin{remark}\label{rem:pure-power-law}
  For a pure power law nonlinearity, given by $V(a)=c a^\gamma$ for some $\gamma>2$ and $c>0$, one can explicitly calculate
  \begin{align*}
  	R(\lambda) =
  		\sup_{\|h\|_2=1} \frac{N(\sqrt{\lambda} h)}{\lambda \|D_+ h\|_2^2}
  		= \lambda^{\frac{\gamma-2}{2}} R_0
  		\quad \text{with } R_0= \sup_{\|h\|_2=1} \frac{N(h)}{\|D_+ h\|_2^2} \in (0,\infty] .
  \end{align*}
  Thus inequality \eqref{eq:R-lower} is very natural. Using the bound
  \eqref{eq:weinstein-type-ineq} below one sees that
  \begin{align*}
    R_0\le \sup_{\|h\|_2=1}  c \int_\R \|T_r h\|_2^{\gamma-2} \mu(dr) = c \mu(\R) <\infty
  \end{align*}
  for all $\gamma\ge 6$ since
  $\|D_+ T_rh\|_2^2 = \la T_rh, -\Delta T_rh\ra = \la h,-\Delta h\ra = \|D_+h\|_2^2$,
  using that $\Delta$ and  $T_r$ commute.
  To see that $R_0=\infty$ if $2<\gamma<6$
  is a little bit trickier.
  If $2<\gamma<6$, then Lemma \ref{lem:negativity} shows
  $E^\dav_\lambda<0$ for all $\dav\ge 0$ and all $\lambda>0$. So with Lemma \ref{lem:threshold-basic} for $\lambda=1$ this gives
  $R_0=R(1)>\dav/2$  for all $\dav\ge 0$.
  Thus  $R_0=\infty$ in this case.
\end{remark}

\begin{proof}[Proof of Lemma \ref{lem:R-lower}]
	Fix $h\in l^2(\Z)\setminus\{0\}$ and define
	\begin{align*}
		A(s)\coloneq  s^{-2}N(sh)
	\end{align*}
	for $s>0$.
	Because of Lemma \ref{lem:differentiability}, $A$ is differentiable with derivative
	\begin{align*}
		A'(s) = s^{-3} \Big( DN(sh)[sh] -2N(sh) \Big)
	\end{align*}
	where
	\begin{align*}
		DN(sh)[sh] - 2N(sh)
			&= \int_\R \sum_{x \in \Z} \left[ V'(|T_r(sh)(x)|)|T_r(sh)(x)| - 2V(|T_r(sh)(x)|)\right] \,\mu(dr) \\
			&\ge 	(\gamma_0-2) N(sh)
	\end{align*}
	where the lower bound follows from assumption \ref{ass:A2}.
	Thus we arrive at the first order differential inequality
	\begin{align}\label{eq:A'lowerbound}
		A'(s) \ge \frac{\gamma_0-2}{s} A(s)
	\end{align}
	for all $s>0$. Using the integrating factor $s^{2-\gamma_0}$, one sees that this implies
	$\frac{d}{ds} (s^{2-\gamma_0}A(s)) \ge 0 $
	and thus
	\begin{align*}
		s^{2-\gamma_0} A(s) \ge s_0^{2-\gamma_0} A(s_0)
	\end{align*}
	for all $0<s_0\le s$. Since $R(\lambda,h) = A(\sqrt{\lambda})/\|D_+h\|_2^2$, this proves
	\begin{align*}
		R(\lambda_2,h) \ge \left( \frac{\lambda_2}{\lambda_1} \right)^{\frac{\gamma_0-2}{2}} R(\lambda_1,h).
	\end{align*}
	for all $0<\lambda_1\le \lambda_2$ and taking the supremum over all $h\in l^2(\Z)$ with
	$\|h\|_2=1$ gives  \eqref{eq:R-lower}.
\end{proof}

\begin{corollary}[Properties of R]\label{cor:properties-of-R}
	Assume that $V$ obeys Assumption \ref{ass:A2}. \\
	\itemthm \label{cor:properties-of-R-1} For any $0\le a\le \infty$
	\begin{align}
		\begin{array}{lll}\label{eq:properties-of-R-1}
			\text{there exist } \lambda_0>0 \text{ with } R(\lambda_0) \ge a
			&\Rightarrow &
				 R(\lambda) \ge a \text{ for all } \lambda\ge  \lambda_0  \\
			\text{there exist } \lambda_0>0 \text{ with } R(\lambda_0) \le  a
			&\Rightarrow &
				 R(\lambda) \le  a \text{ for all } 0<\lambda\le \lambda_0
		\end{array} \,
	\end{align}
	Moreover, for any $0<a<\infty$
	\begin{align}	
		\begin{array}{lll}\label{eq:properties-of-R-2}
			\text{there exist } \lambda_0>0 \text{ with } R(\lambda_0) \ge a
			&\Rightarrow &
				 R(\lambda)> a \text{ for all } \lambda > \lambda_0  \\
						\text{there exist } \lambda_0>0 \text{ with } R(\lambda_0) \le a
			&\Rightarrow &
				 R(\lambda) < a \text{ for all } 0<\lambda< \lambda_0
		\end{array}  .
	\end{align}
   Furthermore, we have the equivalences
	\begin{align}\label{eq:R-equivalence}
		\begin{array}{lcl}
		\text{there exists } \lambda>0 \text{ with } R(\lambda)>0
		&\Leftrightarrow & \quad \lim_{\lambda\to\infty} R(\lambda) =\infty \\
		\text{there exists } \lambda>0 \text{ with } R(\lambda)<\infty
		&\Leftrightarrow & \quad \limsup_{\lambda\to 0+} R(\lambda) \le 0 \, .
		\end{array}
	\end{align}
  \itemthm\label{cor:properties-of-R-2}
  Define the set $A_0\coloneq\{\lambda>0:\, R(\lambda) > 0\}$, then it is either empty or an unbounded interval. Moreover, the map $R$ is increasing on $A_0$ and it is strictly increasing where it is finite.
\end{corollary}

\begin{remarks}
	\itemthm
	Even though Lemma \ref{lem:E-boundedness} shows that under Assumption \ref{ass:A1} the energy is negative, this is not enough to conclude that $R(\lambda)\ge 0$ for all $\lambda>0$, in general.  \\
	\itemthm
	All the conclusions of Corollary \ref{cor:properties-of-R} are trivially true if $V(a)=ca^\gamma$ is a pure power law for some $\gamma>2$ and $c>0$, since in this case
	$R(\lambda)= R_0\lambda^{(\gamma-2)/2}$ as in Remark \ref{rem:pure-power-law}. \\
	\itemthm
		The first equivalence in \eqref{eq:R-equivalence} shows that we have the dichotomy that either $R(\lambda)\le 0$ for all $\lambda>0$, or $\lim_{\lambda\to \infty}R(\lambda)=\infty$.
		
		Similarly, the second equivalence in  \eqref{eq:R-equivalence} shows
		the dichotomy that either $R(\lambda)=\infty $ for all $\lambda>0$
		or $\limsup_{\lambda\to 0+}R(\lambda) \le 0$.
\end{remarks}
\begin{proof}[Proof of Corollary \ref{cor:properties-of-R}]
 	\ref{cor:properties-of-R-1} The implications of \eqref{eq:properties-of-R-1} and \eqref{eq:properties-of-R-2} follow directly from Lemma \ref{lem:R-lower}.  	
 	If $\lambda>\lambda_0$, then choosing $\lambda_1=\lambda_0$ and $\lambda_2=\lambda$ in
 	\eqref{eq:R-lower} shows
 	\begin{align} \label{eq:lower}
 		R(\lambda)\ge \left(\frac{\lambda}{\lambda_0}\right)^{\frac{\gamma_0-2}{2}} R(\lambda_0)\, .
	\end{align}
	Let $0<a<\infty$. If $R(\lambda_0)=\infty$ we also have $R(\lambda)=\infty$ for all $\lambda\ge \lambda_0$. If $R(\lambda_0)=\infty$, then \eqref{eq:lower} shows $R(\lambda)=\infty>a$.
 	If $a\le R(\lambda_0)<\infty$, then necessarily $R(\lambda_0)>0$, hence
 	$\left(\tfrac{\lambda}{\lambda_0}\right)^{\frac{\gamma_0-2}{2}} R(\lambda_0)>R(\lambda_0)$ and \eqref{eq:lower} again gives $R(\lambda)>a$. So the first implication of \eqref{eq:properties-of-R-2} is true. Hence also the first implication of \eqref{eq:properties-of-R-1} is true when $a$ is strictly positive and finite, but when $a=\infty$ or $a=0$, the first implication of \eqref{eq:properties-of-R-1} immediately follows from \eqref{eq:lower}.
 	This finishes the proof of the first implications in \eqref{eq:properties-of-R-1} and \eqref{eq:properties-of-R-2}.
 	\smallskip
 	
 	Now let $0< \lambda<\lambda_0$. Choosing
	$\lambda_1=\lambda$ and $\lambda_2=\lambda_0$ in \eqref{eq:R-lower} gives the upper bound
	\begin{align}\label{eq:upper}
			R(\lambda)\le \left( \frac{\lambda}{\lambda_0} \right)^{\frac{\gamma_0-2}{2}}  R(\lambda_0) \, ,
	\end{align}
 	If $0<a<\infty$ and $R(\lambda_0)=0$, then \eqref{eq:upper} shows $R(\lambda)\le 0 <a$.
 	If $0<R(\lambda_0)\le a$, then $\left( \frac{\lambda}{\lambda_0} \right)^{\frac{\gamma_0-2}{2}}  R(\lambda_0)< R(\lambda_0)$, so \eqref{eq:upper}  again yields $R(\lambda)<a$. This proves the second implication in \eqref{eq:properties-of-R-2}.
 	The second implication of \eqref{eq:properties-of-R-1} when $a=\infty$ or $a=0$ immediately follows from \eqref{eq:upper}.
 	This finishes the proof of \eqref{eq:properties-of-R-1} and \eqref{eq:properties-of-R-2}.
 	\smallskip
	
	For the proof of \eqref{eq:R-equivalence} assume first that $\lim_{\lambda\to\infty}R(\lambda)=\infty$. Then, of course, there exists $\lambda>0$
	with $R(\lambda)>0$. On the other hand, if there exists $\lambda_0>0$ such that $R(\lambda_0)>0$, then the lower bound \eqref{eq:lower} gives
	$\liminf_{\lambda\to\infty}R(\lambda)=\infty$, so the first equivalence in \eqref{eq:R-equivalence} is true.
	
	We can argue similarly for the second equivalence. Certainly
	$\limsup_{\lambda\to 0+} R(\lambda)\le 0$ implies that there exists $\lambda>0$ such that $R(\lambda)<\infty$.  Conversely, if $R(\lambda_0) < \infty$ for some $\lambda_0$, then \eqref{eq:upper}  yields
	$\limsup_{\lambda\to 0+}R(\lambda)\le 0$. This finishes \eqref{eq:R-equivalence}.
	\smallskip
	
	\ref{cor:properties-of-R-2} Note that if $\lambda_0\in A_0$, then \eqref{eq:properties-of-R-2} yields  $\lambda\in A_0$ for all $\lambda > \lambda_0$, so
	\begin{align*}
		A_0 = \bigcup_{ R(\lambda)> 0 } [\lambda,\infty)
	\end{align*}
	is either empty or an unbounded interval. Moreover, the first implication of \eqref{eq:properties-of-R-1} shows that $R$ is increasing on $A_0$ and the first implication of \eqref{eq:properties-of-R-2} shows that it is strictly increasing where it is finite.
\end{proof}

Now we come to our definition of the threshold.
\begin{definition}[Threshold]\label{def:threshold}
		For $\dav\ge 0$ we let
	\begin{align*}
		\lambda_{\mathrm{cr}}\coloneq\lambda_{\mathrm{cr}}(\dav)
			\coloneq
				\inf \{ \lambda>0:\, R(\lambda) > \frac{\dav}{2} \}.
	\end{align*}
\end{definition}

For the properties of the threshold, we note

\begin{proposition}[Properties of the threshold]\label{prop:threshold}
	Assume that $V$ obeys \ref{ass:A2}. \\
	\itemthm\label{prop:threshold-1} The map $\dav\mapsto \lambda_{\mathrm{cr}} (\dav)$ is increasing on $[0, \infty)$ and $0 \le \lambda_{\mathrm{cr}}(\dav) \le \infty$ for every $\dav \ge 0$. \\
	\itemthm\label{prop:threshold-2} If $\dav \ge 0$ then
	$R(\lambda)>\frac{\dav}{2}$ for all
		$\lambda>\lambda_{\mathrm{cr}}(\dav)$ and
 		$R(\lambda) \le \frac{\dav}{2}$ for all $0<\lambda<\lambda_{\mathrm{cr}}(\dav)$. \\
 Furthermore, if $\dav > 0$ then
		$R(\lambda) < \frac{\dav}{2}$ for all $0<\lambda<\lambda_{\mathrm{cr}}(\dav)$. \\
	\itemthm\label{prop:threshold-3} We have the equivalences
		\begin{align}\label{eq:threshold-equivalence-1}
		\begin{array}{lll}
		  \lambda_{\mathrm{cr}}(\dav) <\infty \text{ for all }\dav\ge 0
		   & \Leftrightarrow &
		  \lambda_{\mathrm{cr}}(\dav) <\infty \text{ for some }\dav\ge 0 \\
		  ~
		  	& \Leftrightarrow &
			 R(\lambda)>0 \text{ for some }\lambda>0 	\\
		~  & \Leftrightarrow &
			 \lim_{\lambda\to\infty} R(\lambda) =\infty
		\end{array}		\, ,
		\end{align}
		and
		\begin{align}\label{eq:threshold-equivalence-2}
		\begin{array}{lll}
		  \lambda_{\mathrm{cr}}(\dav) >0 \text{ for all }\dav> 0
		   & \Leftrightarrow &
		  \lambda_{\mathrm{cr}}(\dav) >0 \text{ for some }\dav\ge  0 \\
		  ~
		  	& \Leftrightarrow &
			 \limsup_{\lambda\to 0+}R(\lambda)\le 0
		\end{array}		\, .
		\end{align}
		For zero average diffraction we have
		\begin{align}\label{eq:threshold-equivalence-3}
		\begin{array}{lll}
		  \lambda_{\mathrm{cr}}(0) =0
		  	& \Leftrightarrow &
			 R(\lambda)>0 \text{ for all }\lambda> 0 \\
		  \lambda_{\mathrm{cr}}(0) >0
		  	& \Leftrightarrow &
			 R(\lambda)\le 0 \text{ for some }\lambda> 0
		\end{array}		\, .
		\end{align}
\end{proposition}
\begin{remark}\label{rem:finite-threshold}
	A moments reflection shows that $R(\lambda)>0$ for some $\lambda>0$ if and only if
	there exists  $f\in l^2(\Z)$ with $\|f\|_2^2=\lambda$ and $N(f)>0$. So by \eqref{eq:threshold-equivalence-1}
	one sees  that the critical threshold $\lambda_{\mathrm{cr}}(\dav)$ is finite for
	all $\dav\ge 0$ if and only if $N(f)>0$ for some $f\in l^2(\Z)$.
\end{remark}

Before we prove the proposition we state and prove a corollary, which gives quantitive bounds on the threshold. We do not need these bounds in the following, but the proof is easy and the bounds are very natural, as the example of a pure power nonlinearity shows.
\begin{corollary}[Quantitative bounds on $\lambda_{\mathrm{cr}}$]
	\label{cor:lambda-crit-bounds}
	Assume that $V$ obeys \ref{ass:A2}.
 	If there exist $\lambda_0$ and $0<R_0 <\infty$ such that $R_0\ge  R(\lambda_0)$  then we have the lower bound
	    \begin{align}\label{eq:lambda-crit-lower}
	    	& \lambda_0 \left( \min\Big(\frac{\dav}{2R_0},1\Big) \right)^{\frac{2}{\gamma_0-2}}
	    		\le \lambda_{\mathrm{cr}}(\dav)
	    	\quad \text{ for all } \dav >0
	    \end{align}
	    and if there exist $\lambda_0$ and $0<R_0 <\infty$ such that $R_0\le  R(\lambda_0)$  then we have the upper bound
	    \begin{align}\label{eq:lambda-crit-upper}
	    	&  \lambda_{\mathrm{cr}}(\dav) \le
	    	\lambda_0  \left( \max\Big(\frac{\dav}{2R_0},1\Big) \right)^{\frac{2}{\gamma_0-2}}
	    	\quad \text{ for all } \dav> 0 .
	    \end{align}	
\end{corollary}

\begin{remark}
  If $V(a)=c a^\gamma$ is a pure power law for some $\gamma>2$ and $c>0$, then by Remark \ref{rem:pure-power-law} we have  $R(\lambda)= R_0\lambda^{(\gamma-2)/2}$ for some $0<R_0\le \infty$. In this case one can easily calculates
  \begin{align*}
  	\lambda_{\mathrm{cr}}(\dav) = \left(\frac{\dav}{2R_0}\right)^{\frac{2}{\gamma-2}}
  \end{align*}
  and with this example in mind one sees that the bounds of Corollary \ref{cor:lambda-crit-bounds} and the claims of Proposition \ref{prop:threshold} are very natural.
\end{remark}
\begin{proof}
  Since the proofs of \eqref{eq:lambda-crit-lower} and \eqref{eq:lambda-crit-upper} are very analogous, we give only the proof of \eqref{eq:lambda-crit-lower}.
 Assume that there exist $\lambda_0$ and $0<R_0<\infty $ with
 $ R(\lambda_0)\le R_0$.
 By \eqref{eq:properties-of-R-1} with $a=R_0$ we see that $R(\lambda ) \le R_0$ for all
 $0<\lambda\le \lambda_0$, so  $(\lambda_0,\infty)\supset \{\lambda>0:\, R(\lambda)>R_0\}$.  This shows
 $\lambda_{\mathrm{cr}}(2R_0)\ge \lambda_0$ for $\dav=2R_0$ and using the  monotonicity in
 $\dav$ from Proposition \ref{prop:threshold-1} we also have  $\lambda_{\mathrm{cr}}(\dav) \ge \lambda_{\mathrm{cr}}(2R_0)\ge \lambda_0$ for all $\dav\ge 2R_0$.

 Now let $0<\dav< 2R_0$ and write $\lambda_{\mathrm{cr}}$ for $\lambda_{\mathrm{cr}}(\dav)$.
 Either we have $\lambda_{\mathrm{cr}}\ge\lambda_0$, then \eqref{eq:lambda-crit-lower}
 trivially holds, or $0\le\lambda_{cr}<\lambda_0$.
 In the last case set $\lambda_2=\lambda_0$ and
 $0<\lambda_1=\lambda_{\mathrm{cr}}+\delta <\lambda_0$ for all small enough $\delta>0$,
 then Proposition \ref{prop:threshold-2} shows $R(\lambda_{cr}+\delta)>\tfrac{\dav}{2}$
 which together with \eqref{eq:R-lower} gives
 \begin{align*}
 	R_0 \ge R(\lambda_0)\ge \left(\frac{\lambda_0}{\lambda_{\mathrm{cr}}+\delta}\right)^{\frac{\gamma_0-2}{2}} R(\lambda_{\mathrm{cr}}+\delta)
 	 > \left(\frac{\lambda_0}{\lambda_{\mathrm{cr}}+\delta}\right)^{\frac{\gamma_0-2}{2}} \frac{\dav}{2}
 \end{align*}
 for all small enough $\delta>0$. This proves the lower bound \eqref{eq:lambda-crit-lower} and similarly one proves the upper bound \eqref{eq:lambda-crit-upper}.
\end{proof}

\begin{proof}[Proof of Proposition \ref{prop:threshold}] First some preparations.
For $\dav\ge 0$, define the set
\begin{align*}
	A_{\dav}\coloneq \{ \lambda>0:\, R(\lambda)>\tfrac{\dav}{2} \}
\end{align*}
 so that    $\lambda_{\mathrm{cr}}(\dav)=\inf A_{\dav}$. Arguing as in the proof of Corollary \ref{cor:properties-of-R}\ref{cor:properties-of-R-2} we see that
 if $\lambda_0\in A_{\dav}$, then $\lambda\in A_{\dav}$ for all $\lambda\ge \lambda_0$. Hence
 \begin{align}\label{eq:A-dav-interval}
 	A_{\dav} = \bigcup_{R(\lambda_0)>\frac{\dav}{2}} [\lambda_0,\infty)
 \end{align}
 and so the set  $A_{\dav}$ is either empty,
  or an interval that is bounded from below by $0$ but unbounded from above,
  and they are nested,  in the sense
  that if $0\le \dav_{,1}\le \dav_{,2}$ then  $A_{\dav_{,2}}\subset A_{\dav_{,1}}$.
 In addition, $A_{\dav}$ is empty if and only if the threshold $\lambda_{\mathrm{cr}}(\dav)=\infty$, $A_{\dav}$ is not empty if and only if $0\le \lambda_{\mathrm{cr}}(\dav)<\infty$, and $A_{\dav}=(0,\infty)$ if and only if $\lambda_{\mathrm{cr}}(\dav) =0$.
  \smallskip

\ref{prop:threshold-1} Since, by the above $A_{\dav_{,2}}\subset A_{\dav_{,1}}$ for $0\le \dav_{,1}\le \dav_{,2}$, we immediately see $\lambda_{\mathrm{cr}}(\dav_{,1})\le \lambda_{\mathrm{cr}}(\dav_{,2})$.
\smallskip

 \ref{prop:threshold-2} First let $\dav \ge 0$. Certainly, if $R(\lambda)>\tfrac{\dav}{2}$ then $\lambda \ge \lambda_{\mathrm{cr}}(\dav)$, so, by contrapositive, if $0<\lambda<\lambda_{\mathrm{cr}}(\dav)$ then $R(\lambda)\le \tfrac{\dav}{2}$. Moreover, if $\lambda>\lambda_{\mathrm{cr}}(\dav)$, then \eqref{eq:A-dav-interval} shows $\lambda\in A_\dav$,
 so $R(\lambda)>\tfrac{\dav}{2}$. This proves the first claim.

 Now let $\dav>0$ and $0<\lambda<\lambda_{\mathrm{cr}}(\dav)$. Then, by the first claim, we already know $R(\lambda)\le \tfrac{\dav}{2}$. Moreover, if we suppose that $R(\lambda) = \tfrac{\dav}{2}$ then \eqref{eq:properties-of-R-2} yields a contradiction to the fact that $\lambda_{\mathrm{cr}}(\dav)$ is a lower bound for $A_\dav$. Thus $R(\lambda) < \tfrac{\dav}{2}$ and this proves the second claim.
 \smallskip

\ref{prop:threshold-3}
  We certainly have that
  $\lambda_{\mathrm{cr}}(\dav) <\infty$  for all $\dav\ge 0  $
  implies $ \lambda_{\mathrm{cr}}(\dav) <\infty$ for some $\dav\ge 0$. Next, if there exists $\dav \ge 0$ such that $\lambda_{\mathrm{cr}}(\dav) <\infty $ then $A_\dav$ is not empty, hence  $R(\lambda)>\tfrac{\dav}{2}\ge 0$ for some $\lambda>0$.
  Thirdly, if there exists $\lambda>0$ such that $R(\lambda)> 0$ then, \eqref{eq:R-equivalence} gives
  $\lim_{\lambda\to\infty}R(\lambda)=\infty$. Lastly, if $\lim_{\lambda\to\infty}R(\lambda)=\infty$, then,  for every $\dav\ge 0$, we have $R(\lambda)>\tfrac{\dav}{2}$ for all large enough $\lambda$ which shows that $A_\dav$ is not empty, hence
  $\lambda_{\mathrm{cr}}(\dav) <\infty $ for all $\dav \ge 0$. This finishes the proof of \eqref{eq:threshold-equivalence-1}.
  \smallskip

  For the proof of \eqref{eq:threshold-equivalence-2} we note that  certainly $\lambda_{\mathrm{cr}}(\dav) >0$ for all $\dav> 0$ implies
	 $ \lambda_{\mathrm{cr}}(\dav) >0$ for some $\dav\ge  0$.
    Next, if there exists $\dav \ge 0$ such that $\lambda_{\mathrm{cr}}(\dav) > 0 $ then $R(\lambda) \le \tfrac{\dav}{2} < \infty$ for some $\lambda>0$
    and so, by \eqref{eq:R-equivalence}, $\limsup_{\lambda\to 0+}R(\lambda)\le 0$.
    Lastly, if $\limsup_{\lambda\to 0+}R(\lambda)\le 0$ then, for every $\dav>0$, we have $R(\lambda)\le \f{\dav}{2}$ for all small enough
	 $\lambda>0$, and hence $\lambda_{\mathrm{cr}}(\dav)>0$ for every $\dav>0$.
	 This finishes the proof of \eqref{eq:threshold-equivalence-2}.
  \smallskip

  For the proof of \eqref{eq:threshold-equivalence-3} recall that $\lambda_{\mathrm{cr}}(0) =0$ if and only if $A_0=(0,\infty)$, which is the case if and only if $R(\lambda)>0$ for any $\lambda>0$. Moreover, $\lambda_{\mathrm{cr}}(0) >0$ if and only if $A_0\neq (0,\infty)$, that is, if and only if there exists $\lambda>0$ with $R(\lambda)\le 0$.
\end{proof}

Now we can give the
\begin{proof}[Proof of Theorem \ref{thm:threshold-phenomena}]
  \ref{thm:threshold-phenomena-1}$+$\ref{thm:threshold-phenomena-2}
   Fix $\dav\ge 0$. It follows from Lemma \ref{lem:E-boundedness} and Proposition \ref{prop:strict-subadditivity} that $-\infty < E^{\dav}_\lambda\le 0$, for every $\dav \ge 0$ and $\lambda >0$, and the map $\lambda\mapsto E^{\dav}_\lambda$ is decreasing on $(0, \infty)$.
  Proposition \ref{prop:threshold} gives the existence of a critical threshold $0\le \lambda_{\mathrm{cr}}=\lambda_{\mathrm{cr}}(\dav)\le \infty$ such that if $\lambda>\lambda_{\mathrm{cr}}$, we have $R(\lambda)>\tfrac{\dav}{2}$. In this case, Lemma \ref{lem:threshold-basic} shows that $E^{\dav}_\lambda<0$. Moreover, if $0<\lambda<\lambda_{\mathrm{cr}}(\dav)$, then Proposition \ref{prop:threshold} and Lemma \ref{lem:threshold-basic} also show that $E^{\dav}_\lambda\ge 0$ and so $E^{\dav}_\lambda=0$ for all $0<\lambda<\lambda_{\mathrm{cr}}(\dav)$. This proves the first part of Theorem \ref{thm:threshold-phenomena} and
  Theorem \ref{thm:existence} yields the claims of its second part.
  \smallskip

  \ref{thm:threshold-phenomena-3} Let $\dav>0$ and
  $0<\lambda<\lambda_{\mathrm{cr}}$. If  $f\in l^2(\Z)$ with $\|f\|_2^2=\lambda>0$ is a minimizer, then using
  \eqref{eq:energy-threshold} gives
  \begin{align}\label{eq:nonexistence}
  	0=E^{\dav}_\lambda = H(f) \ge \|D_+f\|^2 \left(\frac{\dav}{2} -R(\lambda)\right)  .
  \end{align}
  By Proposition \ref{prop:threshold}\ref{prop:threshold-2}, we have $R(\lambda)<\tfrac{\dav}{2}$, so the inequality \eqref{eq:nonexistence} implies
  $  \|D_+f\|_2^2 \le 0 $, that is, $\|D_+f\|_2^2 = 0 $.
  Since the kernel of $D_+$ is trivial this shows $f=0$, which contradicts $\|f\|_2^2=\lambda>0$, so no minimizers can exist in this case.
  \smallskip

  \ref{thm:threshold-phenomena-4} By Remark \ref{rem:finite-threshold}  we have
  $\lambda_{\mathrm{cr}}<\infty$
  if and only if there exist $f\in l^2(\Z)$ with $N(f)>0$.
  \smallskip

  \ref{thm:threshold-phenomena-5} We have to show that if $\gamma_1\ge 6$, then $\lambda_{\mathrm{cr}}(\dav)>0$ for all $\dav>0$. For this we use the inequality
  \begin{align}\label{eq:weinstein-type-ineq}
  	\|f\|_{\gamma}^{\gamma}\le \|f\|_2^{\gamma-2} \|D_+f\|_2^2	
  \end{align}
  which holds for all $f\in l^2(\Z)$ and all $\gamma\ge 6$. Assuming
  \eqref{eq:weinstein-type-ineq} for the moment, one can argue as follows: From \eqref{eq:weinstein-type-ineq} we have,
  under assumptions \ref{ass:A1} with $\gamma_1\ge 6$,
  \begin{align*}
  	|N(f)| &\le \int_\R \|V(|T_r f|)\|_1\,\mu(dr)
  	 \lesssim \int_\R \left( \|T_rf\|_2^{\gamma_1-2} + \|T_rf\|_2^{\gamma_2-2}\right)  \|D_+T_rf\|_2^2\, \mu(dr) \\
  	 &\lesssim \left( \|f\|_2^{\gamma_1-2} + \|f\|_2^{\gamma_2-2}\right)
  	 	\|D_+f\|_2^2\, .
  \end{align*}
  So
  \begin{align*}
  	|R(\lambda)| = \sup_{\|f\|_2^2=\lambda} \frac{|N(f)|}{\|D_+ f\|_2^2}
  	 \lesssim \lambda^{\frac{\gamma_1-2}{2}} + \lambda^{\frac{\gamma_2-2}{2}} <\infty
  \end{align*}
  which directly shows $\lim_{\lambda\to 0+} R(\lambda)= 0 $ and then \eqref{eq:threshold-equivalence-2} gives $\lambda_{\mathrm{cr}}(\dav)>0$
  for all $\dav>0$.

It remains to prove \eqref{eq:weinstein-type-ineq}. For this we recall
\begin{align}\label{eq:linftybound}
	\|f\|_\infty^2 \le \|f\|_2 \|D_+f\|_2
\end{align}
for any $f\in l^2(\Z)$. Indeed, to see this let  $x\in\Z$,  then
\begin{align*}
  |f(x)|^2 = \sum_{l\le x} (| f(l)|^2-| f(l-1)|^2)	
  	=
  		\sum_{l\le x} (| f(l)|+| f(l-1)|)(| f(l)|-| f(l-1)|)
\end{align*}
and similarly,
 \begin{align*}
  |f(x)|^2 = - \sum_{l> x} (| f(l)|^2-| f(l-1)|^2)	
  	=
  		-\sum_{l> x} (| f(l)|+| f(l-1)|)(| f(l)|-| f(l-1)|) \, .
\end{align*}
Adding this two inequalities and using Cauchy-Schwarz gives
  \begin{align*}
  	|f(x)|^2 &\le \frac{1}{2}\sum_{l\in\Z} (| f(l)|+| f(l-1)|)|| f(l)|-| f(l-1)|| \\
  	&\le \frac{1}{2}\sum_{l\in\Z } (| f(l)|+| f(l-1)|)| f(l) -  f(l-1)| \\
  	&\le \left(\sum_{l\in\Z } | f(l)|^2\right)^{1/2}
  			\left(\sum_{l\in\Z } | f(l) -  f(l-1)|^2\right)^{1/2}
  		=  \|f\|_2 \|D_+f\|_2
  \end{align*}
  which, since it holds for all $x\in\Z$, gives \eqref{eq:linftybound}. From \eqref{eq:linftybound}, we see that for $\gamma> 4$,
  \begin{align*}
  	\|f\|_{\gamma}^{\gamma}
  	\le \|f\|_{\gamma-4}^{\gamma-4} \|f\|_\infty^4
  	\le 	\|f\|_{\gamma-4}^{\gamma-4} \|f\|_2^2 \|D_+f\|_2^2
  \end{align*}
  from which we immediately get   \eqref{eq:weinstein-type-ineq}
  as soon as $\gamma\ge 6$,
  since in this case $\|f\|_{\gamma-4}\le \|f\|_2$, by the monotonicity properties of $l^p(\Z)$ norms.
\end{proof}

Finally, we come to the
\begin{proof}[Proof of Theorem \ref{thm:threshold-intro}]
  Assume that $V$ obeys \ref{ass:A1} through \ref{ass:A3}.
  Except for the finiteness of the threshold $\lambda_{\mathrm{cr}}$, all claims of
  Theorem \ref{thm:threshold-intro} follow immediately from
  Theorem \ref{thm:threshold-phenomena}. In addition,  Theorem
  \ref{thm:threshold-phenomena} shows that the threshold is finite if and only if
  there exists $f\in l^2(\Z)$ with
  \begin{align*}
     N(f) = \int_\R \sum_{x\in\Z} V(|T_rf(x)|)\, \mu(dr) >0.
  \end{align*}
  Lemma \ref{lem:V-lower-bound} shows that under assumptions \ref{ass:A2} and \ref{ass:A3} we have $\lim_{a\to\infty} V(a)=\infty$ and thus $N(f)$ should be large, in particular positive, if $f$ is `large'.   Since $V$ can be negative and due to the nonlocal nature of $N$, this is not obvious, however.
  Moreover, in the discrete setting there are no nice initial conditions for which one can calculate the time evolution
  $T_r f$ and then also $N(f)$ explicitly, so the proof turns out to be a bit
  technical. It is deferred to Lemma \ref{lem:N-positivity}, where we show that
  under conditions \ref{ass:A1}, \ref{ass:A2}, and \ref{ass:A3}, there exists a simple family $f_l\in l^2(\Z)$ with $\lim_{l\to\infty} N(f_l)=\infty$.

  If, in addition, we assume \ref{ass:A4}, then Lemma \ref{lem:negativity} shows
  $E^{\dav}_\lambda<0$ for all $\lambda>0$ and all $\dav\ge 0$. So in this case we have
  $\lambda_{{cr}}(\dav)=0$ for all $\dav\ge 0$.
\end{proof}

\section{Exponential decay for positive average diffraction}\label{sec:exponential decay}
In this section, we will give the proof of Theorem \ref{thm:exponential decay}.  Our strategy for its proof is to first prove some exponential decay and then to boost this to get it up to what the physical heuristic argument in the introduction predicts.

\subsection{Some exponential decay}\label{subsec:some exponential decay}
The main goal in this section is to prove
\begin{proposition}\label{prop:some exponential decay}
	Assume that $V$ obeys the assumption \ref{ass:A1}. Then any solution $\varphi$ of \eqref{eq:GT} with $\omega<0$ decays exponentially, i.e., there exists $\nu>0$ such that
	\begin{align}\label{eq:some exponential decay}
		 x\mapsto e^{\nu|x|}\varphi(x)\in l^2(\Z) .
	\end{align}
\end{proposition}

To prepare for its proof, define the cutoff function $\chi(s):= \min(1,(|s|-1)_+)$, i.e., $\chi(s)=0$ if $|s|\le 1$, $\chi(s)=1$ for $|s|\ge 2$, and linearly interpolating in between. Furthermore, define the functions of $s$ on $\R$
\begin{align}
	\chi_\tau(s):= \chi(s/\tau), \quad
	F_{\nu,\veps}(s):=  \frac{\nu|s|}{1+\veps|s|}, \quad
	\text{ and}\quad
	\xi_{\nu,\veps,\tau} := e^{F_{\nu,\veps}}\chi_\tau ,
\end{align}
for any $\tau>0,\ \nu,\ \veps \ge 0$.

It is clearly enough to prove that $\xi_{\nu,0,\tau}\varphi\in l^2(\Z)$ for some $\nu>0$ and $\tau\ge 1$ for any solution of \eqref{eq:GT} with $\omega<0$. Choosing $g=\xi^2\varphi$ in \eqref{eq:GT-weak} and using Lemma \ref{lem:differentiability} on has
\begin{align}\label{eq:exp decay basic 0}
	\omega \|\xi\varphi\|_2^2
		&=\re( \omega\la\xi^2\varphi,\varphi\ra )
			= DH(\varphi)[\xi^2 \varphi]
			=
			\dav\re( \la \xi^2\varphi, -\Delta\varphi\ra )
			- \re( DN(\varphi)[\xi^2\varphi] ) \nonumber\\
		&\ge
			- \frac{\dav}{2}\la \varphi, (|D_+\xi|^2 + |D_-\xi|^2)\vphi \ra
			-\re DN(\varphi)[\xi^2\varphi]
\end{align}
where we used the lower bound \eqref{eq:IMS-lower-1} from Lemma \ref{lem:IMS} and $\la \xi\varphi,-\Delta( \xi\varphi) \ra\ge 0$.
Thus for $\omega<0$ we have
\begin{align}\label{eq:exp decay basic}
	\|\xi\varphi\|_2^2
		\le
			\frac{\dav}{2|\omega|} \la\varphi, (|D_+\xi|^2 + |D_-\xi|^2)\vphi \ra
			+ \frac{1}{|\omega|}\re DN(\varphi)[\xi^2\varphi]
\end{align}
which is our starting point for the proof of Proposition \ref{prop:some exponential decay}.
To use it, the following two Lemmata are helpful.
\begin{lemma}\label{lem:IMS error non sharp}
	For all $\nu,\ \veps\ge 0$ and $\tau>0$ we have
	\begin{align}\label{eq:IMS error non sharp 1}
		|D_+\xi_{\nu,\veps,\tau}|^2 + |D_-\xi_{\nu,\veps,\tau}|^2
		\le
			  \frac{4e^{2\nu \tau}}{\tau^2}
			 + 4 (e^\nu -1)^2\xi_{\nu,\veps,\tau}^2.
	\end{align}
	In particular, for the choice $\nu:=\tau^{-1}$ one has
	\begin{align} \label{eq:IMS error non sharp 2}
		\la \varphi, (|D_+\xi_{\nu,\veps,\tau}|^2 + |D_-\xi_{\nu,\veps,\tau}|^2 )\varphi\ra
		\le
			 \frac{4e^2}{\tau^2} \|\varphi\|_2^2
			 +4 (e^{\tau^{-1}}-1)^2 \|\xi_{\nu,\veps,\tau}\varphi\|_2^2
	\end{align}
\end{lemma}
\begin{proof}
Clearly, \eqref{eq:IMS error non sharp 2} follows from \eqref{eq:IMS error non sharp 1}, so it is enough to prove the first claim.
	Denoting $\xi=\xi_{\nu,\veps,\tau}$ and $F=F_{\nu,\veps}$, we have
	\begin{align*}
		D_+\xi(x)
				= e^{F(x+1)} D_+\chi_\tau(x) + D_+e^F(x)\chi_\tau(x)
	\end{align*}
Thus for $x\ge 0$
\begin{equation}
	\begin{split} \label{eq:IMS error 1}
	|D_+\xi(x)|^2
	 	&\le 2 e^{2F(x+1)}|D_+\chi_\tau(x)|^2 + 2 |e^{F(x+1)}-e^{F(x)}|^2\chi_\tau(x)^2 \\
	 	&\le 2 e^{2F(x+1)}\frac{1}{\tau^2}\id_{[\tau,2\tau-1]}(x)
	 		+ 2 |e^{F(x+1)-F(x)}-1|^2\xi(x)^2 \\
	 	&\le \frac{2e^{2F(2\tau)}}{\tau^2}
	 		+ 2 (e^{\nu}-1)^2\xi(x)^2
	\end{split}
\end{equation}
where we used $D_+\chi_\tau(x)\le \frac{1}{\tau}\id_{[\tau,2\tau-1]}(x)$ for $x\ge 0$ in the second inequality, the monotonicity of $F$ and the fact that $F$ obeys the triangle inequality, that is, $F(s_1+s_2)\le F(s_1)+F(s_2)$ for all $s_1,s_2\in\R$  and hence, by the reverse triangle inequality, also $F(x+1)-F(x)\le |F(x+1)-F(x)|\le F(1)\le \nu$, in the third inequality.

Now let $x\ge 1$. Then
\begin{align*}
	D_-\xi(x) = D_-e^F(x)\chi_\tau(x) + e^{F(x-1)}D_-\chi_\tau(x) ,
\end{align*}
so arguing similarly as above,
\begin{equation}\label{eq:IMS error 2}
	\begin{split}
	|D_-\xi(x)|^2
		&\le 2|D_-e^F(x)\chi_\tau(x)|^2 + 2|e^{F(x-1)}D_-\chi_\tau(x)|^2 \\
		&\le 2|1-e^{-|F(x)-F(x-1)|}|^2\xi(x)^2
			+ 2e^{2F(x-1)} \frac{1}{\tau^2}\id_{[\tau+1,2\tau]}(x) \\
		&\le 2(e^\nu-1)^2\xi(x)^2 + \frac{2e^{2F(2\tau)}}{\tau^2}
	\end{split}
\end{equation}
for all $x\ge 1$. Using that $\xi$ is symmetric, and hence
$ D_+\xi(-x)= -D_-\xi(x) $ holds, the bound \eqref{eq:IMS error 1} shows
\begin{align*}
	|D_-\xi(x)|^2 \le \frac{2e^{2F(2\tau)}}{\tau^2}
	 		+ 2 (e^{\nu}-1)^2\xi(x)^2
\end{align*}
for $x\le 0$ and the bound \eqref{eq:IMS error 2} shows
\begin{align*}
	|D_+\xi(x)|^2 \le \frac{2e^{2F(2\tau)}}{\tau^2}
	 		+ 2 (e^{\nu}-1)^2\xi(x)^2
\end{align*}
for $x\le -1$. This proves \eqref{eq:IMS error non sharp 1}.
\end{proof}
From \eqref{eq:exp decay basic} it is clear that we also have to control $DN(\varphi)[\xi^2\varphi]$.
From Lemma \ref{lem:differentiability}, we get a simple bound
\begin{align}\label{eq:DNbound0}
	|DN(f_2)[f_1]|\le \int_\R \|T_rf_1 V'(|T_rf_2|)\|_1 \,\mu(dr)
\end{align}
since $|\la h_1,h_2\ra|\le \|h_1h_2\|_1$.
From our assumptions on $V$ one sees
 \begin{align*}
 	|V'(a)|\lesssim a^{\gamma_1-1} + a^{\gamma_2-1}
 \end{align*}
 for all $a\in\R_+$ and therefore
 \begin{align}\label{eq:DNbound}
 	|DN(f_2)[f_1]| \lesssim L^{\gamma_1}_\mu(f_1,f_2) +  L^{\gamma_2}_\mu(f_1,f_2)
 \end{align}
 where we used
 \begin{definition} \label{def:L} For $\gamma\ge 2$ and $\mu$ a finite measure on $\R$ with compact support, let
 \begin{align}\label{eq:L}	
 	L^{\gamma}_\mu(f_1,f_2):= \int_\R \|T_rf_1 |T_rf_2|^{\gamma-1}\|_1\, \mu(dr)
 \end{align}
 \end{definition}
 A simple bound for the derivative of the nonlocal nonlinearity is given by
 \begin{lemma}\label{lem:DNsimple}
 	Assume that $\mu$ is a finite measure with compact support and $V$ obeys assumption \ref{ass:A1}. Then for any $f_1,f_2\in l^2(\Z)$
 	\begin{align*}
 		|DN(f_2)[f_1]| \lesssim \|f_1\|_2\bigl(\|f_2\|_2^{\gamma_1-1}+ \|f_2\|_2^{\gamma_2-1}\bigr)
 	\end{align*}
 	where the implicit constant depends only on $\mu(\R)$ and $\supp\mu$ (and the constants in \eqref{eq:DNbound}).
 \end{lemma}
 \begin{proof}
 	This follows simply from \eqref{eq:DNbound}
 	\begin{align*}
 		\|T_rf_1 |T_rf_2|^{\gamma-1}\|_1
 		\le
 			\|T_rf_1 T_rf_2\|_1 \|T_rf_2\|_\infty^{\gamma-2}
 	\end{align*}
 	the bound $\|T_rf_2\|_\infty \le \|T_rf_2\|_2= \|f_2\|_2$, by unitarity of $T_r$, the bound \eqref{eq:strong bilinear} and the assumption that $\mu$ is a finite measure with compact support.
 \end{proof}
 We will need a version of Lemma \ref{lem:DNsimple} which is `\emph{exponentially twisted}',
\begin{lemma}\label{lem:Ltwisted} For $x\in\R$. Then for all $\gamma\ge 2$, all finite measures $\mu$ with compact support, and all $0<\alpha<\frac{1}{2}$
	\begin{align}
		L^\gamma_\mu (e^{F_{\nu,\veps}}h_1, e^{-F_{\nu,\veps}}h_2)
		\lesssim
			\min(1,s^{-\alpha s})\|h_1\|_2 \|h_2\|_2 \|e^{-F_{\nu,\veps}}h_2\|_2^{\gamma-2}	
	\end{align}
	where $s:=\dist(\supp h_1,\supp h_2)\ge 0$ and  the implicit constant is independent of
	$\veps>0$ and depends increasingly on $\nu\ge0$, $\mu(\R)$, and the support of $\mu$ and $\alpha$.
\end{lemma}
\begin{remark}
In the equation above, we set $0^{-\alpha 0}:= \lim_{s\to 0} s^{-\alpha s} =1$, when $s=0$.
\end{remark}

\begin{proof}
 Let $B>0$ such that $\supp\mu\subset[-B,B]$. Then
 \begin{align*}
 	L^\gamma_\mu (e^{F_{\nu,\veps}}h_1, e^{-F_{\nu,\veps}}h_2)
 	\le
 		\mu(\R) \sup_{|r|\le B}
 				\|T_r(e^{F_{\nu,\veps}}h_1) |T_r(e^{-F_{\nu,\veps}}h_2)|^{\gamma-1}\|_1
 \end{align*}	
 Fix $r\in\R $, then
 \begin{align*}
 	\|T_r(e^{F_{\nu,\veps}}h_1) |T_r(e^{-F_{\nu,\veps}}h_2)|^{\gamma-1}\|_1
 	\le
 		\|T_r(e^{F_{\nu,\veps}}h_1) T_r(e^{-F_{\nu,\veps}}h_2)\|_1
 		\|T_r(e^{-F_{\nu,\veps}}h_2)\|_\infty^{\gamma-2}
 \end{align*}
 The first factor is bounded by \eqref{eq:twisted strong bilinear} and
 for the second factor we simply note
 \begin{align*}
 	\|T_r(e^{-F_{\nu,\veps}}h_2)\|_\infty
 	\le
 		\|T_r(e^{-F_{\nu,\veps}}h_2)\|_2 = \|e^{-F_{\nu,\veps}}h_2\|_2
 \end{align*}
 since $T_r$ is unitary on $l^2(\Z)$. Since $\frac{2(4Be^\nu)^{\lceil\frac{s}{2}\rceil}}{\lceil\frac{s}{2}\rceil!}\lesssim s^{-\alpha s}$ for any fixed $B,\nu>0$, $0<\alpha<\frac{1}{2}$, and all bounded $\nu $, this finishes the proof.
\end{proof}
A useful consequence of this is
\begin{corollary}\label{cor:DNbound}
Assume that $V$ obeys assumption \ref{ass:A1}. Then for the choice $\nu=\tau^{-1}$
	\begin{align*}
		|DN(\varphi)[\xi_{\nu,\veps,\tau}^2\varphi]| \lesssim \oh(1) \left(\|\xi_{\nu,\veps,\tau}\varphi\|_2^2 + \|\xi_{\nu,\veps,\tau}\varphi\|_2\right)
	\end{align*}
	where the implicit constant depends only on $\gamma_1,\gamma_2$, $\mu(\R)$, the support of $\mu$,
	$\|\varphi\|_2$ and $\oh(1)$ denotes a term which, for fixed $\varphi\in l^2(\Z)$, goes to zero uniformly in $\veps>0$  as $\tau\to\infty$.
\end{corollary}
\begin{proof}
	Set $\xi=\xi_{\nu,\veps,\tau}$ and $F=F_{\nu,\veps}$. Because of \eqref{eq:DNbound}, we need to control
	$L^\gamma_\mu(\xi^2\varphi,\varphi)$ for $\gamma=\gamma_1$ and $\gamma=\gamma_2$.
	Let $\varphi_\tau:= \chi_\tau\varphi$ and $h_\tau:= e^{F} \varphi_\tau$ and
	split $h:= e^{F}\varphi$ into $h_\tau$ and
	$h_{\le\tau}:= (1-\chi_\tau)h$. Then $h=h_\tau + h_{\le\tau}$ and since
	$|a+b|^{\gamma-1}\lesssim |a|^{\gamma-1}+|b|^{\gamma-1}$, we have
	\begin{align*}
		L^\gamma_\mu(\xi^2\varphi,\varphi)
		&= L^\gamma_\mu(e^{F}\chi_\tau h_\tau,e^{-F}h)  \\
		&\lesssim L^\gamma_\mu(e^{F}\chi_\tau h_\tau,e^{-F}h_\tau)
			+ L^\gamma_\mu(e^{F}\chi_\tau h_\tau,e^{-F}h_{\le\tau})
	\end{align*}
	Lemma \ref{lem:Ltwisted} yields
	\begin{align*}
		L^\gamma_\mu(e^{F}\chi_\tau h_\tau,e^{-F}h_\tau)
		&\lesssim
			\|h_\tau\|_2^2 \|\varphi_\tau\|_2^{\gamma-2}	
	\end{align*}
	since $\|\chi_\tau h_\tau\|_2\le \| h_\tau\|_2$. Splitting $h_{\le\tau}=h_{\ll\tau}+h_{\sim\tau}$, where $h_{\ll\tau}:= \id_{[-\tau/2,\tau/2]}(x)h_\tau$ and $h_{\sim\tau}:= h_{\tau}-h_{\ll\tau}$  we also have
	\begin{align*}
		L^\gamma_\mu(e^{F}\chi_\tau h_\tau,e^{-F}h_{\le\tau})
		&\lesssim
			L^\gamma_\mu(e^{F}\chi_\tau h_\tau,e^{-F}h_{\ll\tau})
			+ L^\gamma_\mu(e^{F}\chi_\tau h_\tau,e^{-F}h_{\sim\tau}) \\
		&\le
			(\tau/2)^{-\tau/8}\|h_\tau\|_2 \| h_{\ll\tau} \|_2
		 	\|\varphi_{\ll\tau}\|_2^{\gamma-2}	
		 	+ \|h_\tau\|_2 \| h_{\sim\tau} \|_2
		 	\|\varphi_{\sim\tau}\|_2^{\gamma-2}	\\
		 &\le
		 	(\tau/2)^{-\tau/8} e^{1/2}\|h_\tau\|_2
		 	\|\varphi\|_2^{\gamma-1}	
		 	+ e^2 \|h_\tau\|_2
		 	\|\varphi_{\sim\tau }\|_2^{\gamma-1}
	\end{align*}
	because of Lemma \ref{lem:Ltwisted}, since $h_{\ll\tau}$ and $h_\tau$ have supports separated by at least $\tau/2$ and
	$\| h_{\ll\tau} \|_2\le e^{\nu \tau/2} \|\varphi_{\ll\tau}\|_2 \le e^{1/2} \|\varphi\|_2$
	and $\| h_{\sim\tau} \|_2 \le e^{2\nu\tau} \|\varphi_{\sim\tau}\|_2\le e^2\|\varphi_{\sim\tau}\|_2$. Together, the above bounds show
	\begin{align*}
		L^\gamma_\mu(\xi^2\varphi,\varphi)
			\lesssim (\|\varphi_\tau\|^{\gamma-2} + (\tau/2)^{-\tau/8}\|\varphi\|_2^{\gamma-1} + \|\varphi_{\sim\tau }\|_2^{\gamma-1})
				(\|h_\tau\|_2^2 + \|h_\tau\|_2 ) .
	\end{align*}
	Since, for fixed $\varphi\in l^2(\Z)$, the term $\|\varphi_\tau\|_2^{\gamma-2} + (\tau/2)^{-\tau/8}\|\varphi\|_2^{\gamma-1} + \|\varphi_{\sim\tau }\|_2^{\gamma-1}$
	goes to zero as $\tau\to\infty$, this finishes the proof of the corollary.
\end{proof}

Now we can give the
\begin{proof}[Proof of Proposition \ref{prop:some exponential decay}]
 Let  $\varphi$ be a solution of \eqref{eq:GT} with $\omega<0$ and $\dav>0$. Then with $\chi_\tau$, $F_{\nu,\veps}$, and $\xi_{\nu,\veps,\tau}$ as before together with the choice
 $\nu=\tau^{-1}$, the inequality \eqref{eq:exp decay basic} and Lemma
 	\ref{lem:IMS error non sharp}  and Corollary \ref{cor:DNbound} show
 \begin{align}\label{eq:good news 1}
 	\|\xi_{\nu,\veps,\tau}\varphi\|_2^2\le \oh_1(1)(\|\xi_{\nu,\veps,\tau}\varphi\|_2^2
 		+ \|\xi_{\nu,\veps,\tau}\varphi\|_2) + \oh_2(1)
 \end{align}
 where $\oh_1(1)$ and $\oh_2(1)$ denote terms which, for fixed $\varphi\in l^2(\Z)$ and $\nu=\tau^{-1}$, go to zero as $\tau\to\infty$ \emph{uniformly } in $\veps>0$.
 Choosing $\tau$ so large that $\oh_1(1)\le \frac{1}{2}$, the bound \eqref{eq:good news 1} gives
 \begin{align}\label{eq:good news 2}
 	 \|\xi_{\nu,\veps,\tau}\varphi\|_2^2 - \|\xi_{\nu,\veps,\tau}\varphi\|_2
 	\lesssim 1
 \end{align}
 as long as $\nu=\tau^{-1}$ and $\tau$ is large enough. Clearly, \eqref{eq:good news 2} shows that $\|\xi_{\nu,\veps,\tau}\varphi\|_2$ stays bounded as $\veps\to0$, so
 \begin{align*}
 	\|e^{F_{\nu,0}}\varphi_\tau\|_2 =\lim_{\veps\to 0 }\|e^{F_{\nu,\veps}}\varphi_\tau\|_2 <\infty
 \end{align*}
 as long as $\nu=\tau^{-1}$ and $\tau$ is large enough.
\end{proof}

\subsection{Boosting the decay rate}\label{subsec:boost}
Given Proposition \ref{prop:some exponential decay} we know that a solution $\varphi$ of \eqref{eq:GT} with $\omega<0$ and $\dav>0$ has some exponential decay, that is,
for some $\nu>0$ we have $e^{\nu|\cdot|}\varphi(\cdot)\in l^2(\Z)$.
The goal in this section is to boost this to prove the lower bound \eqref{eq:exponential decay rate} on the exponential decay rate from Theorem \ref{thm:exponential decay}. For this we need a refinement of \eqref{eq:exp decay basic} and of Lemma \ref{lem:IMS error non sharp}. Looking at the proof of \eqref{eq:exp decay basic}, we need to refine the error in the IMS localization formula. This is the context of
\begin{lemma}\label{lem:IMS error sharp}
	Let $F:\Z\to\R $ be bounded. Then for all $\varphi\in l^2(\Z)$
	\begin{align*}
		\re(\la e^{2F} \varphi, -\Delta \varphi \ra)
		 \ge
				- \la e^{F}\varphi, \big(\cosh(D_+F)+\cosh(D_-F)-2\big)e^{F}\varphi \ra
	\end{align*}
\end{lemma}
\begin{proof}
	Using the formula \eqref{eq:IMS} for $\xi=e^F$ one sees
	\begin{align*}
		\re(\la e^{2F} \varphi, -\Delta \varphi\ra)
		 & =
			\la e^{F} \varphi, -\Delta (e^{F} \varphi)\ra
				- \sum_{x\in\Z} |D_+e^{F}(x)|^2\re(\ol{\varphi(x)}\varphi(x+1)) \\
		 &\ge - \sum_{x\in\Z} |D_+e^{F}(x)|^2\re(\ol{\varphi(x)}\varphi(x+1)).
	\end{align*}
	since $\la e^{F} \varphi, -\Delta (e^{F} \varphi)\ra \ge 0$.
	A simple calculation shows
	\begin{align*}
		|D_+e^{F}(x)|^2
		= 2\left(\cosh(F(x+1)-F(x)) -1\right)
			e^{F(x)}e^{F(x+1)}.
	\end{align*}
	 Thus
	\begin{align*}
		\sum_{x\in\Z}& |D_+e^{F}(x)|^2 \re(\ol{\varphi(x)}\varphi(x+1)) \\
		&= \sum_{x\in\Z}\left(\cosh(D_+F(x)) -1\right)
			 2\re(e^{F(x)}\ol{\varphi(x)}e^{F(x+1)}\varphi(x+1)) \\
		&\le \sum_{x\in\Z}\left(\cosh(D_+F(x)) -1\right)
				\left( |e^{F(x)}\varphi(x)|^2+ |e^{F(x+1)}\varphi(x+1)|^2\right) \\
		&= \la e^{F}\varphi, \left(\cosh(D_+F) + \cosh(D_-F)-2\right) e^{F}\varphi\ra
	\end{align*}
\end{proof}

Since $\cosh$ is even and increasing on $\R_+$ and
\begin{align*}
	|D_\pm F_{\nu,\veps}(x)| = |F_{\nu,\veps}(x\pm 1)- F_{\nu,\veps}(x)|\le F_{\nu,\veps}(1)\le\nu ,
\end{align*}
Lemma \ref{lem:IMS error sharp} gives for $F=F_{\nu,\veps}$ and any solution $\varphi$ of \eqref{eq:GT} with $\omega<0$ and $\dav>0$ the bound
\begin{align}
		\omega \|e^{F}\varphi\|_2^2
		&=\re( \omega\la e^{2F}\varphi,\varphi\ra )
			=
			\dav\re( \la e^{2F}\varphi, -\Delta\varphi\ra )
			- \re( DN(\varphi)[e^{2F}\varphi] ) \nonumber\\
		&\ge
			- \dav\la e^F\varphi, 2(\cosh(\nu) -1) e^F\varphi \ra
			-|DN(\varphi)[e^{2F}\varphi]|
\end{align}
In other words, since $\omega<0$, we have the bound
\begin{align}\label{eq:super}
	\left(|\omega|-2\dav (\cosh(\nu)-1)\right)\|e^{F_{\nu,\veps}}\varphi \|_2^2
	\le
		|DN(\varphi)[e^{2F_{\nu,\veps}}\varphi]|
\end{align}
which  will help to control $\|e^{F_{\nu,\veps}}\varphi \|$ as long as $|\omega|>2\dav (\cosh(\nu)-1)$. To control the right hand side of \eqref{eq:super}, we note
\begin{lemma}\label{lem:superduper}
	Assume that $V$ obeys the assumption \ref{ass:A1} and $\mu$ is a finite measure with compact support.
	Then, if $e^{\nu_0|\cdot|} \varphi(\cdot)\in l^2(\Z)$ for some $\nu_0>0$ we have
	\begin{align}\label{eq:superduper}
		\limsup_{\veps\to 0} |DN(\varphi)[e^{2F_{\nu,\veps}}\varphi]| <\infty  	
	\end{align}
	for all $0<\nu\le \frac{\gamma_1}{2}\nu_0$.
\end{lemma}
\begin{remark} One might hope that even despite the nonlocal nature of $DN$ one could have
\begin{align*}
	|DN(\varphi)[e^{F}\psi]|\lesssim |DN(e^{F}\varphi)\psi]| .
\end{align*}
Setting $\psi_\veps\coloneq e^{F_{\nu_0,\veps}}\varphi$ and using  $2F_{\nu,\veps}= F_{2(\nu-\nu_0),\veps}+ 2F_{\nu_0,\veps}$, one could conclude from this
\begin{align*}
	|DN(\varphi)[e^{2F_{\nu,\veps}}\varphi]|
	=
		|DN(\varphi)[e^{F_{\nu_0,\veps}}e^{F_{2(\nu-\nu_0),\veps}}\psi_\veps]|
	\lesssim
			|DN(\psi_\veps)[e^{F_{2(\nu-\nu_0),\veps}}\psi_\veps]|
\end{align*}
but since $\gamma_1>2$, we have $\nu-\nu_0>0$ for $\nu_0<\nu\le \frac{\gamma_1}{2}\nu_0$ and this leaves an \emph{excess exponential weight} $F_{2(\nu-\nu_0),\veps}$. The point of the Lemma is that  this excess weight is absorbed by the nonlinearity even though it is nonlocal.
\end{remark}
\begin{proof}
	Set $\psi_\veps\coloneq e^{F_{\nu_0,\veps}}\varphi$. Then $\limsup_{\veps\to 0}\|\psi_\veps\|_2<\infty$ and with \eqref{eq:DNbound} and Definition \ref{def:L} we have
	\begin{align*}
		|DN(\varphi)[e^{2F_{\nu,\veps}}\vphi]|
		& =
			 |DN(\varphi)[e^{F_{2\nu-\nu_0,\veps}}\psi_\veps]|
		\lesssim	
			L_\mu^{\gamma_1}(e^{F_{2\nu-\nu_0,\veps}}\psi_\veps,\varphi)
				+ L_\mu^{\gamma_2}(e^{F_{2\nu-\nu_0,\veps}}\psi_\veps,\varphi)
	\end{align*}
	Using \eqref{eq:exchange}, we see
	\begin{align*}
		L_\mu^\gamma(e^{F_{2\nu-\nu_0,\veps}}\psi_\veps,\varphi)
		&\le \mu(\R)\sup_{|r|\le B} \|T_r(e^{F_{2\nu-\nu_0,\veps}}\psi_\veps)|T_r\varphi|^{\gamma-1}\|_1 \\
		&\le \mu(\R) (2e^{4B(1+e^\nu)})^\gamma  \|\psi_\veps\|_2 \|e^{F_{(2\nu-\nu_0)/(\gamma-1),\veps}}\varphi\|_2^{\gamma-1}
	\end{align*}
	for $\gamma=\gamma_1,\gamma_2$. By assumption, $\limsup_{\veps\to0}\|\psi_\veps\|_2<\infty$ and in order to have
	\begin{align*}
		\limsup_{\veps\to0}\|e^{F_{(2\nu-\nu_0)/(\gamma-1),\veps}}\varphi\|_2<\infty
	\end{align*}
	we need $2\nu-\nu_0\le (\gamma_1-1)\nu_0$, which is equivalent to $\nu\le\frac{\gamma_1}{2}\nu_0$, so  \eqref{eq:superduper} follows.
\end{proof}
Before we come our key result for boosting the exponential decay rate, we need some more notation. Note that $0\le \nu\mapsto 2\dav(\cosh(\nu)-1)$ is strictly increasing from zero to infinity. Thus for any $\omega<0$ there exist a unique $\ol{\nu}>0$ such that
\begin{align}\label{eq:olnu}
	2\dav(\cosh(\ol{\nu})-1) = |\omega|.
\end{align}
In other words, $\ol{\nu}$ is given by the right hand side of  \eqref{eq:exponential decay rate}.
\begin{proposition}[Boosting the exponential decay rate]\label{prop:exponential boost} Assume that $V$ obeys the assumption \ref{ass:A1} and that $\varphi$ is a solution of \eqref{eq:GT} for some $\omega<0$ and
	$\dav>0$, and $\ol{\nu}$ is given by \eqref{eq:olnu}.  Furthermore, assume that for some $0<\nu<\ol{\nu}$ we have $e^{\nu|\cdot|}\varphi\in l^2(\Z)$.
	If $\delta>0$ is such that
	\begin{align*}
		\nu+ \delta <\ol{\nu} \quad \text{ and }\quad \delta \le \frac{\gamma_1-2}{2}\nu
	\end{align*}
	then $	e^{(\nu+\delta)|\cdot|}\varphi \in l^2(\Z).$
\end{proposition}
\begin{proof} Let $\nu_1 \coloneq \nu+\delta<\ol{\nu}$ . Then $|\omega|-2\dav(\cosh(\nu_1)-1)>0$ and \eqref{eq:super} shows
 \begin{align}\label{eq:super1}
 	\|e^{F_{\nu_1,\veps}}\varphi\|_2^2
 	\lesssim
 		|DN(\varphi)[e^{2F_{\nu_1,\veps}}\varphi]| .
 \end{align}
 Since the condition $\delta\le \frac{\gamma_1-2}{2}\nu$ is equivalent to $\nu_1\le \frac{\gamma_1}{2}\nu$, \eqref{eq:super1},  Lemma \ref{lem:superduper}, and the monotone convergence theorem yield
 \begin{align*}
 	\|e^{F_{\nu_1,0}}\varphi\|_2^2 =\lim_{\veps\to0}\|e^{F_{\nu_1,\veps}}\varphi\|_2^2
 	\le \limsup_{\veps\to0 } |DN(\varphi)[e^{2F_{\nu_1,\veps}}\varphi]| <\infty
 \end{align*}
 which proves the claim.
\end{proof}

Now we come to the
\begin{proof}[Proof of Theorem \ref{thm:exponential decay}] {}From Proposition \ref{prop:some exponential decay} we know that
\begin{align}
	\nu_*:= \sup\left\{ \nu>0|\, (x\mapsto e^{\nu|x|}\varphi(x))\in l^2(\Z) \right\} >0 .
\end{align}
	In order to prove the lower bound \eqref{eq:exponential decay rate}, let us assume that, in the contrary, $0< \nu_*<\ol{\nu} $, where $\ol{\nu}$ is given by \eqref{eq:olnu}. Take any $0<\nu_0<\nu_*$ and choose
	\begin{align*}
		\delta= \delta_{\nu_0}\coloneq \min\left( \frac{\ol{\nu}-\nu_0}{2}, \frac{\gamma_1-2}{2}\nu_0 \right) .
	\end{align*}
	Then Proposition \ref{prop:exponential boost} shows $e^{\nu|\cdot|}\varphi\in l^2(\Z)$ for $\nu=  \nu_0+\delta_{\nu_0}$, that is,
	\begin{align*}
		\nu_0+\delta_{\nu_0} \le \nu_*\quad \text{for any }0<\nu_0<\nu_*
	\end{align*}
	by the definition of $\nu_*$. However, since we assumed $0<\nu_*<\ol{\nu}$ and $\gamma_1>2$ we have
	$\frac{\ol{\nu}+\nu_*}{2} > \nu_*$ and  $\frac{\gamma_1}{2}\nu_*>\nu_*$. Thus
	\begin{align*}
		\nu_0+\delta_{\nu_0} = \min\left( \frac{\ol{\nu}+\nu_0}{2}, \frac{\gamma_1}{2}\nu_0 \right) \rightarrow \min \left(  \frac{\ol{\nu}+\nu_*}{2}, \frac{\gamma_1}{2}\nu_*\right) >\nu_*
		\quad \text{as } \nu_0\nearrow \nu_*
	\end{align*}
	which is a contradiction. So $\nu_*\ge \ol{\nu}$.
\end{proof}

\section{Super-exponential decay for zero average diffraction}\label{sec:super-exponential decay for zero average diffraction}
In this section, we show that any solution $\varphi\in \l^2(\Z)$ of \eqref{eq:GT} for zero average diffraction decays super-exponentially, with an explicit lower bound on the decay rate.  We are guided by the approach of \cite{HuLee 2012} and  follow in part their argument, however, we also need to make substantial modifications. Similar to \cite{HuLee 2012}, we focus on the tail distribution $\beta$ of $\varphi$, where
 \beq\label{eq:beta}
 \beta(n):=\left(\sum _{|x|\geq n}|\varphi(x)|^2\right)^{1/2}
 \eeq
 for $n\in \N_0$. Our main tool for showing this very fast decay is the following self-consistency bound on the tail distribution $\beta$, which generalizes the one in \cite{HuLee 2012}. This bound will be important for establishing \emph{some} super-expoential decay in Section \ref{subsec:some super-exp decay}, as well as \emph{boosting} it to the lower bound in Section \ref{subsec:boost super-exp decay}, which together will yield the proof of Theorem  \ref{thm:super-exp decay}.

\begin{proposition}[Self-consistency bound] \label{prop:self-consistency}
Assume that $V$ obeys the conditions of assumption \ref{ass:A1} and $\omega\neq 0$.
If $\varphi$ is a solution of \eqref{eq:GT} for $\dav=0$ then with $\theta\coloneq\gamma_1-1>1$ and for any $m,n\in \N_0$ and $0<\alpha<\frac{1}{2}$ the bound
\beq\label{eq:self-consistency}
\beta(n+m)\lesssim \beta(n)^{\theta}+(m+1)^{-\alpha(m+1)}
\eeq
holds where the implicit constant depends only on $\alpha$, $\omega$, and $\|\varphi\|_2$.
\end{proposition}
\begin{proof}
	If $\varphi$ is a solution of \eqref{eq:GT} with $\omega \neq 0$ and $\dav=0$, then
\begin{align*}
	\la  \varphi, g\ra = -\omega^{-1} DN(\varphi)[g]
\end{align*}
with $DN(\varphi)[g]$ from Remark \ref{rem:DN}. Now define the hard cutoff $\chi_l(x)\coloneq 1$ if $|x|\ge l$ and $\chi_l(x)= 0$ if $|x|\le l-1$ and choose
	$g= \varphi_{\gg}\coloneq\chi_l \vphi$ with $l=n+m$. Then \eqref{eq:DNbound} again shows
\begin{align*}
	\beta(n+m)^2 = \la \varphi_{\gg},\varphi\ra \lesssim L^{\gamma_1}_\mu(\varphi_{\gg},\varphi) +  L^{\gamma_2}_\mu(\varphi_{\gg},\varphi)
\end{align*}
and splitting $\varphi=\varphi_>+\varphi_<$ with $\varphi_>\coloneq \chi_n\varphi$ and $\varphi_<\coloneq \varphi-\varphi_>$, which has support in $[-(n-1),n-1]$ shows
\begin{align*}
	\beta(n+m)^2
		&\lesssim
		  L^{\gamma_1}_\mu(\varphi_{\gg},\varphi) +  L^{\gamma_2}_\mu(\varphi_{\gg},\varphi)\\
		&\lesssim
		 	L^{\gamma_1}_\mu(\varphi_{\gg},\varphi_>) +  L^{\gamma_2}_\mu(\varphi_{\gg},\varphi_>)
		 	+ L^{\gamma_1}_\mu(\varphi_{\gg},\varphi_<) +  L^{\gamma_2}_\mu(\varphi_{\gg},\varphi_<).
\end{align*}
Lemma \ref{lem:Ltwisted} for $F=F_{\nu,\veps}=0$ shows
\begin{align*}
	L^{\gamma}_\mu(\varphi_{\gg},\varphi_>)\lesssim \|\varphi_{\gg}\|_2 \|\varphi_{>}\|_2^{\gamma-1} = \beta(n+m)\beta(n)^{\gamma-1}
\end{align*}
and
\begin{align*}
	L^{\gamma}_\mu(\varphi_{\gg},\varphi_<)\lesssim (m+1)^{-\alpha(m+1)}\|\varphi_{\gg}\|_2 \|\varphi_{<}\|_2^{\gamma-1} \le  (m+1)^{-\alpha(m+1)}\beta(n+m)\beta(0)^{\gamma-1}
\end{align*}
for $\gamma= \gamma_1,\gamma_2$. Since $\beta(n)^{\gamma_2-1}\le \beta(0)^{\gamma_2-\gamma_1}\beta(n)^{\gamma_1-1}$, this finishes the proof.
\end{proof}
The self-consistency bound from Proposition \ref{prop:self-consistency} is our main tool to prove Theorem \ref{thm:super-exp decay}. Again, we split the argument, first we show some super-exponential decay and then we boost this.  The first part is, with considerable changes, similar to the approach in \cite{HuLee 2012}, but since the decay rate of  Theorem \ref{thm:super-exp decay} for $V(a)\sim |a|^2a$ is quite a bit better than in \cite{HuLee 2012}, we have to do much better in the second step.
\subsection{Some super--exponential decay}\label{subsec:some super-exp decay}
\begin{proposition}[Some super-exponential decay]\label{prop:some super-exp decay}
	Let $\beta$ be a decreasing non-negative function, vanishing at infinity,
	which obeys the self-consistency bound \eqref{eq:self-consistency} of
	Proposition \ref{prop:self-consistency} for some $\theta=\gamma_1-1>1$.
	Then there exists $\nu>0$ such that
 \bdm
    \beta(n)\lesssim (n+1)^{-\nu(n+1)} \quad \text{ for all } n\in\N_0\, .
 \edm
\end{proposition}

\begin{corollary}[$=$ first step in the proof of Theorem \ref{thm:super-exp decay}]\label{cor:super-exp-decay}
Assume $\dav=0$ and $V$ obeys assumption \ref{ass:A1}. Then for any solution $\varphi$ of \eqref{eq:GT} with $\omega \neq 0$, there exists $\nu>0$ such that
 \bdm
    |\varphi(x)| \lesssim (|x|+1)^{-\nu(|x|+1)}
 \edm
for all $x\in \Z$.
\end{corollary}
\bpf
Given Proposition \ref{prop:some super-exp decay}, this follows immediately from $|\varphi(x)|\le \beta(|x|)$, where $\beta$ is defined in \eqref{eq:beta}.
\epf

In order to prove Proposition \ref{prop:some super-exp decay}, for any $\nu \ge 0$,  define the weight $H_{\nu}$
by
\begin{align}\label{eq:H_nu}
	H_\nu(s)\coloneq (s+1)^{\nu(s+1)}
\end{align}
for $s\ge 0$ and its regularized version of $H_{\nu,\veps}$ given by
\begin{align}\label{eq:H_nu_veps}
	H_{\nu,\veps}(s)\coloneq \frac{H_\nu(s)}{1+\veps H_\nu(s)} = \frac{1}{H_\nu(s)^{-1}+\veps} .
\end{align}
for $s,\veps\ge 0$. We need some basic properties of $H_{\nu, \veps}$ given in
\begin{lemma} \label{lem:H}
 \itemthm\label{lem:H1} For any $\veps \ge 0$, the function $(\nu,s) \in \R^+\times\R^+ \mapsto H_{\nu,\veps} (s)$ is bounded above by $\veps^{-1}$.
  Moreover, the function $H_{\nu,\veps} (s)$ is  increasing in $s,\nu\ge 0$, decreasing in $\veps\ge 0$, and depends continuously on  $\nu,\veps,s\ge 0$. \\[0.2em]
 \itemthm\label{lem:H2} Let $0<\sigma<1$ and $0<\nu< \frac{\sigma}{8}$. Furthermore\footnote{Recall $\lfloor s\rfloor\coloneq \max \{k\in\Z\, |\, k\le s\}$.}, let
 $m\coloneq \lfloor \sigma(l+1)\rfloor$ for $l\in\N_0$. Then
 	\begin{align}\label{eq:H2}
 		H_\nu(l) (m+1)^{-\frac{m+1}{4}} \le \exp\Big(-(\frac{\sigma}{8}\ln(l+1) +\frac{1}{4e})(l+1)\Big) .
 	\end{align}
 \itemthm\label{lem:H3} Let $\nu\ge 0$, $\theta>1$, $0<\sigma<\frac{\theta-1}{\theta}$, then there exist a constant $C=C(\theta,\sigma,\nu)$ which is decreasing in $\theta$, increasing in $\sigma$ and $\nu$, such that with $n\coloneq l- \lfloor \sigma(l+1)\rfloor$
 \begin{align}\label{eq:H3}
 	H_{\nu,\veps}(l)\le C H_{\nu,\veps}(n)^\theta
 \end{align}
  for all $l\in\N_0$ and $0\le \veps\le 1$.   \\[0.2em]
 \itemthm\label{lem:H4} Let $\beta:\N_0\to\R $ be a bounded function and $\tau\in\N_0$. Then the map
 \begin{align*}
 	[0,\infty)\times (0,\infty)\ni(\nu,\veps)\mapsto \|\beta\|_{\nu,\veps,\tau}\coloneq \sup_{l\ge \tau} H_{\nu,\veps}(l)\beta(l)
 \end{align*}
 is continuous. \\[0.2em]
 \itemthm \label{lem:H5} For $0<\nu $, $\tau\in\N_0$, and an arbitrary bounded
 function $\beta:\N_0\to\R$
 \begin{align}\label{eq:H5}
    \|\beta\|_{\nu,0,\tau} = \lim_{\veps\to 0}  \|\beta\|_{\nu,\veps,\tau}
    =\sup_{0<\veps\le 1} \|\beta\|_{\nu,\veps,\tau} \,\, .
 \end{align}
\end{lemma}
We will give the proof of this Lemma at the end of this section and come to the
\begin{proof}[Proof of Proposition \ref{prop:some super-exp decay}]
We need to show that for some $\nu>0$ and some $\tau\in\N$
\begin{align}
	\sup_{l\ge\tau} H_{\nu}(l)\beta(l) <\infty .
\end{align}
 	Let $\alpha=\frac{1}{4}$ and $\theta=\gamma_1-1>1$.  The self-consistency bound \eqref{eq:self-consistency} shows
 	\begin{align}\label{eq:great1}
 		\beta(l)\lesssim  \beta(n)^{\theta}+(m+1)^{-(m+1)/4}
 	\end{align}
 for all $l,m,n\in\N_0$ with $l=n+m$.

 We fix $\sigma=\frac{\theta-1}{2\theta}$ then $0<\sigma<1/2$ and we consider $0<\nu\le \sigma/8$, which we choose more precisely below, and let
 \begin{align*}
	m= \lfloor \sigma(l+1)\rfloor .
 \end{align*}
Multiplying \eqref{eq:great1} by $H_{\nu,\veps}(l)$ and using Lemma \ref{lem:H3} shows
  \begin{align*}
  	H_{\nu,\veps}(l)\beta(l)\lesssim  (H_{\nu,\veps}(n)\beta(n))^{\theta}+H_\nu(l)(m+1)^{-(m+1)/4}
  \end{align*}
  uniformly in $0<\veps\le 1$,  hence, for any $\tau \in \N_0$, since $n=l-m\ge \lfloor (1-\sigma)(\tau+1)-1\rfloor$ if $l\ge \tau$, also
  \begin{align*}
  	\|\beta\|_{\nu,\veps,\tau}
  		&= \sup_{l\ge\tau}H_{\nu,\veps}(l)\beta(l)
	  \lesssim
	  		\| \beta\|_{\nu,\veps,\tilde{\tau}}^\theta
	  		+ \sup_{l\ge \tau} H_\nu(l)(m+1)^{-(m+1)/4}
	\end{align*}
	where we introduced  $\tilde{\tau}\coloneq \lfloor (1-\sigma)(\tau+1)-1\rfloor$. Note
	\begin{align*}
		\|\beta\|_{\nu,\veps,\tilde{\tau}}^\theta
		&\le  \|\beta\|_{\nu,\veps,\tau}^\theta
			+ \max_{\tilde{\tau}\le n\le\tau-1} (H_{\nu,\veps}(n)\beta(n))^\theta
		\le  \|\beta\|_{\nu,\veps,\tau}^\theta  + H_\nu(\tau-1)^\theta \beta(\tilde{\tau})^\theta \\
		&= \|\beta\|_{\nu,\veps,\tau}^\theta + \tau^{\theta\nu\tau} \beta(\tilde{\tau})^\theta
	\end{align*}
	by the monotonicity of $H_{\nu,\veps}$ and $\beta$. So setting
	$R_1(\tau)\coloneq \sup_{l\ge \tau} e^{-[\frac{\sigma}{8}\ln(l+1) -\frac{1}{4e}](l+1)} $,  using  Lemma \ref{lem:H2} and $0<\nu\le \frac{\sigma}{8}$, we arrive at
	\begin{align}\label{eq:punchline0}
		\|\beta\|_{\nu,\veps,\tau}
  		&\le
  			C\left(\|\beta\|_{\nu,\veps,\tau}^\theta + \tau^{\theta\nu\tau} \beta(\tilde{\tau})^\theta
  				+ R_1(\tau)\right)
	\end{align}	
	for some universal constant $C$ \emph{independent} of $\tau\in\N$. Now let $\tau_1$ be so large that $\frac{1}{\tau\ln\tau}\le \frac{\sigma}{8}$ for all $\tau\ge\tau_1$.  Choosing $\nu\coloneq \frac{1}{\tau\ln\tau}$ gives $\tau^{\theta\nu\tau}=e^\theta$ and with
	\begin{align*}
		G(u)\coloneq u- C u^\theta \quad\text{ for } u\ge 0
	\end{align*}
 we see that \eqref{eq:punchline0} can be rewritten as
 \begin{align}\label{eq:punchline}
 		G( \|\beta\|_{\nu,\veps,\tau}) \le R_2(\tau)
 \end{align}
 for all $\tau\ge\tau_1$, where $\nu=\frac{1}{\tau\ln\tau}$ and
 $R_2(\tau):= C\left(e^\theta \beta(\tilde{\tau})
  				+ R_1(\tau)\right)$
  with $\tilde{\tau}\coloneq \lfloor (1-\sigma)(\tau+1)-1\rfloor$.

 Now the argument continues exactly as in \cite{HuLee 2012}, we will give it for the convenience of the reader: Certainly $G$ is continuous on $[0,\infty)$ with $G(0)=0$ and $\lim_{u\to\infty}G(u)=-\infty$. Also $G$ has a single strictly positive maximum on $[0,\infty)$, that is there exists a single $u_{\max}>0$ such that
  \begin{align*}
  	G_{\max} \coloneq G(u_{\max}) =\sup_{u\ge 0} G(u) >0
  \end{align*}
  and the inverse image of the set $[0,G_{\max}/2]$ under $G$ is given by
  \begin{align*}
  	G^{-1}([0,G_{\max}/2) = [0,u_1]\cup[u_2,\infty)
  \end{align*}
  for some $0<u_1<u_{\max}< u_2<\infty$.

 Note that $\lim_{\tau\to\infty} R_2(\tau)=0$ since $\beta$ and $R_1$ are going to zero at infinity and  $\lim_{\tau\to\infty}\tilde{\tau}=\infty$. Choose $\tau_2\ge\tau_1$ so large that $R_2(\tau)\le G_{\max}/2$ for all $\tau\ge\tau_2$. Then \eqref{eq:punchline} shows that
 \begin{align}\label{eq:punch2}
 	\|\beta\|_{\nu,\veps,\tau} \in [0,u_1]\cup[u_2,\infty)
 \end{align}
 for all $\tau\ge \tau_2$ and  all $0<\veps\le 1$, as long as $\nu=\frac{1}{\tau\ln\tau}$.

 Step 1: Because of Lemma  \ref{lem:H4}, for fixed $\nu,\tau>0$, the map $0<\veps\mapsto \|\beta\|_{\nu,\veps,\tau} $ is continuous. So since $u_1<u_2$, the intermediate value theorem for continuous functions and \eqref{eq:punch2} show that we have, for all $\tau\ge\tau_2$ and $\nu=\frac{1}{\tau\ln\tau}$, the dichotomy
 \begin{align*}
 	\text{either } 0\le \|\beta\|_{\nu,\veps,\tau} \le u_1 \text{ for all } 0<\veps\le 1,
 	\text{ or }  u_2\le \|\beta\|_{\nu,\veps,\tau} \text{ for all }0<\veps\le 1.
 \end{align*}

 Step 2: Since $H_{\tau,1}\le 1$ by Lemma \ref{lem:H1}, we have $\|\beta\|_{\nu,1,\tau}= \beta(\tau)\to0$ as $\tau\to\infty$. So we can choose $\tau\ge \tau_2$ so large such that
 $\|\beta\|_{\nu,1,\tau}\le  u_1$. For this $\tau$ we have from Step 1 that
 \begin{align*}
 	\|\beta\|_{\nu,\veps,\tau} \le u_1 \text{ for all } 0<\veps\le 1 ,
 \end{align*}
 where $\nu=\frac{1}{\tau\ln\tau}>0$.
 Thus also
 \begin{align*}
 	\|\beta\|_{\nu,0,\tau} = \lim_{\veps\to 0} \|\beta\|_{\nu,\veps,\tau} \le u_1 <\infty
 \end{align*}
 by Lemma \ref{lem:H5}. This finishes the proof of
 Proposition \ref{prop:some super-exp decay}.
\end{proof}
Now we come to the
\begin{proof}[Proof of Lemma \ref{lem:H}]
  Part (i) is clear from the definition of $H_{\nu,\veps}$. For part (ii) we note that
  for fixed  $0<\sigma<1$ and  $m\coloneq \lfloor \sigma(l+1)\rfloor$, one has
  $m\le \sigma(l+1)<m+1$, hence
  \begin{align*}
  	H_\nu(l)(m+1)^{-(m+1)/4}
  		&\le (l+1)^{\nu(l+1)} \big(\sigma(l+1) \big)^{-\sigma(l+1)/4} \\
  	 	&=  \exp\left( -\big( (\frac{\sigma}{4}-\nu)\ln(l+1) +\frac{\sigma \ln\sigma}{4} \big)(l+1)\right) \\
  	 	&\le \exp\left( -\big( \frac{\sigma}{8}\ln(l+1) -\frac{1}{4e} \big)(l+1)\right)
  \end{align*}
 since $0<-\sigma\ln\sigma\le e^{-1}$ for all $0 < \sigma < 1$ and $\nu\le \sigma/8$. This proves \eqref{eq:H2}.

 Of course, \eqref{eq:H3} holds with constant
 \begin{align*}
 	C:= \sup_{l\in \N_0}\sup_{0\le \veps\le 1} g(n,l,\veps)
 \end{align*}
 where $n= l- \lfloor \sigma(l+1) \rfloor$ and
 \begin{align*}
 	g(n,l,\veps)\coloneq \frac{H_{\nu,\veps}(l)}{H_{\nu,\veps}(n)^\theta}
 	= \frac{(H_\nu(n)^{-1}+\veps)^\theta}{H_{\nu}(l)^{-1}+\veps}
 \end{align*}
 where we droped, for simplicity of notation, the dependence of $g$ on $\theta$ and $\nu$.
 Since $H_{\nu,\veps}\ge 1$, $g$ is certainly decreasing in $\theta>1$, and so is $C$.
 A simple computation shows
 \begin{align*}
 	\frac{\partial}{\partial\veps} g(n,l,\veps)
 	= \frac{(H_\nu(n)^{-1}+\veps)^{\theta-1}}{(H_{\nu}(l)^{-1}+\veps)^{2}}
 	\left( \theta H_\nu(l)^{-1} -H_\nu(n)^{-1} + (\theta-1)\veps \right) .
 \end{align*}
 Since $n\le l$ and $\theta>1$, the map $0\le\veps\mapsto \theta H_\nu(l)^{-1} -H_\nu(n)^{-1} + (\theta-1)\veps $ is either positive for all $\veps\ge 0$, or it is negative for small and  positive for large $\veps$, with a single zero for some $\veps>0$. Thus the map
 $0\le \veps\mapsto g(n,l,\veps)$ is either increasing in $\veps\ge 0$, or it decreasing for small and increasing for large $\veps\ge 0$, with a single minimum at some $\veps>0$ and no maximum in $(0,\infty)$. Thus the supremum of $g(n,l,\veps)$ over $0\le \veps\le 1$ is attained at the boundary,
 \begin{align*}
 	\sup g(n,l,\veps) = \max(g(n,l,0), g(n,l,1))
 \end{align*}
  for all $0\le n\le l$. We have
  \begin{align*}
  	g(n,l,1) = \frac{(H_\nu(n)^{-1}+1)^\theta}{H_{\nu}(l)^{-1}+1}\le 2^\theta
  \end{align*}
  for all $n,l\in\N_0$ and, because $n=l-\lfloor \sigma(l+1) \rfloor\ge l- \sigma(l+1) = (1-\sigma)(l+1)-1 $,
  \begin{align*}
  	g(n,l,0)
  		&\le  \frac{(l+1)^{\nu(l+1)}}{((1-\sigma)(l+1))^{\theta\nu(1-\sigma)(l+1)}} \\
  		&=
  			\exp\left( \nu\Big[(1-\theta(1-\sigma)) (l+1)\ln(l+1) - \theta(1-\sigma)\ln(1-\sigma)(l+1) \Big] \right) \\
  		&\le \exp\left( -\nu\Big[(\theta(1-\sigma)-1)\ln(l+1) - e^{-1}\theta \Big](l+1) \right)
  \end{align*}
  A short calculation reveals that for $a,b>0$ the maximum of $B(s)= -(a\ln s-b)s$ over $s>0$ is attained at $  	 B_{\max} = a e^{\frac{b}{a}-1} $
  so with $a= \theta(1-\sigma)-1$ and $b=e^{-1}\theta$ this shows
  \begin{align*}
  	g(n,l,0)
  	\le \exp\left( \nu (\theta(1-\sigma)-1) \exp\left( \frac{e^{-1}\theta}{(\theta(1-\sigma)-1)}-1 \right) \right)
  \end{align*}
  for all $l\in\N_0$ and with $n=l- \lfloor \sigma(l+1) \rfloor$ as long as $\theta(1-\sigma)>1$, which in turn is equivalent to $\sigma< \frac{\theta-1}{\theta}$. This proves  \eqref{eq:H3} and alos shows that the constant $C$ is increasing in $\nu$.

	To prove part (iv) note that because for the triangle inequality
	\begin{align*}
		|\|\beta\|_{\nu',\veps',\tau}-\|\beta\|_{\nu,\veps,\tau}|
		\le
			\sup_{l\in\N_0} \left|H_{\nu',\veps'}(l) - H_{\nu,\veps}(l)\right|
			\sup_{l\in\N_0}|\beta(l)|
	\end{align*}
	for all $\nu,\nu'\ge 0$ and $\veps,\veps'>0$. Note that
	\begin{align*}
		\sup_{l\in\N_0} \left|H_{\nu',\veps'}(l) - H_{\nu,\veps}(l)\right|
		\le \sup_{s\in[0,1]} \left| h(\nu',\veps',s) - h(\nu,\veps,s)  \right|
	\end{align*}
	with $h(\nu,\veps,s)\coloneq (s^{\nu s} +\veps)^{-1}$. The function $h$ is continuous on
	$[0,\infty)\times(0,\infty)\times [0,1]$ and thus uniformly continuous on
	$[0,\kappa^{-1}]\times[\kappa^{-1},\kappa]\times [0,1]$ for any $\kappa>0$.  Thus, for any $r>0$ there exist $\delta>0$ with $|h(\nu',\veps',s') - h(\nu,\veps,s) | \le r$ as long as $0\le \nu,\nu'\le \kappa$, $\kappa^{-1}\le \veps,\veps'\le \kappa$ and $0\le s,s'\le 1$ are such that
	$|\nu'-\nu|,|\veps'-\veps|, |s'-s|\le \delta$. Thus for these $\nu,\nu'$ and $\veps,\veps'$ also
	\begin{align*}
		\sup_{0\le s\le 1}| h(\nu',\veps',s) - h(\nu,\veps,s) | \le r .
	\end{align*}
	Hence also
		\begin{align*}
		\sup_{l\in\N_0} \left|H_{\nu',\veps'}(l) - H_{\nu,\veps}(l)\right|
		\le \sup_{s\in[0,1]} \left| h(\nu',\veps',s) - h(\nu,\veps,s)  \right|
		\le r
	\end{align*}
	for all $0\le \nu,\nu'\le \kappa$, $\kappa^{-1}\le \veps,\veps'\le \kappa$ with
	$|\nu'-\nu|,|\veps'-\veps| \le \delta$. Since $\kappa>1$ is arbitrary, this shows the continuity of $\|\beta\|_{\nu,\veps,\tau}$ in $\nu\ge 0$ and $\veps>0$.
	
	To prove the last claim, we simply note that $H_{\nu,\veps}$ is decreasing in $\veps>0$, so the map $0<\veps\mapsto \|\beta\|_{\nu,\veps,\tau}$ is decreasing. By the monotone convergence theorem and since one can interchange suprema, we get
	\begin{align*}
		\lim_{\veps\to 0} \|\beta\|_{\nu,\veps,\tau}
		&= \sup_{0<\veps\le 1}\|\beta\|_{\nu,\veps,\tau}
		= \sup_{0<\veps\le 1} \sup_{l\ge \tau }H_{\nu,\veps}(l)\beta(l) \\
		&= \sup_{l\ge \tau }\sup_{0<\veps\le 1}H_{\nu,\veps}(l)\beta(l)
		= \sup_{l\ge \tau }H_{\nu,0}(l)\beta(l)
		= \|\beta\|_{\nu,0,\tau}
	\end{align*}
	which proves \eqref{eq:H5} and finishes the proof of Lemma \ref{lem:H}.
\end{proof}

\subsection{Boosting the (super--exponential) decay rate}\label{subsec:boost super-exp decay}
\begin{proposition}[Boosting the super-exponential decay rate]
\label{prop:boost super-exp decay}
	Let $\beta$ be a non-negative function which obeys the self-consistency bound \eqref{eq:self-consistency} of
	Proposition \ref{prop:self-consistency} for some $\theta>1$ and some $0<\alpha<\frac{1}{2}$.
	Furthermore, assume that for
	some $\nu>0$ we have
	\begin{align*}
		\beta(l)\lesssim (l+1)^{-\nu(l+1)} \quad \text{ for all } l\ge0 .
	\end{align*}
	Then for all $0<\alpha_1<\alpha$,  setting $\nu_1\coloneq \frac{2\theta\alpha_1}{\alpha_1+\theta\nu}\nu$, we have
	\begin{align*}
		\beta(l)\lesssim (l+1)^{-\nu_1(l+1)} \quad \text{ for all } l\ge0  .
	\end{align*}
\end{proposition}
\begin{remark}
 $\nu_1>\nu$ is equaivalent to $\nu<\frac{2\theta-1}{\theta}\alpha_1$. So Proposition \ref{prop:boost super-exp decay} allows us to boost the decay rate as long as $\nu<\frac{2\theta-1}{\theta}\alpha$, since $\nu<\frac{2\theta-1}{\theta}\alpha_1$ whenever $\alpha_1$ close enough to $\alpha$.
\end{remark}

\begin{proof}
	The self-consistency bound \eqref{eq:self-consistency} and our assumptions
	on $\beta$ imply
	\begin{align*}
		\beta(l)\lesssim (n+1)^{-\theta\nu(n+1)} + (m+1)^{-\alpha(m+1)}
	\end{align*}
	for all $l,m,n\in\N_0$ with $l=n+m$.
	
	Set $m=\lfloor \sigma(l+1)\rfloor$ for some $0<\sigma<1$, which we choose later. Then $m\le \sigma(l+1)<\sigma(l+1)+1$ and for
	$n= l- m$ we have $(1-\sigma)(l+1)-1\le n< (1-\sigma)(l+1)$, that is,
	$n= \lfloor (1-\sigma)(l+1)-1 \rfloor$.
	Then the self-consistency bound implies
	\begin{align*}
		\beta(l)
			& \lesssim ((1-\sigma)(l+1))^{-\theta\nu(1-\sigma)(l+1)}
						+ (\sigma(l+1))^{-\alpha\sigma(l+1)} \\
			&=
				\exp\left(-\big(\theta\nu(1-\sigma)\ln(1-\sigma) + \alpha\sigma\ln\sigma \big)	(l+1) -\big( \theta\nu(1-\sigma) + \alpha\sigma\big)(l+1)\ln(l+1)
				\right) \\
			&\le
				\exp\left(\big(\theta\nu + \alpha \big)e^{-1}(l+1) -\big( \theta\nu(1-\sigma) + \alpha\sigma\big)(l+1)\ln(l+1)
				\right) \\
			&\lesssim
				\exp\left(-\big( \theta\nu(1-\sigma) + \alpha_1\sigma\big)(l+1)\ln(l+1)
				\right)
	\end{align*}
	for any $0<\alpha_1<\alpha$ and all $l\in\N_0$, where we also used $\sigma\ln\sigma\ge -e^{-1}$ and $(1-\sigma)\ln(1-\sigma)\ge -e^{-1}$ for all $0 < \sigma < 1$ in the third line.
	
	We choose $\sigma$ such that $\theta\nu(1-\sigma) = \alpha_1\sigma$, equivalently
	\begin{align*}
		\sigma = \frac{\theta\nu}{\alpha_1+\theta\nu} .
	\end{align*}
	This yields $0<\sigma<1$ and $\theta\nu(1-\sigma) + \alpha_1\sigma= 2\alpha_1\sigma = \frac{2\theta\alpha_1}{\alpha_1+\theta\nu}\nu=\nu_1$. So
	\begin{align*}
		\beta(l)\lesssim \exp\left( -\nu_1(l+1)\ln(l+1) \right)	
	\end{align*}
	for all $l\in\N_0$, which finished the proof.
\end{proof}

\begin{corollary}\label{cor:super-exp decay rate lower bound}
	Let $\beta:\N_0\to\R$ be a decreasing non-negative function, vanishing at infinity,
	which obeys the self-consistency bound \eqref{eq:self-consistency} of
	Proposition \ref{prop:self-consistency} for some $\theta>1$ and all $0<\alpha<1$. Furthermore, recall $H_{\nu}(l)=(l+1)^{\nu(l+1)}$ for $l\in\N_0$ and $\nu\in\R$.
	Then
	\begin{align*}
		\nu_{**}=\sup\left\{\nu>0|\,  \beta\lesssim H_{-\nu}\right\}
			\ge 1- \frac{1}{2\theta}
	\end{align*}
\end{corollary}

\begin{proof}
	From Proposition \ref{prop:some super-exp decay} we know that $\nu_{**}>0$. Let $0<\nu<\nu_{**}$ and $0<\alpha_1<\frac{1}{2}$, then Proposition \ref{prop:boost super-exp decay} shows
	\begin{align*}
		\beta\lesssim H_{-\nu_1}
	\end{align*}
	for $\nu_1\coloneq \frac{2\theta\alpha_1\nu}{\alpha_1+\theta\nu}$. Thus, by the definition of $\nu_{**}$ we have
	\begin{align*}
		\frac{2\theta\alpha_1\nu}{\alpha_1+\theta\nu} \le \nu_{**}
	\end{align*}
	for all $0<\nu<\nu_{**}$ and all $0<\alpha_1<\frac{1}{2}$. Taking first the limit
	$\nu\nearrow\nu_{**}$ and then $\alpha_1\nearrow\frac{1}{2}$ in the above inequality shows
	\begin{align*}
		\frac{\theta\nu_{**}}{\frac{1}{2}+\theta\nu_{**}} \le \nu_{**} .
	\end{align*}
	Since $\nu_{**}>0$ this implies $\nu_{**}\ge 1-\frac{1}{2\theta}$.
\end{proof}
Now we can give the
\begin{proof}[Proof of Theorem \ref{thm:super-exp decay}]
 Let $\theta=\gamma_1-1$ and $\varphi$ be a solution of \eqref{eq:GT-weak}. Then Propsition \ref{prop:self-consistency} shows that 	the tail distribution $\beta$ of $\varphi$ obeys the self-conistency bound \eqref{eq:self-consistency}. Then the claim follows from
 $|\varphi(x)|\le \beta(|x|)$ for all $x\in\Z$ and Corollary \ref{cor:super-exp decay rate lower bound}.
\end{proof}

\appendix
\setcounter{section}{0}
\renewcommand{\thesection}{\Alph{section}}
\renewcommand{\theequation}{\thesection.\arabic{equation}}
\renewcommand{\thetheorem}{\thesection.\arabic{theorem}}
%
\section{Some useful bounds} \label{sec:useful bounds}
We start with
\begin{lemma}\label{lem:useful}
	\itemthm Let $ 1\le p<\infty$ and $f_1,f_2\in l^p(\Z)$, then
		\begin{align}\label{eq:lp-difference}
			\left| \|f_1\|_p^p - \|f_2\|_p^p \right|
			\leq
				p\max(\|f_1\|_p^{p-1}, \|f_2\|_p^{p-1})
				\| f_1-f_2\|_p
		\end{align}
	\itemthm The free time evolution group $T_r= e^{ir\Delta}$ is bounded on $l^p(\Z)$ for all $1\le p\le \infty$ with
		\begin{align}\label{eq:lpbound}
			\|T_rf\|_p\le e^{4|r| |1-\frac{2}{p}|}\|f\|_p.
		\end{align}
	\itemthm The group $T_r= e^{ir\Delta}$ is norm continuous on $l^p$ for any $1\le p\le \infty$ with
		\begin{align}\label{eq:lpcontinuity}
			\|f-T_rf\|_{p}
				\le
					\left(e^{4|r|}-1\right)\|f\|_p .
		\end{align}\\
	\itemthm For the kernel\footnote{We use the physicists' notation $\la x|T_r|y\ra$ for the kernel, for mathematicians, $\la x|T_r|y\ra = \la \delta_x, T_r\delta_y \ra$, where $\delta_x$ is the Kronecker delta at $x\in\Z$. } of $T_r$, one has the bound
		\begin{align}\label{eq:kernel bound}
			|\la x|T_r|y\ra| \le \min\left(1, e^{4|r|}\frac{(4|r|)^{|x-y|}}{|x-y|!}\right) .
		\end{align}
	\itemthm {\rm(}Strong bilinear bound{\rm)}  For any $1\le p\le \infty$
		\begin{align}\label{eq:strong bilinear}
			\sup_{|r|\le B}\|T_rf_1 T_rf_2\|_p
				\le
				\min\left(1,\frac{8e^{16B}(4B)^{\lceil\frac{s}{2}\rceil}}{\lceil\frac{s}{2}\rceil!} \right) \|f_1\|_2 \|f_2\|_2
		\end{align}
		with $s:=\dist(\supp f_1, \supp f_2)$. \\
	\itemthm {\rm(}Twisted strong bilinear bound{\rm)} For any $1\le p\le \infty$ and $B>0$
		\begin{align}\label{eq:twisted strong bilinear}
			\sup_{r\in [-B,B]}\|T_r(e^{F_{\nu,\veps}} h_1) T_r(e^{-F_{\nu\,\veps}}h_2)\|_p
				\le
				4 e^{8B(1+e^\nu)} \min\left(1,\frac{2(4Re^\nu)^{\lceil\frac{s}{2}\rceil}}{\lceil\frac{s}{2}\rceil!} \right) \|h_1\|_2 \|h_2\|_2
		\end{align}
		\emph{uniformly in } $\veps>0$. Here $s:=\dist(\supp h_1, \supp h_2)\ge 0$ and  $F_{\nu,\veps}(x)=\frac{\nu|x|}{1+\veps|x|}$.\\
	\itemthm {\rm(}Exchange of an exponential weight{\rm)} For any $\alpha\ge 1$, $1\le p\le\infty $, and $B>0$
		\begin{align}\label{eq:exchange}
			\sup_{|r|\le B}\|T_r(e^{F_{\nu,\veps}}h_1)|T_r h_2|^\alpha\|_p
			\le (2e^{4B(1+e^\nu)})^{\alpha+1}  \|h_1\|_2 \|e^{F_{\nu/\alpha,\veps}}h_2\|_2^\alpha .
		\end{align}
	\itemthm For any $\nu>0$ and $A>0$, let $f_\nu(x):= Ae^{-\nu|x|}$. Then
		\begin{equation}
			\begin{split}\label{eq:f_nu}
			\|f_\nu\|_\kappa^\kappa
				&=
					A^\kappa \frac{\cosh(\frac{\kappa}{2}\nu)}{\sinh(\frac{\kappa}{2}\nu)}, \\
			\la f_\nu, -\Delta f_\nu \ra
				&= \|D_+f_\nu\|_2^2
					= 4A^2\, \frac{\sinh^2(\nu/2)}{\sinh(\nu)}.
			\end{split}
		\end{equation}
\end{lemma}

\begin{remark}
  The strong bilinear bound above strengthens the strong bilinear bound from \cite{HuLee 2012-existence, HuLee 2012}, which was proven there only for $p=2$. Moreover,  we will give a proof which is considerably simpler than the one in \cite{HuLee 2012-existence, HuLee 2012}.
  The twisted strong bilinear bound \eqref{eq:twisted strong bilinear} is new and needed in the proof that solutions of \eqref{eq:GT} with $\omega<0$ have some exponential decay for positive average dispersion. It is important that the right hand side of \eqref{eq:twisted strong bilinear} is \emph{independent} of $\veps>0$. The exchange of exponential weights bound \eqref{eq:exchange} is crucial for our strategy of boosting the exponential decay rate to the one given by the physical heuristic. The main feature of \eqref{eq:exchange} is that for $\alpha>1$ its right hand side has an  exponential growth of order $\nu/\alpha$ which is strictly smaller than $\nu$ when $\nu>0$. Thus \eqref{eq:exchange} allows us to  \emph{absorb} some \emph{excess exponential factor} in  the boosting argument of Section \ref{subsec:boost}.  	
\end{remark}
\begin{proof}
	For the first claim, let $f_1,f_2\in l^p(\Z),\ 1 \le p < \infty$ and note that for $a,b\ge 0$ one has
	\begin{align}\label{eq:difference-1}
		|a^p-b^p|\le p \max(a^{p-1}, b^{p-1}) |a-b|
	\end{align}
	since, if $a\le b$, then
	\begin{align*}
		|a^p-b^p|= b^p-a^p =p\int_a^b s^{p-1}\, ds \le p b^{p-1} (b-a)
	\end{align*}
	and the case $a\ge b$ follows by symmetry. Using $a=\|f_1\|_p$ and $b=\|f_2\|_p$ in \eqref{eq:difference-1} shows
	\begin{align*}
		\left| \|f_1\|_p^p - \|f_2\|_p^p\right|
		\le p \max(\|f_1\|_p^{p-1},\|f_2\|_p^{p-1})
			    	\left| \|f_1\|_p - \|f_2\|_p \right|
 	\end{align*}
	which gives \eqref{eq:lp-difference}.

	As a preparation for the proof of the other claims, note that $T_r$ has the norm continuous series expansion
	\begin{align*}
		T_r= \sum_{n=0}^\infty \frac{(ir)^n}{n!} \Delta^n
	\end{align*}
	One easily sees that
        \begin{equation}\label{eq:laplacian}
        \|\Delta f\|_1\le 4\|f\|_1 \text{ and } \|\Delta f\|_\infty\le 4\|f\|_\infty.
        \end{equation}
	Thus  the norm of $\Delta$ on $l^\infty(\Z)$ and $l^1(\Z)$ is bounded by $4$ and from
	the power series for $T_r$ one sees
	\begin{align*}
		\|T_r f \|_p \le \sum_{n=0}^\infty \frac{(4|r|)^n}{n!} \|f\|_p = e^{4|r|} \|f\|_p,\ p=1 \text{ or } \infty.
	\end{align*}
	By self-adjointness of $\Delta$ on $l^2(\Z)$ one has that $T_r$ is unitary on $l^2(\Z)$, so $\|T_rf\|_2= \|f\|_2$ and interpolating this with the bound on $l^1(\Z)$ and $l^\infty(\Z)$ with the help of the  Riesz-Thorin  interpolation theorem proves
	\eqref{eq:lpbound}.

	Moreover, applying Riesz-Thorin interpolation theorem on \eqref{eq:laplacian} yields $\|\Delta f\|_p\le 4\|f\|_p$ for all $1\le p\le \infty$.
	The series expansion for $T_r$ then yields
	\begin{align*}
		\|f-T_rf\|_p \le \sum_{n=1}^\infty \frac{(4|r|)^n}{n!}\|f\|_p = (e^{4|r|}-1) \|f\|_p
	\end{align*}
	which is \eqref{eq:lpcontinuity}. In particular, $T_r$ is norm continuous on $l^p(\Z)$ at $r=0$, which together with the group property of $T_r$ and \eqref{eq:lpbound} shows its continuity for all $r$.

	Because of the norm convergent series expansion for $T_r$, its  `kernel' $\la x|T_r|y\ra$,  for which one has $T_rf(x) = \sum_{y\in\Z}\la x|T_r|y\ra f(y)$, is given by
	\begin{align}\label{eq:kernel}
		\la x|T_r|y\ra = \sum_{n=0}^\infty \frac{(ir)^n}{n!}
		\la x|\Delta^n|y\ra
		= \sum_{n=|x-y|}^\infty \frac{(ir)^n}{n!}
		\la x|\Delta^n|y\ra
	\end{align}
	since $\la x|\Delta^n|y\ra =0$ if $n<|x-y|$. Moreover, $|\la x|\Delta^n|y\ra|\le \|\Delta\|^n =4^n$, so we have the bound
	\begin{align}\label{eq:kernel bound-1}
		|\la x|T_r|y\ra | \le \sum_{n=|x-y|}^\infty \frac{(4|r|)^n}{n!} .
	\end{align}
	By unicity of $T_r$ and the Cauchy-Schwartz inequality one always has
	$|\la x|T_r|y\ra|\le 1$.
	Moreover, by \eqref{eq:kernel bound-1}, we have
	\begin{align}\label{eq:Tr bound}
		|\la x|T_r|y\ra | &\le \sum_{n=|x-y|}^\infty \frac{(4|r|)^n}{n!}
			= (4|r|)^{|x-y|}\sum_{n=0}^\infty \frac{(4|r|)^n}{(|x-y|+n)!} \nonumber\\
			&\le \frac{(4|r|)^{|x-y|}}{|x-y|!} \sum_{n=0}^\infty \frac{(4|r|)^n}{n!}
				=  \frac{e^{4|r|}(4|r|)^{|x-y|}}{|x-y|!}
	\end{align}
	since $(|x-y|+n)!\ge |x-y|! n!$. So \eqref{eq:kernel bound} follows.
	
	To prove the fifth claim, we first note that on the sequence spaces $l^p(\Z)$, the bound
	$\|h\|_p\le \|h\|_1$ holds. Hence $\|T_r f_1 T_rf_2\|_p\le \|T_rf_1 T_rf_2\|_1$, so we only have to prove \eqref{eq:strong bilinear} for $p=1$. Because of the Cauchy Schwarz inequality,
	\begin{align*}
		\|T_r f_1 T_r f_2\|_1
				\le
					\|T_r f_1\|_2  \|T_r f_2\|_2
					= \|f_1\|_2\|f_2\|_2
	\end{align*}
	Now let $s:=\dist(\supp f_1, \supp f_2)>0$. Then with
	\begin{align*}
		A_r(y_1,y_2):= e^{8|r|}\sum_x \frac{(4|r|)^{|x-y_1|}}{|x-y_1|!}
						\frac{(4|r|)^{|x-y_2|}}{|x-y_2|!}
	\end{align*}
	the bound \eqref{eq:kernel bound}, the Cauchy-Schwarz inequality and the symmetry of $A_r$ in $y_1$ and $y_2$ gives
	\begin{align}\label{eq:A_r bound 0}
		\|T_r f_1 T_r f_2\|_1
			&\le
				 \sum_{y_1,y_2} |f(y_1)|A_r(y_1,y_2) |f(y_2)| \nonumber\\
			&\le \bigg( \sum_{\substack{y_1\in\Z\\y_2\in\supp f_2}} |f(y_1)|^2A_r(y_1,y_2)\bigg)^{1/2}
				\bigg(\sum_{\substack{y_2\in\Z\\y_1\in\supp f_1}} A_r(y_1,y_2)|f(y_2)|^2\bigg)^{1/2}  \nonumber\\
			&\le \left(A_{r,1,2} A_{r,2,1}\right)^{1/2}
				\|f_1\|_2 \|f_2\|_2
	\end{align}
	where $A_{r,l,m}= \sup_{y_1\in\supp f_l}\sum_{y_2\in\supp f_m} A_r(y_1,y_2)$.
	
	Fix $y_1\in\supp f_1$, then for all $x\in\Z$ and all $y_2\in\supp f_2$ we have
		$|x-y_1|\ge \frac{s}{2}$ or
		$|x-y_2|\ge \frac{s}{2}$ and since the distance is always an integer, setting $\lceil s\rceil:=\min\{n\in\Z |\, s\le n\}$ and denoting $G_r(y):=\frac{(4|r|)^{|y|}}{|y|!} $,  we get
		\begin{align}\label{eq:A_r bound 1}
			&\sum_{y_2\in\supp f_2} A_r(y_1,y_2)
			  = e^{8|r|}\sum_{y_2\in\supp f_2}\sum_{x\in \Z} G_r(x-y_1)G_r(x-y_2)\nonumber\\
			&\le e^{8|r|}\Big( \sum_{\substack{y_2, x \\ |x-y_1|\ge \lceil\frac{s}{2}\rceil}} G_r(x-y_1)G_r(x-y_2)
				+ \sum_{\substack{y_2,x \\ |x-y_2|\ge \lceil\frac{s}{2}\rceil}}
					G_r(x-y_1)G_r(x-y_2)\Big) \nonumber\\
			& = e^{8|r|}\Big(
				\sum_{|x|\ge \lceil\frac{s}{2}\rceil} G_r(y_1)  \sum_{y_2}G_r(y_2)
			 	+ \sum_{x} G_r(x)\sum_{|y_2|\ge \lceil\frac{s}{2}\rceil}G_r(y_2)\Big) \nonumber\\
			& = 2 e^{8|r|} \sum_{|x|\ge \lceil\frac{s}{2}\rceil} G_r(x)\sum_{y}G_r(y) .
		\end{align}
	A simple calculation gives
	\begin{align}\label{eq:A_r bound 2}
		\sum_{y} G_r(y)
			=
			G_r(0) +2\sum_{y\ge 1} G_r(y)
			=
			1+2 \sum_{y=1}^\infty \frac{(4|r|)^{y}}{y!}
			\le
				2 e^{4|r|}
	\end{align}
	and
	\begin{align}\label{eq:A_r bound 3}
		\sum_{|x|\ge \lceil\frac{s}{2}\rceil} G_r(x)
			&=
				2 \sum_{x= \lceil\frac{s}{2}\rceil}^\infty \frac{(4|r|)^{x}}{x!}
				=
					2(4|r|)^{\lceil\frac{s}{2}\rceil}	\sum_{n= 0}^\infty \frac{(4|r|)^{n}}{(\lceil\frac{s}{2}\rceil+n)!}
				\le \frac{2e^{4|r|}(4|r|)^{\lceil\frac{s}{2}\rceil}}{\lceil\frac{s}{2}\rceil!} .
	\end{align}
 Thus
	\begin{align}\label{eq:A_r bound 4}
		A_{r,1,2} \le \frac{8e^{16|r|}(4|r|)^{\lceil\frac{s}{2}\rceil}}{\lceil\frac{s}{2}\rceil!} .
	\end{align}
	The same argument yields the same bound for $A_{2,1}$ and since
	 \begin{align*}
	 	\|T_r f_1 T_r f_2\|_1\le \min(1, (A_{r,1,2}A_{r,2,1}))^{1/2}) \|f_1\|_2 \|f_2\|_2
	 \end{align*}
	 this  proves \eqref{eq:strong bilinear}.
	
	In order to prove \eqref{eq:twisted strong bilinear}, we can again, without loss of generality, consider the case $p=1$. Fix $\nu\ge 0$ and $\veps>0$ and let $F=F_{\nu,\veps}$. Noting
	\begin{align*}
		\|T_r(e^{F}h_1) &T_r(e^{-F} h_2)\|_1
				=   \|e^{-F}T_r(e^{F} h_1) e^{F}T_r(e^{-F} h_2)\|_1 \\
			&\le   \sum_{x\in\Z} \sum_{y_1,y_2\in\Z} 	
					e^{|F(x)-F(y_1)|} |\la x|T_r|y_1\ra| |h_1(y_1)|
					e^{|F(x)-F(y_2)|} |\la x|T_r|y_2\ra| |h_2(y_2)|
	\end{align*}
	and, because of the reverse triangle inequality for $F$, we have
	$|F(x)-F(y)| \le F(x-y)\le \nu|x-y|$.
	Thus setting $G_{r,\nu}(y):=\frac{(4|r|e^\nu)^{|y|}}{|y|!} $ the bound \eqref{eq:kernel bound} yields
	\begin{align*}
		\|T_r(e^{F}h_1) T_r(e^{-F} h_2)\|_1
			\le   e^{8|r|}\sum_{y_1,y_2\in\Z} |h_1(y_1)| \big(	 \sum_{x\in\Z}
					G_{r,\nu}(x-y_1)
					G_{r,\nu}(x-y_2) \big) |h_2(y_2)|
	\end{align*}
	so denoting $A_{r,\nu}(y_1,y_2):= \sum_{x\in\Z} 	G_{r,\nu}(x-y_1) G_{r,\nu}(x-y_2) $, we can argue as for the bound \eqref{eq:strong bilinear}, except now we cannot simply use unitarity of $T_r$ to get the bound when the supports of $h_1$ and $h_2$ are not separated. Instead, as in \eqref{eq:A_r bound 0}, we use and Cauchy--Schwartz and the symmetry of $A_{r,\nu}(y_1,y_2)$ in $y_1$ and $y_2$  to see
	\begin{align*}
		\sum_{y_1,y_2\in\Z} |h_1(y_1)| A_{r,\nu}(y_1,y_2) |h_2(y_2)|
		\le
			\Big(\sup_{y_1\in\Z} \sum_{y_2\in\Z}A_{r,\nu}(y_1,y_2)\Big) \|h_1\|_2 \|h_2\|_2
	\end{align*}
	By translation invariance,
	\begin{align*}
		\sum_{y_2\in\Z}A_{r,\nu}(y_1,y_2)
			&= \sum_{x}\sum_{y_2} G_{r,\nu}(x-y_1)G_{r,\nu}(x-y_2)
				= \Big(\sum_{y} G_{r,\nu}(y)\Big)^2 \\
			&= \Big( 1+2\sum_{n\in\N}  \frac{(4|r|e^\nu)^{n}}{n!} \Big)^2
				\le 4 e^{8|r|e^\nu}.
	\end{align*}
	Thus
	\begin{align}\label{eq:twisted exponential bound 0}
		\|T_r(e^{F_{\nu,\veps}}h_1) T_r(e^{-F_{\nu,\veps}} h_2)\|_1
			\le 4 e^{8|r|(1+e^\nu)} \|h_1\|_2 \|h_2\|_2
	\end{align}
	and proceeding similarly as in \eqref{eq:A_r bound 0}--\eqref{eq:A_r bound 4}, one sees
	\begin{align}\label{eq:twisted exponential bound separated}
		\|T_r(e^{F_{\nu,\veps}}h_1) T_r(e^{-F_{\nu,\veps}} h_2)\|_1	
		\le 8 e^{8|r|(1+e^\nu)} \frac{(4|r|e^\nu)^{\lceil\frac{s}{2}\rceil}}{\lceil\frac{s}{2}\rceil !}
	\end{align}
	when $s:=\dist(\supp h_1, \supp h_2)\ge 1$. Together, \eqref{eq:twisted exponential bound 0} and \eqref{eq:twisted exponential bound separated} prove \eqref{eq:twisted strong bilinear}.
	
	 To prove \eqref{eq:exchange}, let $\alpha\ge 1$. Again, it is enough to consider the case $p=1$.
	 Then with $\widetilde{h_2}:= e^{F_{\nu/\alpha,\veps}}h_2$
	 \begin{align*}
	 	\left\|T_r(e^{F_{\nu,\veps}} h_1)|T_r(h_2)|^\alpha\right\|_1
	 	 &= \left\| e^{-F_{\nu,\veps}}T_r(e^{F_{\nu,\veps}} h_1) \big|e^{F_{\nu/\alpha,\veps}}T_r(e^{-F_{\nu/\alpha,\veps}} \widetilde{h_2})\big|^\alpha\right\|_1 \\
	 	 &\le \left\| e^{-F_{\nu,\veps}}T_r(e^{F_{\nu,\veps}} h_1) e^{F_{\nu/\alpha,\veps}}T_r(e^{-F_{\nu/\alpha,\veps}} \widetilde{h_2})\right\|_1 \| e^{F_{\nu/\alpha,\veps}}T_r(e^{-F_{\nu/\alpha,\veps}} \widetilde{h_2})\|_\infty^{\alpha-1}.
	 \end{align*}
	Now arguing similarly as in the proof of \eqref{eq:twisted exponential bound 0},
	\begin{align*}
		\big\| e^{-F_{\nu,\veps}}T_r( &e^{F_{\nu,\veps}} h_1) e^{F_{\nu/\alpha,\veps}}T_r(e^{-F_{\nu/\alpha,\veps}} \widetilde{h_2})\big\|_1 \\
			&\le \sum_{x} \sum_{y_1,y_2}
							e^{\nu|x-y_1|}|\la x|T_r|y_1\ra| |h_1(y_1)|
							e^{(\nu/\alpha)|x-y_2|}|\la x|T_r|y_2\ra| |\widetilde{h_2}(y_2)| \\
			&\le 4e^{4|r|(2+e^\nu+e^{\nu/\alpha}) } \|h_1\|_2 \|\widetilde{h_2}\|_2
				\le 4e^{8|r|(1+e^\nu) } \|h_1\|_2 \|\widetilde{h_2}\|_2
	\end{align*}
	and
	\begin{align*}
		\| e^{F_{\nu/\alpha,\veps}}T_r&(e^{-F_{\nu/\alpha,\veps}} \widetilde{h_2})\|_\infty
		\le \| e^{F_{\nu/\alpha,\veps}}T_r(e^{-F_{\nu/\alpha,\veps}} \widetilde{h_2})\|_2
			= \|e^{2F_{\nu/\alpha,\veps}} |T_r(e^{-F_{\nu/\alpha,\veps}} \widetilde{h_2}) |^2\|_1^{1/2} \\
			&\le \left( \sum_{x} \sum_{y_1,y_2}
							e^{(\nu/\alpha)|x-y_1|}|\la x|T_r|y_1\ra| |\widetilde{h_2}(y_1)|
							e^{(\nu/\alpha)|x-y_2|}|\la x|T_r|y_2\ra| |\widetilde{h_2}(y_2)|\right)^{1/2} \\
			&\le 2e^{4|r|(1+e^{\nu/\alpha}) } \|h_1\|_2 \|\widetilde{h}_2\|_2 				\le 2e^{4|r|(1+e^\nu) } \|\widetilde{h_2}\|_2
	\end{align*}		
	for all $\alpha\ge 1$. This proves \eqref{eq:exchange}.

	Lastly, let  $f_\nu(x)=A e^{-\nu|x|}$ with $\nu>0$ and $A>0$, then
	 \begin{align*}
	 	\|f_\nu\|_\kappa^\kappa
	 		= A^\kappa \sum_{x\in \Z} e^{-\kappa\nu|x|}
	 		= A^\kappa (1+ 2\sum_{x=1}^\infty e^{-\kappa\nu x} )
	 		= A^\kappa (\frac{1+ e^{-\kappa\nu x}}{1-e^{-\kappa\nu x}})
	 		= A^\kappa \frac{\cosh(\frac{\kappa}{2}\nu )}{\sinh(\frac{\kappa}{2}\nu)}.
	 \end{align*}
	 Moreover,
	 \begin{align*}
	 	\|D_+f_\nu\|_2^2
	 	&=
	 		A^2 \sum_{x} |e^{-\nu|x+1|}- e^{-\nu|x|}|^2
	 		= A^2 \left( \sum_{x\ge 0} (e^{-\nu}-1)^2 e^{-2\nu x}
	 		  + \sum_{x\leq -1} (e^{\nu}-1)^2 e^{2\nu x}\right) \\
	 	&=
			A^2 \left( (e^{-\nu}-1)^2 \frac{1}{1-e^{-2\nu}}   + (e^{\nu}-1)^2 \frac{e^{-2\nu}}{1-e^{-2\nu}}\right)
			=
				4A^2\, \frac{\sinh^2(\nu/2)}{\sinh(\nu)}
	 \end{align*}
\end{proof}

\section{Boundedness, negativity, and strict subadditivity of the energy} \label{sec:Strict concavity of the ground state energy}
Recall that for $\dav \ge 0$
\begin{align*}
	H(f):= \f{\dav}{2} \la f, -\Delta f \ra - N(f)
\end{align*}
and
\begin{align*}
	E_\lambda^{\dav}:= \inf \left\{ H(f): \|f\|^2=\lambda\right\}.
\end{align*}

In this section we will give an a-priori bound on the ground-state energy which is an essential ingredient in the construction of strongly convergent minimizing sequences.

\begin{lemma} \label{lem:E-boundedness} Assume that assumption \ref{ass:A1} holds.
Then for every $\lambda \ge  0$
  \begin{align*}
  	-\lambda^{\gamma_1/2}-\lambda^{\gamma_2/2}\lesssim E^{\dav}_\lambda \le 0 ,
  \end{align*}
  where the implicit constant in the lower bound depends only on $\mu(\R)$ and the support of $\mu$.
  In particular, the variational problem is well-posed.
\end{lemma}

\begin{proof}
  The lower bound follows immediately from $H(f)\ge -N(f)$ and Proposition \ref{prop:N-boundedness}. For the upper bound we argue similarly as in the beginning of the proof of Theorem \ref{thm:existence}. Note that
  \begin{align*}
  	E^{\dav}_\lambda \le H(f) = \frac{\dav}{2}\|D_+f\|_2^2 -N(f)
  	\le \frac{\dav}{2}\|D_+f\|_2^2 +|N(f)|
  \end{align*}
  To bound the nonlinearity, we use \eqref{eq:Vbound} to see that with $B$ so that $\supp\mu\subset [-B,B]$,
  \begin{align*}
  	|N(f)| \lesssim \sup_{|r|\le B} (\|T_rf\|_{\gamma_1}^{\gamma_1} + \|T_rf\|_{\gamma_2}^{\gamma_2})
  	  \lesssim  \|f\|_{\gamma_1}^{\gamma_1} + \|f\|_{\gamma_2}^{\gamma_2} .
  \end{align*}
  where we also used the bound \eqref{eq:lpbound} from Lemma \ref{lem:useful}. Now define
  $f_n(x)$ as in the proof of Theorem \ref{thm:existence} by
  \begin{align*}
    f_n(x) \coloneq c_n \id_{[-n,n]}(x)
  \end{align*}
 with $c_n= \left(\tfrac{\lambda}{2n+1}\right)^{1/2}$. Then $\|f_n\|_2^2=\lambda$.
 Note that
\begin{align*}
	\|D_+ f_n\|_2^2 = 2c_n^2 \to 0 \quad\text{as } n\to\infty
\end{align*}
and for any $\gamma>2$
\begin{align*}
	\|f_n\|_{\gamma}^\gamma = \left(\frac{\lambda}{2n+1}\right)^{\gamma/2} (2n+1) \to 0
\end{align*}
as $n\to\infty$. So $E^{\dav}_\lambda=\inf_{\|f\|_2^2=\lambda}H(f)\le \lim_{n\to\infty}H (f_n)=0$.
\end{proof}

Similar to \cite{ChoiHuLee2015}, we get the following strict concavity and strict subadditivity of $E_\lambda^\dav$.

 \begin{proposition}[Strict subadditivity]\label{prop:strict-subadditivity}
Under assumptions \ref{ass:A1} and \ref{ass:A2} and
for any  $\lambda>0$, $0<\delta<\lambda/2$, and
$ \lambda_1,\ \lambda_2\geq \delta$ with  $\lambda_1+\lambda_2\leq \lambda$, we have
 \begin{align}\label{eq:strict subadditivity}
 	E_{\lambda_1}^\dav+E_{\lambda_2}^\dav\geq \left[1-(2^{\f{\gamma_0}{2}}-2)\left(\frac{\delta}{\lambda}\right)^{\f{\gamma_0}{2}}\right]E_{\lambda}^\dav
 \end{align}
where $\gamma_0>2$ as in \ref{ass:A2}.
 \end{proposition}
\begin{remark}
  In particular, Proposition \ref{prop:strict-subadditivity} shows that for any $\lambda_1, \lambda_2>0$ one has
  \begin{align*}
  	E^{\dav}_{\lambda_1} + E^{\dav}_{\lambda_2} > E^{\dav}_{\lambda_1+\lambda_2}
  \end{align*}	
  as soon as $E^{\dav}_{\lambda_1+\lambda_2}<0$. That is, the map $\lambda\mapsto E^{\dav}_{\lambda}$ is \emph{strictly  subadditive} where it is \emph{strictly negative}.
\end{remark}

In order to prove this, we need a little preparation.
\begin{lemma}\label{lem:V-lower-bound}
	$V$ obeys \ref{ass:A2} if and only if for all $ t \ge 1 $ we have
	\begin{align}\label{eq:V-lower-bound}
		V(ta) \ge t^{\gamma_0} V(a) \quad \text{for all } a>0.
	\end{align}
\end{lemma}
\begin{proof}
	Assume that $V$ obeys \ref{ass:A2}. Then
	\begin{align*}
		\frac{d}{dt} V(ta) = V'(ta)a \ge \frac{\gamma_0}{t} V(ta)
	\end{align*}
	for all $a>0$ and $t>1$. Thus
	\begin{align*}
		\frac{d}{dt} (t^{-\gamma_0} V(ta)) \ge 0	
	\end{align*}
	and integrating this yields \eqref{eq:V-lower-bound}.
	
	Conversely, since \eqref{eq:V-lower-bound} is an equality for $t=1$, we can differentiate it at $t=1$ to get \ref{ass:A2}.
\end{proof}

\begin{proof}[Proof of Proposition \ref{prop:strict-subadditivity}]
	Let $t \ge 1$, then
	\ref{ass:A2} and Lemma \ref{lem:V-lower-bound} imply
	$N(t f)\ge t^{\gamma_0}N(f)$ for any $f\in l^2(\Z)$. Thus also
	\begin{align}\label{eq:subadd0}
		H(tf) \le t^2\dav\la D_+f, D_+f \ra - t^{\gamma_0} N(f)
		\le t^{\gamma_0} H(f)
	\end{align}
	since $t\ge 1$ and  $\gamma_0>2$.  Hence
	\begin{align*}
		s^{\gamma_0} H(f) \le H(s f)
	\end{align*}
	for all $f\in l^2(\Z)$ and all $0\le s\le 1$ and
	\begin{align}\label{eq:subadd1}
		s^{\gamma_0} E^{\dav}_\lambda = s^{\gamma_0}\inf_{\|f\|_2^2=\lambda} H(f)
		\le \inf_{\|f\|_2^2=s^2 \lambda} H(f) = E^{\dav}_{s^2\lambda} \; .
	\end{align}
	For $\lambda_1,\lambda_2>0$ and $\lambda_1+\lambda_2\le \lambda$ choose $0<\mu_j<1$ with $\lambda_j=\mu_j\lambda$ for $j=1,2$. Then
	\begin{align*}
		E^{\dav}_{\lambda_1} + E^{\dav}_{\lambda_2}
			= E^{\dav}_{\mu_1\lambda} +E^{\dav}_{\mu_2\lambda}
			\ge   \big(\mu_1^{\gamma_0/2}  +\mu_2^{\gamma_0/2}\big)
				E^{\dav}_{\lambda}
	\end{align*}
	because of \eqref{eq:subadd1}. Without loss of generality we can assume $\mu_1\ge \mu_2$, otherwise we simply exchange $\lambda_1$ and $\lambda_2$. Then, since $0< \mu_1+\mu_2\le 1$, we have
	\begin{align*}
		\mu_1^{\gamma_0/2} + \mu_2^{\gamma_0/2}
		&\le
			1 -\big( (\mu_1+\mu_2)^{\gamma_0/2} - \mu_1^{\gamma_0/2}- \mu_2^{\gamma_0/2} \big) \\
		&= 1- \mu_2^{\gamma_0/2} \left(
					\left(1+\frac{\mu_1}{\mu_2} \right)^{\gamma_0/2} -\left(\frac{\mu_1}{\mu_2}\right)^{\gamma_0/2}
					- 1
			  \right) \\
		&\le
			1- \mu_2^{\gamma_0/2}
				\left(
					2^{\gamma_0/2} -2
				\right)
	\end{align*}
	where the last inequality follows from the fact that for $\gamma_0>2$ the map
	$0<s\mapsto \left(1+s \right)^{\gamma_0/2} -s^{\gamma_0/2} -1$
	is increasing on $[1, \infty)$.
	Thus, since $E^{\dav}_\lambda\le 0$ by Lemma \ref{lem:E-boundedness}
	and $\mu_2\ge \delta/\lambda$, the inequality
	\eqref{eq:strict subadditivity} follows.
\end{proof}

\begin{lemma}\label{lem:negativity}
	Assume that assumption \ref{ass:A4} holds and $\lambda>0$. Then
	$E^{\dav}_\lambda<0$.
\end{lemma}

\begin{proof} Unlike the continuous case, where Gaussians provide a nice class of initial conditions $f$ for which one can explicitly calculate the time evolution $T_rf$, no such class of functions exists in the discrete case.
	Hence, the proof that $E^\dav_\lambda$ is strictly negative is quite different from the continuous case.
	
 Recall that $E^{\dav}_\lambda$ is defined in \eqref{eq:min}.
 We consider the case $\dav=0$ first.
 Let $h_{1,p}(r):=e^{4|r||1-\frac{2}{p}|}$. Assumption \ref{ass:A4} says that
 if  $\dav=0$, then  there exists $\veps>0$ such that
 $V(a)>0$ for all $0<a\le \veps$.  Let $B>0$ such that $\supp\mu\subset[-B,B]$ and for any $\nu>0$ take $f_\nu(x)=A_\nu e^{-\nu|x|}$ with
 \begin{align}\label{eq:A_nu}
 	A_\nu:= \lambda^{1/2} \left(\frac{\sinh(\nu)}{\cosh(\nu)}\right)^{1/2} .
 \end{align}
 Then \eqref{eq:f_nu} from Lemma \ref{lem:useful} shows $\|f_\nu\|_2^2=\lambda$, i.e., $f_\nu$ is a valid test function.
 Moreover, $A_\nu$ is increasing in $\nu$ with $A_\nu\to 0$ as $\nu\to 0+$ so $\|f_\nu\|_\infty=A_\nu\le \veps/h_{1,\infty}(B)$ for all small enough $\nu>0$, hence, because of \eqref{eq:lpbound}, there exists $\nu_1>0$ such that $\|T_rf_\nu\|_\infty\le\veps$ for all $|r|\le B$ and $0<\nu\le\nu_1$. In this case, by assumption \ref{ass:A4},
 \begin{align}\label{eq:positivity}
 	V(|T_rf_\nu(x)|)\ge 0 \quad \text{for all }x\in\Z,\ |r| \le B,\ 0<\nu\le\nu_1,
 \end{align}
 hence $N(f_\nu)\ge 0$.
 If $E^0_\lambda = 0$, we would have $0=E^0_\lambda \le -N(f_\nu)\le 0$, so
 \begin{align*}
   	N(f_\nu)= \int_\R \sum_{x\in\Z} V(|T_rf_\nu(x)|) \,\mu(dr) = 0.
 \end{align*}
 Because of \eqref{eq:positivity}, this implies for all $0<\nu\le\nu_1$,
 \begin{align*}
 	V(|T_rf_\nu(x)|) = 0 \quad \text{for } \mu\text{-almost all }r \text{ and all }x\in\Z.
 \end{align*}
 and since $0\le |T_rf_\nu(x)|\le\veps$, the only way this can be is if
 \begin{align*}
 	T_rf_\nu =0 \quad  \text{for } \mu\text{-almost all }r
 \end{align*}
 and since $T_r$ is unitary on $l^2(\Z)$, this implies $f_\nu=0$ for all small enough $\nu$, which is a contradiction. Thus $E^0_\lambda<0$ if
 $\lambda>0$.
 \smallskip

 In the case $\dav>0$, \ref{ass:A4} shows that there exist $\veps>0$ and $2\le\kappa<6$ such that $V(a)\gtrsim a^\kappa$ for all $0\le a\le \veps$. Again let $B>0$ such that
 $\supp\mu\subset[-B,B]$ and choose $f_\nu(x):= A_\nu e^{-\nu|x|}$ with $A_\nu$ given by \eqref{eq:A_nu} and $\nu_2>0$ such that $\|f_\nu\|_\infty= A_\nu\le \veps/h_{1,\infty}(2B)$ for all $0<\nu\le \nu_2$. Then the second part of Lemma \ref{lem:useful} guarantees $\|T_{r-r_0}f_\nu\|_\infty\le\veps$ for all $r_0,r\in\supp\mu$ and $0<\nu\le\nu_2$.

 Set $g:= T_{-r_0}f_\nu$, then $0\le |T_rg|\le \veps$ for all $r\in\supp\mu$, hence
 \begin{align*}
 	N(g) &= \int_\R \sum_{x\in\Z} V(|T_rg(x)|) \,\mu(dr)
 		\gtrsim  \int_\R \sum_{x\in\Z} |T_rg(x)|^\kappa \,\mu(dr)
 		\ge \int_{r_0-\delta}^{r_0+\delta}  \|T_rg\|_\kappa^\kappa \,\mu(dr)
 \end{align*}
 for all $r_0\in\supp\mu$ and any $\delta>0$. Define $h_2(r):= e^{4|r|}-1$. Then  the bounds from Lemma \ref{lem:useful} give
 \begin{align*}
 	\|T_rg\|_\kappa^\kappa
 		&\ge
 			\|f_\nu\|_\kappa^\kappa -\left| \|f_\nu\|_\kappa^\kappa
 				- \|T_{r-r_0}f_\nu\|_\kappa^\kappa \right| \\
 		&\ge \|f_\nu\|_\kappa^\kappa -\kappa \max(\|f_\nu\|_\kappa^{\kappa-1},\|T_{r-r_0}f_\nu\|_\kappa^{\kappa-1} )
 			\|f_\nu- T_{r-r_0}f_\nu\|_\kappa \\
 		&\ge \left( 1-\kappa (h_{1,\kappa}(r-r_0))^{\kappa-1} h_2(r-r_0) \right) \|f_\nu\|_\kappa^\kappa .
 \end{align*}
 Thus
 \begin{align*}
 	N(g) \gtrsim \left( 1-\kappa (h_{1,\kappa}(2B))^{\kappa-1} h_2(\delta) \right) \mu((r_0-\delta,r_0+\delta)) \|f_\nu\|_\kappa^\kappa
 \end{align*}
Since the (complement of the) support of the measure $\mu$ is given by
 \begin{align*}
 	(\supp\mu)^c = \{r_0\in \R|\, \exists \delta>0\text{ such that } \mu((r_0-\delta,r_0+\delta))=0\}
 \end{align*}
 one has, for any $r_0\in\supp\mu$,
 \begin{align*}
 	\mu((r_0-\delta,r_0+\delta))>0 \quad\text{for all }\delta>0.
 \end{align*}
 So choosing any $r_0\in\supp\mu$ and $\delta>0$ small enough that
 $1-\kappa (h_{1,\kappa}(2B))^{\kappa-1} h_2(\delta)>0$ yields
 \begin{align*}
 	N(g) \gtrsim  \|f_\nu\|_\kappa^\kappa = \lambda^{\kappa/2} \left(\frac{\sinh(\nu)}{\cosh(\nu)}\right)^{\kappa/2}\frac{\cosh(\frac{\kappa}{2}\nu)}{\sinh(\frac{\kappa}{2}\nu)}
 \end{align*}
 where the implicit constant depends only on $\delta>0$ and the constant in the lower bound on $V$ from assumption \ref{ass:A3}, in particular, it \emph{does not} depend on $0<\nu\le \nu_2$.

 Since $\Delta$ and $T_r$ commute, by \eqref{eq:f_nu},
 \begin{align*}
 	\la g, -\Delta g\ra= \la f_\nu, -\Delta f_\nu\ra =
 	4\lambda \, \frac{\sinh^2(\nu/2)}{\cosh(\nu)}
 \end{align*}
 and choosing $g:= T_{-r_0}f_\nu$ as a test function in the energy $H$ shows
 \begin{align*}
 	E_\lambda &\le H(g) = \frac{\dav}{2}\la g,-\Delta g\ra -N(g) \\
 		&\le 2\dav \lambda \frac{\sinh^2(\nu/2)}{\cosh(\nu)}
 			-C \lambda^{\kappa/2} \left(\frac{\sinh(\nu)}{\cosh(\nu)}\right)^{\kappa/2}\frac{\cosh(\frac{\kappa}{2}\nu)}{\sinh(\frac{\kappa}{2}\nu)} \\
 		&= \sinh(\nu/2)^2
 			\left(
 				 \frac{\dav \lambda}{\cosh(\nu)}
 				 -C \lambda^{\kappa/2} \frac{\cosh(\frac{\kappa}{2}\nu)}{\cosh^{\kappa/2}(\nu)}\frac{\sinh^{\kappa/2}(\nu) }{\sinh(\frac{\kappa}{2}\nu)\sinh^2(\nu/2)}
 			\right)
 \end{align*}
	As $\nu\to 0$ $	\sinh(s \nu ) = \Oh(\nu)$ and $\cosh(s\nu)= \Oh(1)$
	for any fixed $s \neq 0$. So
	\begin{align*}
		E_\lambda &\le \Oh(\nu^2)\left(1- \Oh(\nu^{\frac{\kappa}{2}-3})\right) <0
	\end{align*}
	for small enough $\nu>0$, since $\kappa<6$. This shows that $E_\lambda<0$ for all $\lambda>0$.

\end{proof}

\begin{lemma}\label{lem:N-positivity}
	Assume that assumptions \ref{ass:A1} through \ref{ass:A3} hold. Then there exists
	$f\in l^2(\Z)$ such that
	\begin{align*}
		N(f) = \int_\R \sum_{x\in\Z} V(|T_rf(x)|)\, \mu(dr) >0 \, .
	\end{align*}
\end{lemma}
\begin{proof}
  Let $l\in\N$ and set $u_l(r,\cdot)\coloneq T_r\id_{[-2l,2l]} $. Since $\mu$ is a finite measure with compact support there exits $0<B<\infty$ with $\supp\mu\subset [-B,B]$.
  We claim that for some constant $c>0$ and all large enough $l\in\N$ the bounds
  \begin{align}
  	|u_l(r,x)| -1 \gtrsim -e^{-cl}  &\quad\text{for all } |x|\le l, |r|\le B  \label{eq:nearly constant initial condition-1} \\
  	|u_l(r,x)|  \lesssim l e^{-c(|x|-2l)}  &\quad\text{for all } |x|\ge 3l, |r|\le B
  	\label{eq:nearly constant initial condition 2}
  \end{align}
	hold. We will prove them later. Assumptions \ref{ass:A2} and \ref{ass:A3}, together with Lemma \ref{lem:V-lower-bound} show that there exists $a_0>0$ such that $V(a)\gtrsim a^{\gamma_0}$ for all $a\ge a_0$ and using assumption \ref{ass:A1}, we have $V(a)\gtrsim - a^{\gamma_1}$ for
	$0\le a<a_0 $. Thus, with $\gamma\coloneq \min(\gamma_0,\gamma_1)$, we see that the lower bound	
	\begin{align}\label{eq:V pointwise-lower-bound}
		V(a) \gtrsim  -a^{\gamma}\id_{[0,a_0)}(a) + a^\gamma \id_{[a_0,\infty)}(a)
	\end{align}
	holds and $V$ is bounded from below.
	
	By \eqref{eq:nearly constant initial condition-1} we can choose $l$ and $\alpha$
	large enough such that
	$\alpha|u_l(r,x)|\ge \tfrac{\alpha}{2}\ge a_0$ for all $|x|\le l, |r|\le B$. Then \eqref{eq:V pointwise-lower-bound} yields
	\begin{align*}
		I \coloneq \sum_{|x|\le l} V(\alpha|u_l(r,x)|) \gtrsim  l \alpha^\gamma .
	\end{align*}
	Since $V$ is bounded from below, we also have
	\begin{align*}
		II\coloneq \sum_{l<|x|\le 3l} V(\alpha|u_l(r,x)|) \gtrsim - l ,
	\end{align*}
	and \eqref{eq:V pointwise-lower-bound} together with
	\eqref{eq:nearly constant initial condition 2} gives
	\begin{align*}
		III\coloneq \sum_{|x|> 3l} V(\alpha|u_l(r,x)|)
		\gtrsim -  (\alpha l)^\gamma \sum_{|x|>3l} e^{-c\gamma(|x|-2l)}
		\gtrsim - (\alpha l)^\gamma e^{-c\gamma(l+1)}
	\end{align*}
	for all $|r|\le B$. Thus, since $\mu$ is a finite measure with support in
	$[-B,B]$, this gives the lower bound
	\begin{align*}
		N(\alpha\id_{[-2l,2l]}) \gtrsim I + II + III
		\gtrsim l \alpha^\gamma - l -  (\alpha l)^\gamma e^{-c\gamma (l+1)}
	\end{align*}
	for all large enough $\alpha$ and $l$. Setting $\alpha=l$ shows
	$\lim_{l\to\infty}N(l \id_{[-2l,2l]})=\infty$, in particular, $N(f)>0$ for some $f\in l^2(\Z)$.
	
	It remains to prove \eqref{eq:nearly constant initial condition-1} and \eqref{eq:nearly constant initial condition 2}. From  Lemma \ref{lem:useful}, more precisely, \eqref{eq:kernel bound}, we have the bound
	$|\la x|T_r|y\ra|\le  \min\big(1,e^{4|r|}\tfrac{(4|r|)^{|x-y|}}{|x-y|!} \big)$
	for any $r\in\R $. Thus, for all $|r|\le B $
	\begin{align*}
		|u_l(r,x)| \le  \sum_{|y|\le 2l} |\la x|T_r|y\ra|
			\lesssim  \sum_{|y|\le 2l} \frac{(4B)^{|x-y|}}{(|x-y|)!},
	\end{align*}
	The map $n\mapsto \frac{(4B)^{n}}{n!}$ is decreasing for all $n\ge 4B-1$.
	 For $|x|\ge 3l$ and all $|y|\le 2l$ we will have $n=|x-y| \ge  |x|-2l \ge 4B-1$
	 for all large enough $l$, hence we can replace $|x-y|$ above by $|x|-2l$ and use
	 $n!\ge e^{n\ln n -n}$ to arrive at
	 to see
	\begin{align*}
		|u_l(r,x)|  \lesssim  l \frac{(4B)^{|x|-2l|}}{(|x|-2l)!}
		\le l e^{(1+\ln(4B)-\ln(|x|-2l))(|x|-2l)}
		\lesssim l e^{-c(|x|-2l)}
	\end{align*}
	for some constant $c>0$ and all  $|x|\ge 3l$ with $l$ large enough.  This proves \eqref{eq:nearly constant initial condition 2}.
	
	 For any initial condition $f_0$, the time evolution $u(r,\cdot)= T_rf_0$ is given by the convergent series
	 \begin{align*}
	 u(r,x) &= \la \delta_x, T_rf_0\ra
	 	= f_0 + \sum_{n=1}^\infty \frac{(ir)^n}{n!} \la \delta_x, \Delta^n f_0\ra
	 \end{align*}
	 If $f_0 = \id_{[-2l,2l]}$, then
	 $\Delta f_0 = -\delta_{2l+1} + \delta_{2l} -\delta_{-(2l+1)} +\delta_{-2l}$, where
	 $\delta_y$ is the Kronecker delta at $y\in\Z$. Moreover, since $\Delta $ increases
	 the  support by at most one, that is, $\min(\supp \Delta f)\ge \min(\supp f)-1$ and
	 $\max(\supp \Delta f)\leq \max(\supp f)+1$, and $\|\Delta f\|_\infty \le 4\|f\|_\infty$, we see that for any $1\le n\le 2l$ there exists $g_{n,l}:\Z\to \R $ with $\|g_{n,l}\|_\infty\le 1$, $\supp g_{n,l}\subset [2l-n+1, 2l+n]$ and
	 \begin{align*}
	 	\la \delta_x,\Delta^n \id_{[-2l,2l]}\ra = \big(\Delta^n \id_{[-2l,2l]}\big)(x) = 4^n\big( g_{n,l}(x) + g_{n,l}(-x) \big)
	 \end{align*}
	 for all $x\in\Z$.
	 In particular, $\la \delta_x,\Delta^n \id_{[-2l,2l]}\ra = 0$ for all $|x|\le l$ and all $1\le n\le l$. So the series for $u_l(r,x)$ gives
	 \begin{align*}
	 	u_l(r,x) = \id_{[-2l,2l]}(x) +  \sum_{n\ge l+1}  \frac{(ir)^n}{n!} \la \delta_x,\Delta^n \id_{[-2l,2l]}\ra
	 	\quad \text{for all } |x|\le l \, ,
	 \end{align*}
	 hence by the same calculation as for \eqref{eq:Tr bound}
	 \begin{align*}
	 	|u_l(r,x)| \ge  1-  \sum_{n\ge l+1}  \frac{(4|r|)^n}{n!}
	 	\ge 1- \frac{e^{4|r|}|4r|^{l+1}}{(l+1)!}
	 	\quad \text{for all } |x|\le l\, .
	 \end{align*}
	 Bounding $(l+1)! \ge e^{(l+1)\ln(l+1)-(l+1)}$ shows that \eqref{eq:nearly constant initial condition-1} is true for some $c>0$ and all large enough $l\in\N$.
\end{proof}

\section{The discrete IMS localization formula}\label{sec:IMS}
Here we give a simple bound which is useful for localizing the discrete kinetic energy.
\begin{lemma}\label{lem:IMS}
	Let $f\in l^2(\Z)$ and  $\xi:\Z\to\R$ be a bounded function. Then,
	\begin{align}\label{eq:IMS}
		\re(\la\xi^2 f, -\Delta f\ra)
		 =
			\la \xi f, -\Delta (\xi f)\ra
				- \sum_{x\in\Z} |D_+\xi(x)|^2\re(\ol{f(x)}f(x+1)) .
	\end{align}
 In particular, the lower bound
	\begin{align}\label{eq:IMS-lower-1}
		\re(\la\xi^2 f, -\Delta f\ra)
		\ge
			\la \xi f, -\Delta (\xi f)\ra
				- \frac{1}{2}\la f,(|D_+\xi|^2+|D_-\xi|^2)f\ra
	\end{align}
	holds and if $\xi_j:\Z\to\R$, $j=1,\ldots,n$ are finitely many bounded functions with $\sum_{j=1}^n \xi_j^2=1$, then
	\begin{align}\label{eq:IMS-lower-2}
		\la f, -\Delta f\ra
			\ge
				\sum_{j=1}^n\la \xi_j f, -\Delta (\xi_j f)\ra
				- \frac{1}{2}\sum_{j=1}^n\la f,(|D_+\xi_j|^2+|D_-\xi_j|^2)f\ra  .
	\end{align}
	Here $D_+\xi(x):= \xi(x+1)-\xi(x)$ and $D_-\xi(x):= \xi(x)-\xi(x-1)$ the forward and backward differences.
\end{lemma}
\begin{proof}
	A simple commutator calculation shows
	\begin{align*}
		[\Delta, \xi] = (D_+\xi) S_+ - (D_-\xi) S_-
	\end{align*}
	where $(S_+f)(x):= f(x+1)$ is the left shift and $(S_-f)(x):= f(x-1)$ is the right                                  shift. Another calculation shows
		\begin{align*}
		\left[[\Delta, \xi],\xi\right] = |D_+\xi|^2 S_+ + |D_-\xi|^2 S_-
	\end{align*}
	and expanding the commutator gives
	\begin{align*}
		\xi^2\Delta -2\xi \Delta\xi +\Delta \xi^2 = \left[[\Delta, \xi],\xi\right].
	\end{align*}
	Thus
	\begin{align*}
		2\re(\la\xi^2 f, &-\Delta f\ra)
		=
			\la f, -(\xi^2\Delta + \Delta \xi^2) f\ra
			=
			2\la\xi f, -\Delta (\xi f)\ra - \la f, \left[[\Delta, \xi],\xi\right] f\ra \\
		&=  2\la\xi f, -\Delta (\xi f)\ra
			-\sum_{x\in\Z} \left( |D_+\xi(x)|^2 \ol{f(x)}f(x+1)+ |D_-\xi(x)|^2f(x-1)\ol{f(x)}\right) \\
		&=
			2\la\xi f, -\Delta (\xi f)\ra
			-\sum_{x\in\Z} |D_+\xi(x)|^2 2\re(\ol{f(x)}f(x+1)),
	\end{align*}
	since $D_-\xi(x+1)= D_+\xi(x)$,
	which proves \eqref{eq:IMS}. The bound \eqref{eq:IMS-lower-1} follows from \eqref{eq:IMS} since
	\begin{align*}
		\sum_{x\in\Z} |D_+\xi(x)|^2 \re(\ol{f(x)}f(x+1))
			&\le
				\frac{1}{2}\sum_{x\in\Z} |D_+\xi(x)|^2 (|f(x)|^2 + |f(x+1)|^2)\\
			&=
				\frac{1}{2}\sum_{x\in\Z} (|D_+\xi(x)|^2 +|D_-\xi(x)|^2)|f(x)|^2
	\end{align*}
  Moreover, if $\sum_j\xi_j^2=1$, then
	\begin{align*}
		\la f, -\Delta f\ra = \re(\la f, -\Delta f\ra)
			= \sum_{j=1}^n \re(\la \xi_j^2 f, -\Delta f\ra)
	\end{align*}
	so \eqref{eq:IMS-lower-2} follows from \eqref{eq:IMS-lower-1}.
\end{proof}

\section{The connection with nonlinear optics}\label{sec:connection-nonlinear-optics}

Our main motivation for studying \eqref{eq:GT} and the related minimization problems \eqref{eq:min} comes from the fact that the solutions are related to
breather-type solutions of the diffraction managed discrete nonlinear Schr\"odinger equation
 \beq\label{eq:discreteNLS}
      i \partial_t u = -d(t) \Delta u - P(u),
 \eeq
where $\Delta$ is the nearest neighbour discrete Laplacian, $t$ the distance along the waveguide, $x\in\Z$ the location of the waveguide, $d(t)$ the local diffraction along the waveguide, and $P(u)$ is an on site nonlinear interaction. This equation describes, for example, an array of coupled nonlinear waveguides \cite{AART,AALR, ESMBA98, MPAES, WY}, but it also models a wide range of effects ranging from molecular crystals
\cite{BS75, TKS88} to biophysical systems \cite{Davydov73, ELS85}.
By symmetry, one assumes that $P$ is odd and $P(0)=0$ can always be enforced by adding a constant term. Most often one makes a Taylor series expansion, keeping just the lowest order nontrivial term leads to $P(u)\simeq |u|^2u$, the Kerr nonlinearity, but we will not make this approximation.
The study of bound states of the discrete nonlinear Schr\"odinger equation \eqref{eq:discreteNLS} has attraction a lot of attention, see, for example,
\cite{JW} and the references therein.

The idea to periodically alter the diffraction along the waveguide
by creating a zigzag geometry of the waveguides, similar to what has been done in dispersion management cables, see, for example,
\cite{GT96b, Tetal03, zharnitsky2001} and the references therein,
was probably first conceived in \cite{ESMA00} in order to create low power stable discrete pulses.
In this case, the total diffraction
$d(t)$ along the waveguide is given by
 \beq\label{eq:local-diffraction}
    d(t) = \veps^{-1} d_0(t/\veps) + \dav  .
 \eeq
Here $\dav$ is the average component of the diffraction and $d_0$ its periodic
mean zero part with period $L$.

A technical complication is the fact that \eqref{eq:discreteNLS} is a non-autonomous equation. We seek to rewrite \eqref{eq:discreteNLS} into a more convenient form in order to find breather type solutions.
In the region of \emph{strong diffraction management} $\veps$ is a small positive
parameter.
In this parameter region an average equation which describes the evolution of the slow
part of solutions of \eqref{eq:discreteNLS} was derived in Fourier
space in \cite{AM01,AM02,AM03}, using the same general method as in the
continuum case, see, e.g., \cite{zharnitsky2001}. The numerical studies of \cite{AM01,AM02,AM03} showed that this average equation possesses
stable solutions which evolve nearly periodically when used as initial data in the
diffraction managed non-linear discrete Schr\"odinger equation.
To derive this equation in our notation, let $T_r\coloneq e^{-ir\Delta}$ be the free discrete Schr\"odinger evolution, set
\begin{align*}
	D(s)\coloneq \int_0^s d_0(\zeta)\,d\zeta,
\end{align*}
 and make the ansatz
\begin{align*}
	u(t,x) = T_{D(\frac{t}{\veps})} v
\end{align*}
for some function $v$. Then, since $\partial_t T_{D(\frac{t}{\veps})}= \frac{1}{\veps}d_0(\frac{t}{\veps})\Delta T_{D(\frac{t}{\veps})}$, we get from \eqref{eq:discreteNLS} and \eqref{eq:local-diffraction} that $v$ solves
\begin{align}\label{eq:discreteNLS2}
	i \partial_t v(t,x) = -\dav \Delta v(t,x) - T_{D(\frac{t}{\veps})}^{-1}\left[ P(T_{D(\frac{t}{\veps})}v(t,\cdot)) \right](x)
\end{align}
for $t\ge 0$ and $x\in\Z$, which is equivalent to \eqref{eq:discreteNLS} and still a non-autonomous equation. Since $d_0$ has average zero and period $L$, $D$ is periodic with the same period $ L$ and thus for small $\veps>0$ the function $t\mapsto D(\frac{t}{\veps})$ is highly oscillatory with period $\veps L$. Similar to Kapitza's treatment of the stabilization of the unstable pendulum by high frequency oscillations of the pivot, see  \cite{LandauLifshitz}, the evolution of $v$ should evolve on two different scales, a slow one plus a high frequency one with a small amplitude. The evolution of the slow part $v_{\mathrm{slow} }$ is described by an averaged equation, where one averages over the fast oscillating terms,
\begin{align*}
	i \partial_t v_{\mathrm{slow}}(t,x)
		&=
			-\dav \Delta v_{\mathrm{slow}}(t,x)
			- \frac{1}{\veps L}\int_0^{\veps L}
					T_{D(\frac{s}{\veps})}^{-1}\left[
						P(T_{D(\frac{s}{\veps})}v_{\mathrm{slow}}(t,\cdot))
					\right](x)\, ds \\
		&=
			-\dav \Delta v_{\mathrm{slow}}(t,x)
				- \frac{1}{L}\int_0^{L}
						T_{D(s)}^{-1}\left[
							P(T_{D(s)}v_{\mathrm{slow}}(t,\cdot))
						\right](x)\, ds .
\end{align*}
Making the substitution $r=D(s)$ and introducing the probability measure $\mu$ on $\R $ defined by
\begin{align*}
	\int_{\R } F(r)\, \mu(dr) \coloneq \frac{1}{L}\int_0^L F(D(s))\, ds
\end{align*}
for any nonnegative (Borel) measurable functions $F$, one has
\begin{align}\label{eq:GT-time}
		i \partial_t v_{\mathrm{slow}}(t,x)
		&=
			-\dav \Delta v_{\mathrm{slow}}(t,x)
				- \int_\R
						T_{r}^{-1}\left[
							P(T_{r}v_{\mathrm{slow}}(t,\cdot))
						\right](x)\, \mu(dr)
\end{align}
which is the time dependent version of \eqref{eq:GT}. To derive \eqref{eq:GT} from it, one  simply makes the ansatz $v_{\mathrm{slow}}(t,x)= e^{i\omega t} \varphi(x)$, to see that this solves \eqref{eq:GT-time} if and only if $\varphi$ solves \eqref{eq:GT}.

Physically it makes sense to assume that the diffraction profile $d_0$ is bounded, or
even piecewise constant along the waveguide, but one might envision much more
complicated scenarios.  The simplest case of dispersion management, $L=2$, $d_0=1$ on $[0,1)$ and $d_0=-1$ on $[1,2)$, i.e., $d_0= \id_{[0,1)}-\id_{[1,2)}$, which is the case most studied in the literature, correspond to a very simple zigzag geometry of the waveguides, \cite{AM01, AM02, AM03}. In this case,  the measure $\mu$ is very simple, having density $\id_{[0,1]}$, the uniform distribution on $[0,1]$, with respect to Lebesgue measure.
This assumption was made in \cite{Moeser05,Panayotaros05,Stanislavova07}, where equation \eqref{eq:GT} was studied for the Kerr type  nonlinearity  $P(a)= |a|^2a$ and also some pure power type modifications thereof in \cite{Moeser05}.

For our results, which also hold for a much larger class of nonlinearities $P$, we need only to assume the much weaker condition that the probability measure
$\mu$ has \emph{bounded support}, i.e., there exists $B>0$ such that
 \beq\label{eq:support}
    \mu([-B,B]^c) = \mu((-\infty,-B)) + \mu((B,\infty)) = 0.
 \eeq
The support condition  \eqref{eq:support} is guaranteed if  $d_0$ is
locally integrable, in which case one take
 \beq\label{eq:D-bounded}
    B:= \sup_{r\in[0,L]}|D(r)|
    \le \int_0^L |d_0(\xi)|\, d\xi<\infty \, .
 \eeq
Clearly, this is a very weak assumption on the diffraction profile $d_0$ and it has to be assumed in order to even make sense out of equation \eqref{eq:discreteNLS}. Thus our results cover the most general physically allowed local diffraction profiles $d_0$, the singular case $d_0=0$ leading to the usual discrete NLS which is even local, and cover a large class of nonlinearities $P$.

\bigskip

\noindent
\textbf{Acknowledgements: } Mi-Ran Choi and Young-Ran~Lee thank the Department of Mathematics at KIT and Dirk Hundertmark thanks the Department of Mathematics at Sogang University for their warm hospitality. Dirk Hundertmark gratefully acknowledges financial support by the Deutsche Forschungsgemeinschaft (DFG) through CRC 1173. He also thanks the Alfried Krupp von Bohlen und Halbach Foundation for financial support. Young-Ran Lee thanks the National Research Foundation of Korea(NRF) for financial support funded by the Korea government(MOE) under grant No.~2014R1A1A2058848.
We would also like to thank an eclectic array of coffee shops in Seoul and Karlsruhe for providing us with much needed undisturbed time and lots of coffee.

\renewcommand{\thesection}{\arabic{chapter}.\arabic{section}}
\renewcommand{\theequation}{\arabic{chapter}.\arabic{section}.\arabic{equation}}
\renewcommand{\thetheorem}{\arabic{chapter}.\arabic{section}.\arabic{theorem}}

\def\cprime{$'$}

\end{document}